%% file: main-arxiv.tex
\documentclass[a4paper,UKenglish,cleveref, autoref, thm-restate]{lipics-v2021}
\usepackage{graphicx} 
\usepackage{boxedminipage}
\usepackage[dvipsnames]{xcolor}
\definecolor{nicered}{RGB}{204,0,0}
\definecolor{lightblue}{RGB}{153,204,255}
\usepackage{tikz}
\usetikzlibrary{calc}
\usetikzlibrary{fpu}
\usetikzlibrary{math}
\usetikzlibrary{decorations.pathreplacing,angles,quotes, calligraphy}
\usetikzlibrary{decorations.pathmorphing}
\tikzstyle{vertex}=[thin,circle,inner sep=0.cm, minimum size=1.7mm, fill=black, draw=black]
\tikzstyle{rvertex}=[thin,circle,inner sep=0.cm, minimum size=1.7mm, fill=nicered, draw=nicered]
\tikzstyle{edge}=[thick, draw = gray]
\tikzstyle{redge}=[ultra thick, draw = nicered]
 \tikzstyle{br} = [decorate, ultra thick, decoration = {calligraphic brace}]

\newcommand{\NP}{{\sf NP}}
\newcommand{\cP}{{\sf P}}

\newcommand{\stf}{{\sc Steiner Forest}}

\newcommand{\fl}[1]{\textcolor{black}{#1}}
\usepackage{regexpatch}
\makeatletter
\xpatchcmd\thmt@restatable{%
\csname #2\@xa\endcsname\ifx\@nx#1\@nx\else[{#1}]\fi
}{%
\ifthmt@thisistheone
\csname #2\@xa\endcsname\ifx\@nx#1\@nx\else[{#1}]\fi
\else
\csname #2\@xa\endcsname[{Restated}]
\fi}{}{}
\makeatother

\title{Steiner Forest for $H$-Subgraph-Free Graphs}
\titlerunning{Steiner Forest for $H$-Subgraph-Free Graphs}

\author{Tala Eagling-Vose}{Department of Computer Science, Durham University, Durham, UK}{tala.j.eagling-vose@durham.ac.uk}{https://orcid.org/0009-0008-0346-7032}{}
\author{David C. Kutner}{Department of Computer Science, Durham University, Durham, UK}{david.c.kutner@durham.ac.uk}{https://orcid.org/0000-0003-2979-4513}{Partly supported by EPSRC grant EP/T004878/1, \emph{Multilayer Algorithmics to Leverage Graph Structure (MultilayerALGS)}.}
\author{Felicia Lucke}{ENS Lyon, Lyon, France}{felicia.lucke@ens-lyon.fr}{https://orcid.org/0000-0002-9860-2928}{supported by SNSF Postdoc Mobility Grant 230578.}
\author{Barnaby Martin}{Department of Computer Science, Durham University, Durham, UK}{barnaby.d.martin@durham.ac.uk}{https://orcid.org/0000-0002-4642-8614}{supported by Leverhulme Trust Research Project Grant RPG-2024-182.}
\author{D\'aniel Marx}{CISPA Helmholtz Center for Information Security, Saarbr\"ucken, Germany}{marx@cispa.de}{https://orcid.org/0000-0002-5686-8314}{}
\author{Dani\"el Paulusma}{Department of Computer Science, Durham University, Durham, UK}{daniel.paulusma@durham.ac.uk}{https://orcid.org/0000-0001-5945-9287}{supported by Leverhulme Trust Research Project Grant RPG-2024-182.}
\author{Erik Jan van Leeuwen}{Department of Information and Computing Sciences, Utrecht University, Utrecht, The Netherlands}{e.j.vanleeuwen@uu.nl}{https://orcid.org/0000-0001-5240-7257}{}

\authorrunning{T. Eagling-Vose, D. Kutner, F. Lucke, D. Marx, B. Martin, D. Paulusma, E.J. van Leeuwen}

\Copyright{Tala Eagling-Vose, David Kutner, Felicia Lucke, D\'aniel Marx, Barnaby Martin, Dani\"el Paulusma, Erik Jan van Leeuwen}



\category{}
\relatedversion{}
\acknowledgements{}

\nolinenumbers
\hideLIPIcs  

\EventEditors{John Q. Open and Joan R. Access}
\EventNoEds{2}
\EventLongTitle{42nd Conference on Very Important Topics (CVIT 2016)}
\EventShortTitle{CVIT 2016}
\EventAcronym{CVIT}
\EventYear{2016}
\EventDate{December 24--27, 2016}
\EventLocation{Little Whinging, United Kingdom}
\EventLogo{}
\SeriesVolume{42}
\ArticleNo{23}

\begin{document}
\maketitle

\begin{abstract} 
Our main result is a full classification, for every connected graph~$H$, of the computational complexity of \stf{} on $H$-subgraph-free graphs. To obtain this dichotomy, we establish the following new algorithmic, hardness, and combinatorial results:
\begin{itemize}
\item \textbf{Algorithms:} We identify two new classes of graph-theoretical structures that make it possible to solve \stf{} in polynomial time. Roughly speaking, our algorithms handle the following cases:
(1) a set $X$ of vertices of \emph{bounded} size that are pairwise connected by subgraphs of treewidth~$2$ or bounded size, possibly together with an independent set of arbitrary size that is connected to $X$ in an arbitrary way (``(entangled) bushes of $\ell$-lemons'');
(2) a set $X$ of vertices of \emph{arbitrary} size that are pairwise connected in a cyclic manner by subgraphs of treewidth~$2$ or bounded size (``cycle of $\ell$-lemons'').

\item \textbf{Hardness results:} We show that \stf{} remains \NP-complete for graphs with 2-deletion set number~3. (The \emph{$c$-deletion set number} is the size of a smallest cutset~$S$ such that every component of $G-S$ has at most~$c$ vertices.)

\item \textbf{Combinatorial results:} To establish the dichotomy, we perform a delicate graph-theoretic analysis showing that if $H$ is a path or a subdivided claw, then excluding $H$ as a subgraph either yields one of the two algorithmically favourable structures described above, or yields a graph class for which \NP-completeness of \stf{} follows from either our new hardness result or a previously known one.
\end{itemize}
Along the way to classifying the hardness for excluded subgraphs, we establish a dichotomy for graphs with $c$-deletion set number at most~$k$. Specifically, our results together with pre-existing ones show that \stf{} is polynomial-time solvable if
(1) $c=1$ and $k\geq 0$, or
(2) $c=2$ and $k\leq 2$, or
(3) $c\geq 3$ and $k=1$,
and is \NP-complete otherwise.
\end{abstract}

\section{Introduction}

We consider the classical \stf{} problem, which belongs to a broad class of graph problems where the input graph is augmented with additional data that specifies the desired characteristics of the output. In this problem, we are given a graph $G$ that contains a pre-defined {\it terminal set} $T=\{\{s_1,t_1\},\ldots,\{s_p,t_p\}\}$ of {\it terminal pairs} for some integer~$p\geq 1$. The goal is to ensure that for every pair $(s_i,t_i)$, there exists a (communication) path between $s_i$ and $t_i$. In other words, we seek a subgraph~$F$ of $G$ such that for each $i\in \{1,\ldots,p\}$, the vertices $s_i$ and $t_i$ lie in the same connected component of $F$. At  the same time, the number of edges of $F$, which represent the total communication cost, must not exceed a given budget~$k$. Consequently, we may assume that $F$ contains no cycles. This makes the problem a relaxation of the well-known {\sc Steiner Tree} problem, where all terminals must be connected within a single tree. Thus, $F$ is called a {\it Steiner forest} for $(G,T)$, and the  \stf{} problem is to decide if a given pair $(G,T)$ admits a Steiner forest with at most $k$ edges for some given integer~$k\geq 0$.

Both {\sc Steiner Tree} and \stf{} have been widely studied due to their applications in network design (see, e.g., the surveys~\cite{GK11,HR92,Lu21} and recent papers on approximability~\cite{AGHJM25b,AGHJM25,GT25} and parameterized complexity~\cite{FL25,GHKKO22}).
As illustrated by these papers, both problems are computationally hard even on highly restricted instances. This is especially true for the more general \stf{} problem, which inherits the hardness of {\sc Steiner Tree} but becomes computationally hard even more quickly. For instance, {\sc Steiner Tree} is solvable in polynomial time for graph classes of bounded treewidth~\cite{ALS91}. In contrast, \stf{}, while polynomial-time solvable for graphs of treewidth at most~$2$~\cite{BHM11}, is \NP-complete for graphs of treewidth~$3$~\cite{BHM11,Ga10}.
This leads to a natural research question:

\medskip
\noindent
{\it Can we still overcome the hardness of \stf{} via input restrictions?}

\medskip
\noindent
Bodlaender et al.~\cite{BJMOPPSV25} approached this question in a systematic way.
They first noted that the hardness gadget from~\cite{BHM11} for graphs of treewidth~$3$ has $3$-deletion set number~$2$ and proved that \stf{} becomes solvable in polynomial time for graphs with $2$-deletion set number~$2$. For $c\geq 1$, a  {\it $c$-deletion set} of $G$ is a set $S\subseteq V$ such that each connected component in $G-S$ has at most $c$ vertices. Specifically, a $1$-deletion set corresponds to a vertex cover. The {\it $c$-deletion set number} of~$G$ is the size of a smallest $c$-deletion set in $G$. 

Afterwards, Bodlaender et al.~\cite{BJMOPPSV25} studied {\it monotone} graph classes, that is, classes closed under vertex and edge deletion. Such classes include many well-known classes, such as planar and bipartite graphs, and are characterized by a set ${\cal H}$ of forbidden subgraphs. To better understand the algorithmic behaviour of such classes, it is natural to first consider the case where $|{\cal H}|=1$.
Bodlaender et al.~\cite{BJMOPPSV25} combined their polynomial-time result for graphs of $2$-deletion set number~$2$ with an {\sf FPT} algorithm for \stf{} parameterized by $1$-deletion set number~\cite{GHKKO22,BJMOPPSV25,FL25} to obtain polynomial-time algorithms for \stf{} on $H$-subgraph-free graphs for several graphs~$H$. Here, a graph~$G$ is {\it $H$-subgraph-free} if $H$ cannot be obtained from $G$ by vertex and edge deletions. However, the complexity status of {\sc Steiner Forest} on $H$-subgraph-free graphs is still open for many graphs $H$.
  
\medskip
\noindent
{\bf Our Results.} Building on the results of~\cite{BHM11,BJMOPPSV25}, we introduce new methodology to completely classify the complexity of \stf{} for $H$-subgraph-free graphs, for every connected graph~$H$. To explain our classification, let $P_r$ denote the path on $r$ vertices, and let $K_{1,3}$ denote the claw (4-vertex star). For $1\leq h\leq i\leq j$, let $S_{h,i,j}$ denote the {\it subdivided claw}, 
in which each of the three edges of the claw is subdivided $h-1$, $i-1$ and $j-1$ times, respectively (so that $S_{1,1,1}=K_{1,3}$). Moreover, let $G_1+G_2=(V(G_1)\cup V(G_2), E(G_1)\cup E(G_2))$ denote the disjoint union of two vertex-disjoint graphs $G_1$ and $G_2$.
We now present our first result:

\begin{restatable}{theorem}{MakeDichoThm}\label{t-dicho}
For a connected graph~$H$, \stf{} on $H$-subgraph-free graphs is polynomial-time solvable if $H\subseteq P_{11}$, $S_{1,3,6}$, $S_{2,2,7}$, $S_{2,3,5}$, $S_{2,4,4}$, or $S_{3,3,4}$, and \NP-complete otherwise. 
\end{restatable}

\noindent
We obtain the new polynomial-time cases in Theorem~\ref{t-dicho} by giving
general algorithms that
{\it leverage the interplay between deletion sets and treewidth}. The algorithms generalize the known polynomial-time algorithms for graphs of $2$-deletion number~$2$~\cite{BJMOPPSV25}, graphs of treewidth~$2$~\cite{BHM11}, and graphs of bounded vertex cover number~\cite{GHKKO22,BJMOPPSV25,FL25}, both separately and \emph{simultaneously}.
We introduce novel types of structural decompositions, where a set of special vertices is pairwise connected by components that are either of bounded size or have treewidth at most~2. Section~\ref{s-techniques} defines these decompositions in terms of ``citruses'', ``lemons'', ``bushes'', etc., and proves the algorithmic results.

The main graph-theoretical part of our paper is showing that excluding any graph $H$ in the polynomial-time cases of Theorem~\ref{t-dicho} makes one of the algorithmic results applicable (Section~\ref{s-poly}).
Earlier work showed that the problem is polynomial-time solvable on $P_{9}$-subgraph-free graphs, thus we may assume that the graph contains a path on 9 vertices. We show that such a path in a 2-connected graph can be extended to a path on 11 vertices, unless a certain decomposition exist. Thus we may assume that the graph has a path on 11 vertices. With significantly more effort, we show that this path can be extended to one of the subdivided claws in Theorem~\ref{t-dicho}, unless again some decompositions exist.
 
For the negative side of Theorem~\ref{t-dicho}, we show that \stf{} is
\NP-complete even for graphs with $2$-deletion set number~$3$ and treewidth~$3$ (Section~\ref{s-hard}). The proof is by a reduction from a Constraint Satisfaction Problem (CSP) with domain size 3, where the value of the variables correspond to which of the three vertices of the $2$-deletion set a terminal is connected to in a solution. The key idea is that a simple gadget can easily express constraints of the form $(x=a)\to (y=a)$, which is sufficient to simulate an NP-hard CSP. This hardness result also yields a new dichotomy in the following sense:

\begin{restatable}{theorem}{MakeDeletionThm}\label{t-deletion}
For constants $c,k \geq 1$, \stf{} on graphs with $c$-deletion set number~$k$
is polynomial-time solvable if $c=1$, $k\geq 0$, or $c=2$, $k\leq 2$, or $c\geq 3$, $k=1$, and \NP-complete otherwise.
\end{restatable}

\noindent
As explained in Section~\ref{s-hard}, Theorems~\ref{t-dicho} and~\ref{t-deletion} also rely on the results from \cite{BHM11,BJMOPPSV25,GHKKO22} given~below.

\medskip
\noindent
{\bf Related Work.}
Our paper belongs to a recent line of systematic research that considers classical graph problems on monotone classes, for the case when the set  of forbidden subgraphs~${\cal H}$ is finite. Whereas earlier examples~\cite{AK92,GP14,Ka12} were more scattered through the literature, Johnson et al.~\cite{JMOPPSV25} presented a unified approach.
They showed that, subject to $\cP\neq \NP$, every problem that is C123 (meaning it satisfies three conditions called C1, C2, and C3), is polynomial-time solvable on ${\cal H}$-subgraph-free graphs if and only if ${\cal H}$ contains a 
disjoint union of paths and subdivided claws. 

The family of C123-problems includes
 many partitioning, covering, and packing problems as well as width parameters and network design problems, in particular {\sc Steiner Tree}. The C1 condition is to be efficiently solvable on graphs of bounded treewidth.  As mentioned earlier, \stf{} fails C1, which justified the study in~\cite{BJMOPPSV25}. We refer to~\cite{EJLMP25,JMPPSV23,Lo26,LMPPSSV24} for partial classifications of problems on $H$-subgraph-free graphs that do satisfy C1 but not C2 or C3. Another well-known example of a problem that fails C1 is  {\sc Subgraph Isomorphism}, which is \NP-complete even for input pairs of path-width~$1$. Its complexity classification on $H$-subgraph-free graphs was settled by Bodlaender et al.~\cite{BHKKOO20} apart from 
 two open cases. 

We now discuss some specific results,
which directly motivate our contribution. Let 
$\cal S$ is the set of graphs for which every connected component is a path or subdivided claw. 

\begin{theorem}[\cite{BBJPPL21}]\label{t-s}
For a graph $H\notin {\cal S}$, {\sc Steiner Tree}, and thus \stf{}, are \NP-complete for $H$-subgraph-free graphs.
\end{theorem}

\noindent
Hence, we may assume that $H\in {\cal S}$.  Bodlaender et al.~\cite{BJMOPPSV25} observed several properties of the gadget of  Bateni, Hajiaghayi and Marx~\cite{BHM11} for proving \NP-completeness of \stf{} for graphs of treewidth~$3$.

\begin{theorem}[\cite{BHM11}]\label{t-hard}
\stf{} is \NP-complete for $(3K_{1,3},2K_{1,3}+P_4,K_{1,3}+2P_4,3P_4,\\ S_{1,1,8},S_{1,4,5})$-subgraph-free graphs with treewidth~$3$, treedepth~$4$, and $3$-deletion set number~$2$.
\end{theorem}

\noindent
They also proved the following polynomial-time results:

\begin{theorem}[\cite{BJMOPPSV25}]\label{t-pol}
For a graph~$H$, \stf{} on $H$-subgraph-free graphs is polynomial-time solvable if  
$H\subseteq 2K_{1,3}+P_3$, 		
$2P_4+P_3$, 					
$P_9$, or 					
$S_{1,1,4}$.
\end{theorem}

\noindent
Our polynomial-time cases in Theorem~\ref{t-dicho} extend Theorem~\ref{t-pol}. The cases $H=2P_4+P_3$ and $H=P_9$ generalize to $H=P_{11}$, and $H=S_{1,1,4}$ generalizes to $H\in \{S_{1,3,6}, S_{2,2,7}, S_{2,3,5},S_{2,4,4},S_{3,3,4}\}$. 
The case $H=2K_{1,3}+P_3$ is not generalized. We discuss this case and directions for future work in Section~\ref{s-con}. On a final note,
Bodlaender et al.~\cite{BJMOPPSV25} also observed that for every graph~$H$, a polynomial-time algorithm for \stf{} on $H$-subgraph-free graphs implies one for $(H+P_2)$-subgraph-free graphs. This yields a small extension of Theorems~\ref{t-dicho} and~\ref{t-pol}.
 
\section{Preliminaries and Basic Observations}\label{s-pre}

Let $G=(V,E)$ be a graph.
For $v \in V$, we denote  the \emph{neighbourhood} of $v$ by $N(v) = \{u \in V(G)\; |\; uv \in E(G)\}$. For a set $S \subseteq V(G)$, we let $N(S) = \bigcup_{u \in S} N(u) \setminus S$. We write $G[S]$ for the subgraph of $G$ \emph{induced} by $S$. 
For $uv\in E$, the graph $G \cup \{uv\}$ is the graph with vertex set $V(G)$ and edge set $E(G) \cup \{uv\}$. We let $G-S = G[V\setminus S]$. If $S = \{u\}$, we write $G-u$. 
We \emph{identify} vertices $u$,$v$ by replacing them by a single vertex $w$ with neighbourhood $N(\{u,v\})$. 

A vertex $v$ is a \emph{cut vertex} in a connected graph $G$ if $G-v$ has more than one connected component. A graph with no cut vertices is \emph{$2$-connected}. A \emph{$2$-connected component} of a graph is a maximal subgraph that is $2$-connected.
The \emph{size} of a component is the number of vertices contained in it. The graph $C_r$ is a cycle on $r$ vertices. 

A \emph{tree decomposition} of $G$ is a tree $T$ with a bag $\beta(t) \subseteq V(G)$ associated with each $t \in V(T)$ such that: $\bigcup_{t\in V(T)} \beta(t) = V(G)$; for every edge $uv \in E(G)$, there is a $t \in V(T)$ for which $u,v\in\beta(t)$; for every vertex $v \in V(G)$, the set $\{t \in V(T) \mid v \in \beta(t)\}$ induces a connected subtree of $T$. The \emph{width} of the tree decomposition is $\max_{t \in V(T)} |\beta(t)|-1$. The \emph{treewidth} of $G$ is the minimum width of any tree decomposition of $G$.

A graph class is \emph{hereditary} if it is closed under deleting vertices.

\begin{lemma}[\cite{BJMOPPSV25}]\label{l-2con}
For every hereditary graph class ${\cal G}$, if {\sc Steiner Forest} is polynomial-time solvable for the subclass of $2$-connected graphs of ${\cal G}$, then it is polynomial-time solvable for ${\cal G}$.
\end{lemma}

\noindent
The following result is folklore; see Section~\ref{s-fp-prelim} for a proof.

\begin{restatable}{lemma}{lemExponential}
\label{l-exponential}
\stf{} can be solved in $4^nn^{O(1)}$ time.
\end{restatable}

\section{Technical Overview}
\fl{In the following we give an overview of the technical content of the paper.
We begin in Section~\ref{s-techniques} by defining our novel graph decomposition and sketch several algorithms leveraging it. 
We then outline in Section~\ref{s-poly} how they can be applied to obtain polynomial-time algorithms for $H$-subgraph-free graphs. 
We conclude the overview by presenting a hardness result in Section~\ref{s-hard} which shows that those techniques from Section~\ref{s-techniques} are in some sense tight.}

\subsection{\fl{Algorithmic Results on Citrus Bushes}}\label{s-techniques}
We first define the components that are allowed to remain when we remove a cutset.
Our terminology is inspired by biological citruses.
Full proofs for this section appear in Section~\ref{s-fp-techniques}.

\begin{restatable}{definition}{defWedge}
Let $G$ be a graph. A \emph{wedge} in $G$ is a set $L \subseteq V(G)$ with $|L| \geq 3$ for which there exist two vertices $x,y \in L$ such that $N(L \setminus\{x,y\}) = \{x,y\}$ and $G[L \setminus \{x,y\}]$ is connected. We call $x,y$ the \emph{ends} of the wedge. We let $G_L$ denote the graph $G[L]$ minus the edge $xy$ (if it exists). A wedge $L$ with ends $x,y$ is \emph{juicy} if $G_L \cup \{xy\}$ has treewidth~$2$. A wedge $L$ is \emph{seeded} if $|L|=4$. For any integer $\ell \geq 5$, a wedge $L$ is \emph{$\ell$-pulped} if $5 \leq |L| \leq \ell$.
\end{restatable}

\noindent
If $L \not= V(G)$ is a wedge with ends $x,y$, then $\{x,y\}$ forms a cutset in the graph, as the vertices of $L \setminus \{x,y\}$ cannot be adjacent to the vertices of $V(G) \setminus L$ by definition. Besides the case $L=V(G)$, another edge case is when $x$ is a cut vertex, in which case $y$ can be any vertex of~$L$. We note that any wedge with $|L| = 3$ is juicy. An $\ell$-pulped or seeded wedge may be a juicy wedge, in which case we will consider it to be a juicy wedge for the purposes of this paper. Moreover, by generalizing an argument of Bodlaender et al.~\cite{BJMOPPSV25}, we can show that we can always turn a seeded wedge into one or more juicy wedges with the same ends, without changing the problem:

\begin{restatable}{lemma}{lemWedgeTransform}
    \label{l-wedge-transform}
Let $G$ be a graph and $T$ a set of terminal pairs in $G$. Then we can compute in polynomial time a subgraph $G'$ of $G$ such that any seeded wedge in $G'$ is juicy and the size of a minimum Steiner forest for $(G,T)$ is equal to that of a minimum Steiner forest for $(G',T)$.
\end{restatable}

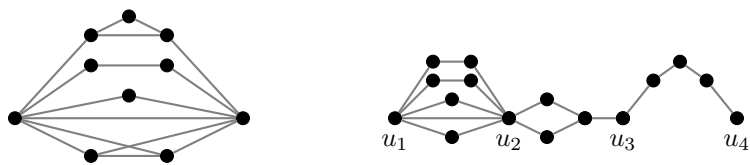
\begin{figure}
    \centering
    \input{figures/lemons.tex}
    \caption{Left: A lemon. Right: A citrus bush flowering from a $4$-vertex path $u_1u_2u_3u_4$, where each citrus is has only juicy wedges. Note that $\{u_1,u_2,u_3,u_4\}$ is the stem set.}
    \label{f-lemon}
\end{figure}

\noindent
We now consider the combination of wedges, see Figure~\ref{f-lemon} (left):

\begin{restatable}{definition}{defCitrus}
Let $G$ be a graph. A \emph{citrus} in $G$ is a pair of vertices $x,y \in V(G)$ and a 
set~$\mathcal{L}$ of wedges in $G$ with ends $x,y$ such that $\mathcal{L} \not= \emptyset$ or $xy \in E(G)$. We call $x,y$ again the \emph{ends} of the citrus.
The \emph{vesicle set} $V(x,y,\mathcal{L})$ of a citrus is the vertex set that is the union of $\{x,y\}$ with the union of the wedges in $\mathcal{L}$. A citrus is a \emph{lemon} if its wedges are all juicy or seeded. For any integer $\ell \geq 5$, a citrus is an \emph{$\ell$-lemon} if its wedges are all juicy, seeded, or $\ell$-pulped and the number of $\ell$-pulped wedges is bounded by $\ell$.
\end{restatable}
\noindent
Note that a citrus with ends $x,y \in V(G)$ may have a nonempty wedge set and $xy \in E(G)$ at the same time. The definition is also slightly broader than that of a vertex cutset of size~$2$, as any edge of $G$ can also be a citrus. Finally, a lemon is trivially an $\ell$-lemon for any $\ell \geq 5$.

\begin{restatable}{definition}{defCitrusBush}
A graph $G$ is a \emph{citrus bush} if there exists a set $X \subseteq V(G)$, a graph $G'$ with vertex set $X$, a partition $C_1,\ldots,C_{|E(G')|}$ of $V(G) \setminus X$, and a bijection $\sigma : E(G') \rightarrow \{1,\ldots,|E(G')|\}$, such that for every $xy \in E(G')$, there is a citrus in $G$ with ends $x,y$ and wedge set $\mathcal{L}_{xy}$ such that $V(x,y,\mathcal{L}_{xy}) = C_{\sigma(xy)} \cup \{x,y\}$. We similarly define a \emph{lemon bush} and for any integer $\ell \geq 5$, an \emph{$\ell$-lemon bush}.
In all cases, we say that $G$ \emph{flowers from} $G'$ and call $X$ the \emph{stem set} of $G$, see Figure~\ref{f-lemon} (right).
\end{restatable}

\noindent
It is important to note that some of $C_1,\ldots,C_{|E(G')|}$ may be empty, particularly if $xy \in E(G)$ for some $x,y \in X$. Moreover, a lemon bush is trivially an $\ell$-lemon bush for any integer $\ell \geq 5$.

We first provide the following general (but simple) result, which demonstrates that a large cutset and large components still lead to a polynomial-time algorithm.

\begin{restatable}{theorem}{thmLemonTWtwo}
\label{t-lemon-tw2}
Let $G$ be a lemon bush flowering from a graph $G'$ of treewidth at most~$2$. Then we can solve \stf{} on $G$ in polynomial time.
\end{restatable}

\begin{proof}[Proof Sketch]
First, use Lemma~\ref{l-wedge-transform} to turn any seeded wedge into one or more juicy wedges with the same ends. Any juicy wedge has a tree decomposition of width~$2$ with a bag containing the ends of the wedge. Since the ends are connected by an edge in $G'$, we can augment a tree decomposition of $G'$ of width~$2$ to a tree decomposition of $G$ of width~$2$. Then we apply the algorithm for graphs of treewidth~$2$ by Bateni, Hajiaghayi, and Marx~\cite{BHM11}.
\end{proof}

\noindent
We now consider $\ell$-lemon bushes. We first show that we can deal with $\ell$-lemon bushes individually, when the terminal set is ``local'' to the $\ell$-lemon. We need \stf{} to be polynomial-time solvable in three cases to derive very general results. First, we give a notion that describes the structure of a solution.

\begin{restatable}{definition}{defBiStemmed}
Let $G$ be a graph and let $\mathcal{L}$ be a citrus in $G$ with ends $x,y$. Let $T$ be a set of terminal pairs and let $F$ be a Steiner forest for $(G,T)$. Then $F$ is \emph{bi-stemmed} with respect to the citrus if there is no path in $F$ between $x$ and $y$ using only vertices in $V(x,y,\mathcal{L})$ and \emph{intertwined} if such a path does exist. 
\end{restatable}

\noindent
We also use the following definition. Let $G$ be a graph and let $D \subseteq V(G)$. Then $G \dagger D$ denotes the graph obtained from $G$ by identifying the vertices of $D$ to a single vertex~$d$. For a set $T$ of pairs of vertices in $G$, we use $T \dagger D$ to denote the set of pairs obtained from $T$ by replacing each occurrence of an element of $D$ in a pair by $d$.

\begin{restatable}{lemma}{lemLemonSupport}
\label{l-l-lemon-support}
Let $G$ be an $n$-vertex graph, let $\ell \geq 5$ be any integer, and suppose $G$
contains an $\ell$-lemon, ends $x,y$, and wedge set $\mathcal{L}$ such that $V(x,y,\mathcal{L}) = V(G)$. Let $T$ be a set of terminal pairs in $G$. Then we can compute, in $2^{O(\ell^2)} n^{O(1)}$ time: 
\begin{enumerate}[(i)]
\item\label{l-l-lemon-support-i} a minimum Steiner forest for $(G \dagger \{x,y\},T \dagger \{x,y\})$;
\item\label{l-l-lemon-support-ii} a minimum Steiner forest for $(G,T)$ over all Steiner forests for which the $\ell$-lemon is intertwined;
\item\label{l-l-lemon-support-iii} a minimum Steiner forest for $(G,T)$.
\end{enumerate}
\end{restatable}

\smallskip
\noindent
Lemma~\ref{l-l-lemon-support} is sufficient to prove general theorems, including a generalization of Theorem~\ref{t-lemon-tw2}.

\begin{restatable}{theorem}{thmPathbush}
\label{t-pathbush}
Let $\ell \geq 5$ be an integer, let $G'$ be a path, and let $G$ be an $\ell$-lemon bush flowering from $G'$. Then we can solve \stf{} on $G$ in $2^{O(\ell^2)} n^{O(1)}$ time.
\end{restatable}
\begin{proof}[Proof Sketch]
Note that each $\ell$-lemon of the bush forms a $2$-connected component of $G$. Hence, by Lemma~\ref{l-2con}, we can apply Lemma~\ref{l-l-lemon-support}.
\end{proof}

\begin{restatable}{theorem}{thmCyclebush}
\label{t-cyclebush}
Let $\ell \geq 5$ be an integer, let $G'$ be a cycle, and let $G$ be an $\ell$-lemon bush flowering from $G'$. Then we can solve \stf{} on $G$ in $2^{O(\ell^2)} n^{O(1)}$ time.
\end{restatable}

\begin{proof}[Proof Sketch]
Since $G'$ is a cycle, there is at least one $\ell$-lemon for which the solution is bi-stemmed. We consider three cases, depending on the number of bi-stemmed $\ell$-lemons in an optimum solution. We give the first case in some detail, and sketch the other two cases.

\smallskip\noindent\underline{\emph{Case 1:}} If there is an optimum solution with a single bi-stemmed $\ell$-lemon, we guess it. There is a path in the solution between the ends of the guessed $\ell$-lemon. After projecting the terminal set to it and identifying its ends, we apply Lemma~\ref{l-l-lemon-support}(\ref{l-l-lemon-support-i}) to find the part of the solution inside the guessed $\ell$-lemon. For all other $\ell$-lemons, we can apply Theorem~\ref{t-pathbush}.

In detail, we branch over all edges in $G'$. Let $xy$ be the edge of $G'$ corresponding to the current branch. Let $\mathcal{L}_{xy}$ denote the $\ell$-lemon corresponding to $xy$ and let $C_{xy} = V(x,y,\mathcal{L})$. We create two new sets $T_{xy}$ and $T'_{xy}$ of terminals as follows. Let $(s,t) \in T$. If $s,t \in C_{xy}$, then add $(s,t)$ to $T_{xy}$. If $s,t \not\in C_{xy}$, then add $(s,t)$ to $T'_{xy}$. If $s \in C_{xy}$ and $t \not\in C_{xy}$, then add $(s,x)$ to $T_{xy}$ and $(t,x)$ to $T'_{xy}$. Finally, add $(x,y)$ to $T'_{xy}$. 
Compute a minimum Steiner forest $\hat{F}_{xy}$ for $(G[C_{xy}] \dagger \{x,y\}, T_{xy} \dagger \{x,y\})$ and compute a minimum Steiner forest $\hat{F}'_{xy}$ for $(G-(C_{xy} \setminus\{x,y\}) - \{xy\}, T'_{xy})$. This can be done in $2^{O(\ell^2)} n^{O(1)}$ time: $G[C_{xy}]$ is a single $\ell$-lemon and thus we apply Lemma~\ref{l-l-lemon-support}(\ref{l-l-lemon-support-i}); $G-(C_{xy} \setminus\{x,y\}) - \{xy\}$ is an $\ell$-lemon bush flowering from a path and thus we apply Theorem~\ref{t-pathbush}. Output the minimum $\hat{F}_{xy} \cup \hat{F}'_{xy}$ over all edges $xy$ of $G'$; we note that this union is performed over the edges, as there is a clear mapping from the edges of $G[C_{xy}] \dagger \{x,y\}$ to those of $G[C_{xy}]$. We omit the proof of optimality in this sketch.

\smallskip\noindent\underline{\emph{Case 2:}} If there is an optimum solution with two bi-stemmed $\ell$-lemons, we guess both of them. Going around the cycle, there is a path in the solution from one end $y_1$ of the first guessed $\ell$-lemon to an end $x_2$ of the second guessed $\ell$-lemon and similarly a path from the other end $y_2$ of the second guessed $\ell$-lemon to the other end $x_1$ of the first guessed $\ell$-lemon. To find the part of the solution inside the two guessed $\ell$-lemons, we identify $y_1, x_2$ and $y_2,x_1$, reducing to a single $\ell$-lemon. After projecting the terminal set appropriately, we apply Lemma~\ref{l-l-lemon-support}(\ref{l-l-lemon-support-iii}). For all other $\ell$-lemons, they form two $\ell$-lemon bushes flowering from a path and thus we apply Theorem~\ref{t-pathbush} after projecting the terminal set appropriately and adding the pairs $(y_1,x_2)$ and $(y_2,x_1)$ respectively.

\smallskip\noindent\underline{\emph{Case 3:}} If there is an optimum solution with three or more bi-stemmed $\ell$-lemons, we guess two of them that are closest along the cycle. Going around the cycle, there is a path in the solution from one end $y_1$ of the first guessed $\ell$-lemon to an end $x_2$ of the second guessed $\ell$-lemon. The $\ell$-lemons between these ends form an $\ell$-lemon bush flowering from a path and thus we apply Theorem~\ref{t-pathbush} after projecting the terminal set appropriately and adding the pair $(y_1,x_2)$. For the guessed $\ell$-lemons themselves, we identify $y_1,x_2$ and apply Theorem~\ref{t-pathbush} to the resulting $\ell$-lemon bush flowering from a path of length~$2$ after we project the terminal set appropriately. The remaining $\ell$-lemons again form an $\ell$-lemon bush flowering from a path and thus we apply Theorem~\ref{t-pathbush}.
\end{proof}

\noindent
Next, we consider $\ell$-lemon bushes with a stem set of bounded size. In fact, our algorithm holds for a more general case, which generalizes all polynomial cases of Theorem~\ref{t-deletion}.

\begin{restatable}{definition}{defEntangledBush}
A graph $G$ is an \emph{entangled} citrus bush if there is an independent set $Y$ in $G$ such that $G-Y$ is a citrus bush with stem set $X$ and $N_G(Y) \subseteq X$. We call $Y$ the \emph{tangle set} and $X$ still the stem set of $G$. An \emph{entangled lemon bush} and for any integer $\ell \geq 5$, an \emph{entangled $\ell$-lemon bush} are defined similarly.
\end{restatable}
\noindent
Observe that if, for an entangled $\ell$-lemon bush with stem set $X$, every $\ell$-lemon has an empty wedge set, then $X$ forms a vertex cover of $G$. Moreover, if some vertex of $Y$ is adjacent to exactly two vertices of $X$, then it forms a juicy wedge and can alternatively be seen as part of the $\ell$-lemon bush.

\begin{restatable}{theorem}{thmSolvetangle}
\label{t-solvetangle}
Let $\ell \geq 5$ be an integer and let $G$ be an entangled $\ell$-lemon bush with a stem set of size $k$. Then we can solve \stf{} on $G$ in $2^{O(k^2\ell^2)} \cdot (n(k+1))^{O(k)}$ time.
\end{restatable}
\begin{proof}[Proof Sketch]
Let $X$ be a stem set of $G$ and $Y$ a tangle set. Let $G' = G-Y$. Since $G'$ is an $\ell$-lemon bush, let $G''$ be the graph such that $G'$ flowers from $G''$ with stem set $X$. The intuition of the proof is to enumerate all possible patterns of a solution and find a minimum solution for each possible pattern. The pattern describes which vertices of $X$ are in the same component of a solution and which $\ell$-lemons must have an intertwined solution. At the same time, the pattern will prescribe exactly which vertices in $Y$ will connect the different components in $X$.

Let $H$ be the auxiliary graph on vertex set $X \cup Y$ and the following edges: for any $u,v \in X \cup Y$, $uv \in E(H)$ if and only if $uv \in E(G)$ or $uv \in E(G'')$. Branch on all subgraphs (\emph{patterns}) $P$ of $H$ such that $P$ is a forest, $X \subseteq V(P)$, and every vertex of $Y$ has degree at least~$2$ in $P$. For any such $P$, a Steiner forest $F$ of $(G,T)$ \emph{adheres} to $P$ if:
\begin{itemize}
\item if $uv \in E(P) \cap E(G'')$ (note that $u,v \in X$), then $F$ is intertwined with respect to the $\ell$-lemon corresponding to $uv$;
\item if $uv \in E(G'') \setminus E(P)$ (note that $u,v \in X$), then $F$ may be bi-stemmed with respect to the $\ell$-lemon corresponding to $uv$;
\item if $y \in Y \cap V(P)$, then $N_P(y) = N_F(y)$ (note that $N_P(y),N_F(y) \subseteq X$);
\item if $y \in Y \setminus V(P)$, then $y$ has degree at most~$1$ in $F$;
\end{itemize}
We can show that, for any minimum Steiner forest $F$ for $(G,T)$, there exists a pattern $P_F$ such that $F$ adheres to $P_F$.

We branch on the patterns; there are $(n(k+1))^{O(k)}$ of them. For each pattern, we will find a minimum Steiner forest for $(G,T)$ that adheres to the pattern (if one exists). Note that the connected components of the pattern induce a partition $\mathcal{X}$ of the vertices of $X$ (but any part of $\mathcal{X}$ does not necessarily induce a connected subgraph of $G''$). The main challenge is to decide where to connect the terminals in $Y \setminus V(P)$; call these bad terminals. By an analysis of the structure of a solution and several polynomial-time rejection and reduction rules, we can reduce the number of candidate parts of $\mathcal{X}$ that bad terminals $y$ can be connected to in the solution to at most two. If there is one candidate part, then we add an arbitrary edge from it to $y$ to the solution. If there are two candidates, then we can similarly reduce the degree of $y$ to $2$. Then we treat it as a juicy wedge between two vertices of $X$ (it may be that no edge exists in $G''$ between these vertices, but then we add the edge just for this branch). 

Finally, we find the solutions for the $\ell$-lemons. This requires a careful projection of the terminal set on the $\ell$-lemons, depending on the unique part to which a bad terminal has an edge. Then, for each $uv \in E(G'') \cap E(P)$, we need to find an intertwined solution for the $\ell$-lemon $\mathcal{L}_{uv}$ corresponding to $uv$. For each $uv \in E(G'') \setminus E(P)$ for which $u,v$ are in the same part of $\mathcal{X}$, we need to find a solution in $G[V(u,v,\mathcal{L}_{uv})] \dagger \{u,v\}$ for $\mathcal{L}_{uv}$, because the solution will contain a path that connects $u,v$ but has no edge in $G[V(u,v,\mathcal{L}_{uv})]$. For all remaining $uv$ in $E(G'')$, we note that their ends are between two parts of $\mathcal{X}$. For every pair of parts of $\mathcal{X}$, we group such $\ell$-lemons together into a super-lemon, which is a $k^2\ell$-lemon. Inside this super-lemon, we find a bi-stemmed solution. By application of Lemma~\ref{l-l-lemon-support}, these solutions can be found in $2^{O(k^2\ell^2)}n^{O(1)}$ time.
\end{proof}

\subsection{Polynomial-Time Results}\label{s-poly}

In this section, we discuss the polynomial cases from Theorem~\ref{t-dicho}.
Full proofs of these results can be found in Section~\ref{s-fp-poly}. 
In all our proof sketches, we assume that
for the graph~$H$ under consideration,
$G$ is a $H$-subgraph-free graph from an instance $(G,T)$ of {\sc Steiner Forest}. Moreover, by Lemma~\ref{l-2con}, we will always assume that $G$ is $2$-connected.

We start with the case $H=P_{11}$.

\begin{restatable}{theorem}{thmPIIfree}
\label{t-p11free}
\stf{} is polynomial-time solvable for $P_{11}$-subgraph-free graphs.
\end{restatable}

\begin{proof}[Proof Sketch]
We first check if $G$ contains a $P_9$. If not, we apply Theorem~\ref{t-pol}.
Else we~found a $9$-vertex path $P$.
We check if $G$ has a $P_{10}$.
If not, we analyse the components of 
$G-V(P)$~and show that $G$ is an entangled lemon bush with stem set $V(P)$ of size~$9$ so we apply Theorem~\ref{t-solvetangle}.
If we found a $P_{10}$, we use similar but more involved arguments as in the $P_9$ case.
\end{proof}

\noindent
We now consider the case $H=S_{1,3,6}$.

\begin{restatable}{theorem}{thmSIEbfree}
\label{t-s136}
\stf{} is polynomial-time solvable for $S_{1,3,6}$-subgraph-free graphs.
\end{restatable}

\begin{proof}[Proof Sketch]
    We determine a longest path $P = (v_1, \dots, v_r)$ in $G$.
    By Theorem~\ref{t-p11free}, we may assume that $r \geq 11$.  
 If $r = 11$, we show that $G$ is an entangled lemon bush with stem set~$V(P)$ of size~$11$ and apply Theorem~\ref{t-solvetangle}.  
 If $r \geq 12$, we show $V(G) = V(P)$.
    Hence, the case $r=12$ is trivial. Else,
    $r \geq 13$. In this case we find that $G=C_r$ and apply Theorem~\ref{t-lemon-tw2}. 
\end{proof}

\noindent
We now consider the cases $H\in \{S_{2,2,7},
S_{2,3,5},S_{2,4,4}\}$, for which
 we use the following lemma.

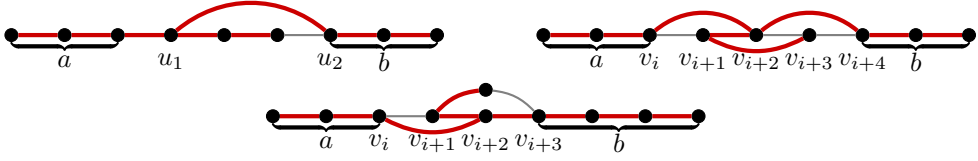
\begin{figure}[t]
    \centering
    \input{figures/s2ab}
    \caption{Forbidden jumps due to occurrences of $S_{2,a,b}$ in the proof of Lemma~\ref{l-s2ab}.}
    \label{f-s2ab-start}
\end{figure}

\begin{restatable}{lemma}{lemSTwoABFree}
\label{l-s2ab}
For $b\geq a \geq 2$, let $G$ be a $2$-connected $S_{2,a,b}$-subgraph-free graph with a longest path $P$.
If $|V(P)|\geq 3(a+b)$ and $V(P)$ is a vertex cover of $G$, then $G$ is a lemon bush flowering from a cycle.
\end{restatable}

\begin{proof}[Proof Sketch]
 We call a path a \emph{jump} for $P$ if only the endpoints of this path are contained in $P$. As $V(P)$ is a vertex cover of $G$, every jump has length at most~$2$.
 We first show that no jumps or overlapping jumps can appear as in the three situations displayed in Figure~\ref{f-s2ab-start}, as in all these cases we would find a forbidden $S_{2,a,b}$.
 This reasoning allows us to find many cut-sets of size~$2$ along $P$.
We then show that all these cut-sets form the ends of a lemon.
So we conclude that the ``middle'' of $P$ is a lemon bush flowering from a path. 

We continue by considering the vertices at the ends of the path. We observe that there cannot be any jump between a vertex close to the end of $P$ and a vertex sufficiently far away from both ends of $P$, as this results in some forbidden $S_{2,a,b}$. As $G$ is $2$-connected, this means that there must be a jump between a vertex close to the start of $P$ and a vertex close to the end of $P$. This jump implies that $G$ contains a long cycle. Now we can re-use our structural arguments used for the middle of $P$ repeatedly to find that $G$ is lemon bush flowering from a cycle.
\end{proof}

\begin{restatable}{theorem}{thmSTwoTwoSevenFree}
\label{t-s227}
\stf{} is polynomial-time solvable for $S_{2,2,7}$-subgraph-free graphs.
\end{restatable}

\begin{proof}[Proof Sketch]
We find a longest path $P$ of $G$. 
If $|V(P)|\leq 10$, we apply Theorem~\ref{t-p11free}. Suppose $|V(P)|\geq 11$.
We prove that $V(P)$ is a vertex cover of $G$.
If $|V(P)| \leq 26$, then we use the known {\sf FPT} algorithm for {\sc Steiner Forest} parameterized by vertex cover number~\cite{GHKKO22,BJMOPPSV25,FL25}.
If $|V(P)| \geq 27 = 3\cdot(2+7)$, then
    we apply Lemma~\ref{l-s2ab} and obtain that $G$ is a lemon bush flowering from a cycle, meaning we can apply Theorem~\ref{t-lemon-tw2}.
\end{proof}

\noindent
We obtain our next result, 
Theorem~\ref{t-s235}, by the same approach as for Theorem~\ref{t-s227}, but for Theorem~\ref{t-s244} we need some other arguments as well.

\begin{restatable}{theorem}{thmSTwoThreeFiveFree}
\label{t-s235}
\stf{} is polynomial-time solvable for $S_{2,3,5}$-subgraph-free graphs.
\end{restatable}

\begin{restatable}{theorem}{thmSTwoFourFourFree}
\label{t-s244}
\stf{} is polynomial-time solvable for $S_{2,4,4}$-subgraph-free graphs.
\end{restatable}

\begin{proof}[Proof Sketch]
As before, take a longest path $P$ of $G$. If $|V(P)| \geq 12$, then again we find that $V(P)$ is a vertex cover of $G$.
If $|V(P)| = 11$, we analyze the structure of $G-V(P)$ to~show $G$ is an entangled lemon bush with stem set $V(P)$ of size~$11$, so we apply Theorem~\ref{t-solvetangle}.
\end{proof}

\noindent
Finally, we consider the case $H=S_{3,3,4}$ whose proof is similar but more involved than the proofs of Theorems~\ref{t-s227}--\ref{t-s244}.
 As a first difference, we need to prove a stronger lemma than Lemma~\ref{l-s2ab}.

 \begin{restatable}{lemma}{lemSThreeABFree}
\label{l-s3ab}
For $b\geq a \geq 3$, let $G$ be a $2$-connected $S_{3,a,b}$-subgraph-free graph with a longest path $P$.
If $|V(P)|\geq 3a+2b-4$,  then $G$ is a citrus bush flowering from a cycle. Moreover, every citrus in this bush is either a lemon or has a vesicle set with size at most~$6$. 
\end{restatable}

\begin{restatable}{theorem}{thmSThreeThreeFourFree}
\label{t-s334}
\stf{} is polynomial-time solvable for $S_{3,3,4}$-subgraph-free graphs.
\end{restatable}

\begin{proof}[Proof Sketch]
We find again a longest path $P$ in $G$, but now we must also deal with the case where $G-V(P)$ contains components of size~$2$. If $|V(P)| \leq 13$, then we show that $G$ is an entangled lemon bush with stem set $V(P)$ and we can apply Theorem~\ref{t-solvetangle}. Else, by performing a much more technical analysis, Lemma~\ref{l-s3ab} implies that $G$ is an $\ell$-lemon bush flowering from a cycle. Note that this is in contrast to Lemma~\ref{l-s2ab}, which implied  that $G$ was a lemon bush, i.e. there were no $\ell$-pulped lemons. We conclude by applying Theorem~\ref{t-cyclebush}.
\end{proof}

\subsection{Hardness Result}\label{s-hard}

In this section we prove our hardness result: Theorem~\ref{t-newhard}. The last part of it follows from the definition of $c$-deletion set number.
Full proof details and the definition of a CSP instance can be found in Section~\ref{s-fp-hard}.

\begin{figure}[b]
    \centering
    \includegraphics[page=16,width=0.6\linewidth]{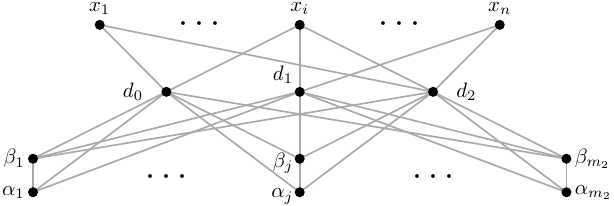}
    \caption{The graph $G$ from the proof of Theorem~\ref{t-newhard}.}
    \label{fig:nph-structure}
\end{figure}

\begin{restatable}{theorem}{thmNewHard}
\label{t-newhard}
\stf{} is \NP-complete for graphs with $2$-deletion set number~$3$ of treewidth~$3$, and thus in particular for $4P_3$-subgraph-free graphs.
\end{restatable}

\begin{proof}[Proof Sketch]
We reduce from the CSP problem $\textup{CSP}(\overline{\mathcal{R}})$ that has domain $\overline{D}=\{0,1,2\}$ and a set~$\overline{\mathcal{R}}$ with six relations: for every $d\in \overline D$, the unary relation $x\neq d$, and for every $d\in \overline D$, the binary relation $(x=d)\to (y=d)$. In Section~\ref{s-fp-hard}, we give a self-contained proof showing that $\textup{CSP}(\overline{\mathcal{R}})$ is \NP-complete.

From an instance $I$ of $\textup{CSP}(\overline{\mathcal{R}})$ with $n$ variables and $m_1$ unary constraints and $m_2$ binary constraints, we construct a graph $G$ and a set $T$ of terminal pairs as follows (see also Figure~\ref{fig:nph-structure} for an illustration of $G$). We introduce three vertices $d_0$, $d_1$, $d_2$. For every variable $x_i$, we introduce a vertex $x_i$ and, for $\ell=0,1,2$, connect $x_i$ to $d_\ell$ \emph{unless} the constraint $(x_i\neq \ell)$ exists. For every binary 
constraint~$c_j$ of the form $(x_{i_1}=\ell)\to (x_{i_2}=\ell)$, we introduce two adjacent vertices $\alpha_j$ and~$\beta_j$, and we make both of them adjacent to $d_0$, $d_1$, $d_2$ \emph{except} that $\alpha_j$ is not adjacent to~$d_\ell$. Furthermore, we add the terminal pairs $(x_{i_1},\alpha_j)$ and $(x_{i_2},\beta_j)$ to $T$.

It is not difficult to prove that $\{d_0, d_1, d_2\}$ is a $2$-deletion set of $G$, and that $G$ has treewidth~$3$. 
We can also show that  $I$ has a satisfying assignment if and only if $(G,T)$ has a Steiner forest with $k=n+2m_2$ edges. 
\end{proof}

\noindent
{\bf Remark.}
We compare Theorem~\ref{t-hard} and Theorem~\ref{t-newhard} to Lemma~\ref{l-l-lemon-support}.
By inspecting the proof of Theorem~\ref{t-hard} in~\cite{BHM11}, we notice that the \NP-hardness reduction constructs a citrus with an unbounded number of $5$-pulped wedges that have treewidth~$3$ (when the edge between the two vertices of the deletion set is added). Hence, the bound in the definition of $\ell$-lemons on the number of $\ell$-pulped wedges and the bound of~$2$ on the treewidth of juicy wedges is necessary for Lemma~\ref{l-l-lemon-support} (particularly part~\ref{l-l-lemon-support-iii}), unless P$=$NP. 
Defining wedges as having two ends (as opposed to three or more) is also necessary: allowing wedges with three ends would allow exactly the structure shown in Figure~\ref{fig:nph-structure} -- an \NP-complete case by Theorem~\ref{t-newhard}.

\section{Two Dichotomies and Discussions}\label{s-con}

\medskip
\noindent
We first prove Theorems~\ref{t-dicho} and~\ref{t-deletion} followed with some discussion.

\MakeDichoThm*

\begin{proof}
Let $H$ be a connected graph. If $H\in {\cal S}$, then we apply Theorem~\ref{t-s}.
From now on, assume that $H\in {\cal S}$. As $H$ is connected, this means that $H$ is a path or a subdivided claw.
Suppose $H=P_r$ for some $r\geq 1$.
If $r\leq 11$, apply Theorem~\ref{t-p11free}.
If $r\geq 12$, apply Theorem~\ref{t-hard} as $3P_4\subseteq H$.
If $H=S_{h,i,j}$ for some $1\leq h\leq i\leq j$, then we consider all possible cases:
\begin{itemize}
\item $h\geq 4$, $i\geq h$ and $j\geq i$: as $3P_4\subseteq H$, we apply Theorem~\ref{t-hard};
\item $h=3$, $i\geq 4$ and $j\geq i$: as $4P_3 \subseteq H$, we apply Theorem~\ref{t-newhard};
\item $h=3$, $i=3$ and $j\geq 5$: as $4P_3 \subseteq H$, we apply Theorem~\ref{t-newhard};
\item $h=3$, $i=3$ and $j\leq 4$: as $H\subseteq S_{3,3,4}$, we apply Theorem~\ref{t-s334};
\item $h=2$, $2\leq i\leq j$ and $j\geq 8$: as $S_{1,1,8}\subseteq H$, we apply Theorem~\ref{t-hard};
\item $h=2$, $3\leq i\leq j$ and $6\leq j \leq 7$: as $4P_3\subseteq H$, we apply Theorem~\ref{t-newhard};
\item $h=2$, $i=2$ and $6\leq j\leq 7$: as $H\subseteq S_{2,2,7}$, we apply Theorem~\ref{t-s227};
\item $h=2$, $4\leq i \leq 5$ and $j=5$: as $S_{1,4,5}\subseteq H$, we apply Theorem~\ref{t-hard};
\item $h=2$, $2\leq i\leq 3$ and $j=5$: as $H\subseteq S_{2,3,5}$, we apply Theorem~\ref{t-s235};
\item $h=2$, $2\leq i \leq j$ and $j\leq 4$: as $H\subseteq S_{2,4,4}$, we apply Theorem~\ref{t-s244};
\item $h=1$, $1\leq i\leq j$ and $j\geq 8$: as $S_{1,8,8}\subseteq H$, we apply Theorem~\ref{t-hard};
\item $h=1$, $3\leq i\leq 7$ and $j=7$: as $4P_3\subseteq H$, we apply Theorem~\ref{t-newhard};
\item $h=1$, $1\leq i\leq 2$ and $j=7$: as $H\subseteq S_{2,2,7}$, we apply Theorem~\ref{t-s227};
\item $h=1$, $4\leq i\leq j$ and $5\leq j\leq 6$: as $S_{1,4,5}\subseteq H$, we apply Theorem~\ref{t-hard};
\item $h=1$, $1\leq i\leq 3$ and $5\leq j\leq 6$: as $H\subseteq S_{1,3,6}$, we apply Theorem~\ref{t-s136}; 
\item $h=1$, $2\leq i\leq j$ and $j\leq 4$: as $S_{2,2,4}\subseteq H$, we apply Theorem~\ref{t-s244}.
\end{itemize}
\noindent
As the above list is exhaustive, the theorem has been proven.
\end{proof}

\MakeDeletionThm*

\begin{proof}
Gima et al.~\cite{GHKKO22} proved that \stf{} is {\sf FPT} when parameterized by vertex cover number, which shows the case $c=1$, $k\geq 0$.
Bodlaender et al.~\cite{BJMOPPSV25} proved the case $c=2$, $k\leq 3$.
The case $c\geq 3$, $k=1$ follows from Lemma~\ref{l-2con}. In the remaining cases, either $c=2$, $k\geq 3$ and we apply Theorem~\ref{t-newhard}, or $c\geq 3$, $k\geq 2$ and we apply Theorem~\ref{t-hard}.
\end{proof}

\noindent
In light of the above, we discuss two natural next steps.

The first direction is to further exploit our technique in order to establish a complexity classification for disconnected~$H$ as well. The current state-of-the-art for disconnected~$H$ is given in Theorem~\ref{t-pol}. We recall that \stf{} is polynomial-time solvable on $(H+P_2)$-subgraph-free graphs if it is so for $H$-subgraph-free graphs~\cite{BJMOPPSV25}. Hence, in Theorem~\ref{t-pol}, only the cases where $H\in \{2K_{1,3},2P_4\}$ are relevant. Our dichotomy for connected~$H$ captures the case where $H=2P_4$, but not the case $H=2K_{1,3}$. It is not difficult to slightly generalize the case $H=2K_{1,3}$. However, a full classification for disconnected $H$ requires dealing with some highly challenging cases, and we leave this for future work.

The second direction is to improve Theorem~\ref{t-solvetangle}. Recall that {\sc Steiner Forest} is fixed-parameter tractable parameterized by the vertex cover number of the graph~\cite{GHKKO22,BJMOPPSV25,FL25}. Hence, it is natural to investigate whether {\sc Steiner Forest} is fixed-parameter tractable for entangled $\ell$-lemon bushes parameterized by the size of the stem set.

\section{Basic Results}\label{s-fp-prelim}

The remainder of this paper is dedicated to the presentation of our results with full details.
We first present some lemmas and theorems that we need for our proofs, together with some additional terminology.

Let $G = (V(G),E(G))$ be a graph and $v \in V(G)$.
The \emph{degree} of $v$ is the size of the neighbourhood of $v$.
The \emph{length} of a path is its number of edges. The \emph{distance} of two vertices is the length of a shortest path connecting them.
We say that we \emph{contract} an edge if we identify its two end-vertices.

A non-standard term which we will use in several of our proof concerns jumps. A {\it jump} with respect to some path $P$, is a path between some pair $u,v \in V(P)$ that is internally vertex-disjoint from $P$ and edge disjoint from $P$. The {\it length} of the jump is the length of this path.

We will need the following observation.

\begin{observation}\label{o-structure-p4sfree}
    Let $G$ be a $P_4$-subgraph-free graph. Then $G$ is either a triangle or a star.
\end{observation}

\noindent
We also frequently make use of the following two lemmas.

\begin{lemma}\label{l-longestpath}
For every graph $H\in {\cal S}$, it is possible to find a longest path or longest cycle in an $H$-subgraph-free graph in polynomial time.
\end{lemma}

\begin{proof}
Let $H\in {\cal S}$.
    It follows from~\cite{RS84} that the class of $H$-subgraph-free graphs has bounded treewidth. Hence, we may use the algorithm from~\cite{CNPPRW22}.
\end{proof}

\begin{lemma}\label{l-adjacencytopath}
    Let $G = (V,E)$.
    Let $P = (v_1 , \dots , v_{r})$ be a longest path in $G$, and let $D$ be a connected component of $G-V(P)$.
    Then the following holds:
    \begin{enumerate}
        \item $D$ is not adjacent to two consecutive vertices of $P$, and
        \item $D$ does not contain vertices $a,b$ such that $a\,v_i\in E$ and $b\, v_{i+2}\in E$, for $i \in \{1,\dots, r-2\}$.
    \end{enumerate}
\end{lemma}
\begin{proof}
    Let $D$ be a connected component of $G-V(P)$.
   We first prove 1. 
Suppose there is $i \in \{1,\dots, r-1\}$ such that $D$ is adjacent to $v_i$ and $v_{i+1}$. Then, $(v_1 , \dots ,v_i ,D , v_{i+1} , \dots , v_{r})$ contains a $P_{r+1}$, a contradiction.
    
We now prove 2. Suppose there are $a,b \in D$ such that  $a$ is adjacent to $v_i$ and $b$ to $v_{i+2}$, for $i \in \{1,\dots, r-2\}$. Let $Q$ be any path between $a$ and $b$ in $D$. Then, $(v_1 , \dots ,v_i ,a ,Q , b , v_{i+2} , \dots , v_{r})$ contains a $P_{r+1}$, a contradiction. 
\end{proof}

\noindent
We will also need the following known results that we will use as lemmas in our proofs.

\begin{lemma}[\cite{BHM11}] \label{l-tw}
{\sc Steiner Forest} is polynomial-time solvable on graphs of treewidth at most~2.
\end{lemma}

\begin{lemma}[\cite{BJMOPPSV25}] \label{l-2d2}
{\sc Steiner Forest} is polynomial-time solvable on graphs with a $2$-deletion set of size at most~$2$.
\end{lemma}

\noindent
Bodlaender et al.~\cite{BJMOPPSV25} and Feldmann and Lampis~\cite{FL25} independently showed the following result, which was an improvement of the original {\sf FPT} algorithm by Gima et al.~\cite{GHKKO22}.

\begin{theorem}[\cite{BJMOPPSV25,FL25}] \label{t-vc}
{\sc Steiner Forest} has a $2^{O(k \log k)} n^{O(1)}$ time algorithm, where $k$ is the vertex cover number of the input graph, even in the weighted case.
\end{theorem}

\noindent
We need two more lemmas that help us to reduce the input size.

\begin{lemma}\label{l-benjamins}    
    If $(G,T,k)$ is an instance of \stf{} and $x,y \in V(G)$ are two vertices such that $N(x) \subsetneq N(y)$ and $x$ is not a terminal, then $(G,T,k)$ is a yes-instance if and only if $(G - x, T,k)$ is a yes-instance.    
\end{lemma}

\begin{proof}
    Since $x$ is not a terminal, it is used only as an intermediate vertex to connect terminals. Thus, in any solution to $(G,T,k)$, paths using $x$ can be rerouted through $y$ without changing the solution weight.
    Formally, let $F \subseteq E(G)$ be a solution for $(G,T,k)$. Then $(F\setminus \{ux : u \in V(G)\}) \cup \{uy : ux \in F\}$ is a solution of $(G-x,T,k)$ of equal weight to $F$. Conversely, as $x$ is not a terminal, any solution of $(G-x,T,k)$ is a solution of $(G,T,k)$.
\end{proof}

\begin{lemma}\label{l-adj-terminals}
    Let $(G,T,k)$ be an instance of \stf{} and $xy \in E(G)$ be an edge such that $(x,y)\in T$. Let $G'$ be the graph obtained by identifying $x$ with $y$ and $T' = T \setminus \{(x,y)\}$. Then $(G,T,k)$ is a yes-instance, if and only if $(G', T,k-1)$ is a yes-instance.
\end{lemma}

\begin{proof}
    We show that any solution $F \subsetneq E(G)$ for the instance $(G,T,k)$ can be transformed into an equal-size solution $F' \subsetneq E(G)$ satisfying $xy \in F'$.
    If $xy\in F$ then we are done. Otherwise, since $F$ is a solution and $(x,y)\in T$, there must be some path from $x$ to $y$ using edges of $F$. Let $e$ be some arbitrary edge along this path. Then $F \cup \{xy\} \setminus \{e\}$ is a solution to $(G,T,k)$. This can be seen from the fact that $F \cup \{xy\}$ contains a cycle which includes the edge $e$. 
    It follows, without loss of generality, that $xy$ is included in some optimal solution of $(G,T,k)$. That is, we can safely contract this edge to represent its inclusion in the solution, at cost $1$ meaning $(G,T,k)$ is a yes-instance, if, and only if $(G', T,k-1)$ is a yes-instance.    
\end{proof}

\noindent
We may assume throughout that $T$ is closed under transitivity. Indeed, if $(s,t),(t,x) \in T$, then a single component of any Steiner forest for $(G,T)$ will contain $s$, $t$, and $x$, and thus we can add $(s,x)$ to $T$ if not already present. Since $T$ is closed under transitivity, there exist disjoint sets $T_1,\ldots,T_p \subseteq V(G)$ such that $s,t \in T_i$ if and only if $(s,t) \in T$. We call the sets $T_1,\ldots,T_p$ \emph{schools}.
We use this notion in the proof of Lemma~\ref{l-exponential}.

\lemExponential*
\begin{proof}
Let $(G,T)$ be an instance of \textsc{Steiner Forest}. We assume that $T$ is closed under transitivity; let $T_1,\ldots,T_q \subseteq V(G)$ denote the schools of $T$. We perform the following dynamic programming algorithm. We call a set $S \subseteq V(G)$ a \emph{full set} if $\bigcup_{i \in I} T_i \subseteq S$ for some set $I \subseteq \{1,\ldots,1\}$. We call $I = I(S)$ the \emph{index set} of $S$. For any full set $S$, we compute a minimum Steiner forest $f(S)$ for the instance with graph $G[S]$ and schools $T_i$ for all $i \in I(S)$. Clearly, $f(S) = 0$ for any full set $S$ with $I(S) = \emptyset$. For any other full set $S$, compute $f(S)$ as the minimum $f(S')$ plus the size of a spanning tree of the graph $G[S \setminus S']$ (which has size $|S| - |S'| - 1$, or $\infty$ if $G[S \setminus S']$ is not connected) over all full sets $S'$ with $I(S') \subset I(S)$.

For any full set $S$, the computation always returns the size of a Steiner forest for the instance with graph $G[S]$ and schools $T_i$ for all $i \in I(S)$. To prove optimality, for some full set $S$, let $F$ be a minimum Steiner forest for the instance with graph $G[S]$ and schools $T_i$ for all $i \in I(S)$. Let $C$ be the subgraph induced by the edges of any connected component of $F$. Note that $E(C)$ is a spanning tree of $G[V(C)]$, and thus $|E(C)| = |V(C)|-1$. Clearly, any school is either fully contained in $C$ or fully disjoint from $C$. Let $S' = S \setminus C$. Note that $I(S) = I(S') \cup I(C)$ and $I(S') \subset I(S)$. Since $F[S']$ is a Steiner forest for the graph $G[S']$ and schools $T_i$ for all $i \in I(S')$, $f(S') \leq |E(F[S'])|$. Then $f(S) \leq f(S') + |S| - |S'| - 1 \leq |E(F[S'])| + |V(C)|-1 = |E(F)|$.

Finally, $f(V(G))$ is the size of a minimum Steiner forest for $(G,T)$. The running time of the dynamic programming algorithm is $4^nn^{O(1)}$.
\end{proof}

\section{\fl{Algorithmic Results on Citrus Bushes}}
\label{s-fp-techniques}

\fl{Here we give the results and proofs for Section~\ref{s-techniques} in full. We begin with the fundamental building-block for our decompositions: wedges.}

\subsection{Wedges}
We start by restating the definition of a wedge and several useful observations.
\defWedge*

\noindent
If $L \not= V(G)$ is a wedge with ends $x,y$, then $\{x,y\}$ forms a cutset in the graph, as the vertices of $L \setminus \{x,y\}$ cannot be adjacent to the vertices of $V(G) \setminus L$ by definition. Besides the case $L=V(G)$, another edge case is when $x$ is a cut vertex, in which case $y$ can be any vertex of $L$. 
It is important to observe that, in the definition of a wedge $L$, its ends $x,y$ are in $L$. Hence, a seeded wedge contains exactly two vertices in addition to its ends and an $\ell$-pulped wedge, for some $\ell \geq 5$, contains at least three and at most $\ell-2$ vertices in addition to its ends.

We note that any wedge with $|L| = 3$ is juicy. An $\ell$-pulped or seeded wedge may be a juicy wedge, in which case we will consider it to be a juicy wedge for the purposes of this paper. Moreover, as we show in the following lemma, we always turn a seeded wedge into a juicy wedge without changing the problem. The lemma generalizes an argument given in the proof of Lemma~\ref{l-2d2} in~\cite{BJMOPPSV25}.

\lemWedgeTransform*
\begin{proof}
Let $G$ be a graph. To prove the lemma, it suffices to show that in polynomial time we can find a seeded wedge $L$ in $G$ that is not juicy and, if one exists, we can compute a strict subgraph $G'$ of $G$ such that the size of a minimum Steiner forest for $(G,T)$ is equal to the size of a minimum Steiner forest for $(G',T)$.

We first observe the following. Let $L$ be a seeded wedge in $G$ that is not a juicy wedge, with ends $x,y$. Let $a,b$ denote the two vertices in $L \setminus \{x,y\}$. Since $L$ is a wedge, $ab \in E(G)$. If $ax \not \in E(G)$, then $L$ is a juicy wedge (use bags $\{x,y\}$, $\{x,y,b\}$, and $\{y,b,a\}$ to obtain a tree decomposition of $G_L \cup \{xy\}$ of width~$2$), a contradiction. Using similar arguments for any of the other four possible edges between $\{a,b\}$ and $\{x,y\}$, we conclude that $\{a,b\}$ is complete to $\{x,y\}$ in $G$.

Using the above observation, finding a seeded wedge in $G$ that is not juicy can be done straightforwardly by exhaustive enumeration of all $4$-vertex sets in $G$ and counting edges. If one exists, let $L$ be a seeded wedge in $G$ that is not juicy, with ends $x,y$ and let $a,b$ be the two vertices in $L \setminus \{x,y\}$. By the above observation, $\{a,b\}$ is complete to $\{x,y\}$ in $G$. We now seek to compute a strict subgraph $G'$ of $G$ such that the size of a minimum Steiner forest for $(G,T)$ is equal to the size of a minimum Steiner forest for $(G',T)$.

We distinguish two cases. In the first case, suppose that $(a,b) \in T$. Then we claim that a minimum Steiner forest of $(G,T)$ has the same size as a minimum Steiner forest of $(G-\{ax\},T)$. (An alternative proof could apply Lemma~\ref{l-adj-terminals}.) It suffices to prove the forward direction; the other is straightforward. Let $F$ be a minimum Steiner forest of $(G,T)$. It suffices to prove that we may assume that $ax \not\in E(F)$. If $F$ contains at least three edges of $G_L$, then we replace all of those by $ab, bx, by$ to obtain a Steiner forest of $(G,T)$ with size at most that of $F$ which does not contain $ax$. Hence, we may assume that $F$ contains at most two edges of $G_L$. Following this, if $ax \in E(F)$, then $F$ contains exactly one edge incident to $b$, as $(a,b) \in T$. If $ab \in E(F)$, then we replace $ax$ by $bx$; if $bx \in E(F)$, then we replace $ax$ by $ab$; if $by \in E(F)$, then we replace $ax$ by $ab$. In all three cases, we obtain a Steiner forest of $(G,T)$ with size at most that of $F$ which does not contain $ax$ (in the latter case, we observe that there must be a path between $x$ and $y$ in $F$ that avoids $a,b$). Hence, we may indeed assume that $ax \not\in E(F)$ and the equivalence is proven.

In the second case, suppose that $(a,b) \not\in T$. Then we claim a minimum Steiner forest of $(G,T)$ has the same size as a minimum Steiner forest of $(G-\{ab\},T)$. It suffices to prove the forward direction of this equivalence; the other is straightforward. Let $F$ be any minimum Steiner forest of $(G,T)$. If $ab \in E(F)$, then due to the minimality of $F$, also at least one of the edges between $\{a,b\}$ and $\{x,y\}$ is in $F$. Say it is $ax$ (all other cases are symmetrical). Then removing $ab$ from $F$ and adding $bx$ yields a Steiner forest of $(G-\{ab\},T)$ of the same size as $F$. Hence, the equivalence is proven.

In both cases, we obtain a strict subgraph $G'$ of $G$ in which the size of a minimum Steiner forest for $(G,T)$ is equal to the size of a minimum Steiner forest for $(G',T)$. Also note that $L$ has become one or two juicy wedges, respectively.
\end{proof}

\subsection{Citrus}
We now consider the combination of wedges.
\defCitrus*

\noindent
We extend the definition by the notion of a \emph{juicy citrus}, which is a citrus bush that consists of only juicy wedges.

\noindent
For sake of clarity, we note that a citrus with ends $x,y \in V(G)$ may have a nonempty wedge set and $xy \in E(G)$ at the same time. Moreover, the definition is slightly broader than that of a vertex cutset of size~$2$, as any edge of $G$ can also be a citrus. Finally, a lemon is trivially an $\ell$-lemon for any integer $\ell \geq 5$.

The following proposition is an easy consequence of the definition of a juicy wedge.

\begin{proposition} \label{p-juicy-citrus}
Let $G$ be a graph and suppose $G$ contains a juicy citrus with ends $x,y$ and wedge set $\mathcal{L}$ such that $V(x,y,\mathcal{L}) = V(G)$. Then $G$ has treewidth at most~$2$.
\end{proposition}
\begin{proof}
Let $L \in \mathcal{L}$ be any juicy wedge. By definition, the treewidth-bound of a juicy wedge is based on $G_L \cup \{xy\}$. Hence, $L$ must have a tree decomposition of minimum width ($2$) in which $x,y$ are in the same bag. In fact, we can always obtain such a decomposition in which $x,y$ are the only two vertices in a bag. We now glue such a tree decomposition of each juicy wedge on a bag $\{x,y\}$. Note that this also accounts for the edge $xy$ if it exists in $G$. Hence, $G$ has treewidth at most~$2$.
\end{proof}

\noindent
We provide a supporting lemma that deals with a single juicy citrus. First, we give a notion that describes the structure of a solution:

\defBiStemmed*

\noindent
We also need the following notions. Let $G$ be a graph and let $D \subseteq V(G)$. Then $G \dagger D$ denotes the graph obtained from $G$ by identifying the vertices of $D$ to a single vertex~$d$. For a set $L$ of vertices in $G$, we use $L \dagger D$ to denote the set of vertices obtained from $L$ by replacing each occurrence of an element of $D$ by $d$. Similarly, for a set $T$ of pairs of vertices in $G$, we use $T \dagger D$ to denote the set of pairs obtained from $T$ by replacing each occurrence of an element of $D$ in a pair by $d$.

The following lemma is now a direct consequence of Lemma~\ref{l-tw}.

\begin{lemma} \label{l-lemon-support}
Let $G$ be an $n$-vertex graph and suppose $G$ contains a juicy citrus with ends $x,y$ and wedge set $\mathcal{L}$ such that $V(x,y,\mathcal{L}) = V(G)$. Let $T$ be a set of terminal pairs in $G$. Then we can compute, in polynomial time: 
\begin{enumerate}[(i)]
\item\label{l-lemon-support-i} a minimum Steiner forest for $(G \dagger \{x,y\},T \dagger \{x,y\})$;
\item\label{l-lemon-support-ii} a minimum Steiner forest for $(G,T)$ over all Steiner forests for which the $\ell$-lemon is intertwined;
\item\label{l-lemon-support-iii} a minimum Steiner forest for $(G,T)$.
\end{enumerate}
\end{lemma}
\begin{proof}
For~(\ref{l-lemon-support-i}), consider the graph $G' = G \cup \{xy\}$. Then $G'$ still contains a juicy citrus with ends $x,y$ and wedge set $\mathcal{L}$ such that $V(x,y,\mathcal{L}) = V(G')$. By Proposition~\ref{p-juicy-citrus}, $G'$ has treewidth at most~$2$. Note that $G \dagger \{x,y\}$ can be obtained from $G'$ by contracting the edge $xy$. Hence, $G \dagger \{x,y\}$ has treewidth at most~$2$ and we can apply Lemma~\ref{l-tw}.

For~(\ref{l-lemon-support-ii}), we add the pair $(x,y)$ to $T$. By Proposition~\ref{p-juicy-citrus}, we can apply Lemma~\ref{l-tw}. For~(\ref{l-lemon-support-iii}), we can apply Lemma~\ref{l-tw} again by Proposition~\ref{p-juicy-citrus}.
\end{proof}
Note that Lemma~\ref{l-lemon-support}(\ref{l-lemon-support-iii}) generalizes Lemma~\ref{l-2d2}. Effectively, we have shown that \textsc{Steiner Forest} remains polynomial if we are given a treewidth-$2$-deletion set of size~$2$ instead of just a $2$-deletion set of size~$2$.

We now extend Lemma~\ref{l-lemon-support} to $\ell$-lemons.
\lemLemonSupport*
\begin{proof}
\textsf{\textbf{Part~(\ref{l-l-lemon-support-i}).}}
For simplicity of notation, let $G' = G \dagger \{x,y\}$ and let $T' = T \dagger \{x,y\}$. Let $d$ denote the vertex to which $x$ and $y$ were identified. To start the algorithm, we apply the following separation rule. Intuitively, the separation rule allows us to treat the wedges of the $\ell$-lemon independently. Let $L,L' \in \mathcal{L}$. Suppose that there is a terminal pair $(s,t) \in T$ such that $s \in L \setminus \{x,y\}$ and $t \in L' \setminus\{x,y\}$. Replace the terminal pair $(s,t)$ in $T'$ by the terminal pairs $(s,d)$ and $(t,d)$. After applying the separation rule exhaustively, let $T''$ denote the resulting set of terminal pairs. Observe that, by construction, for each terminal pair $(s,t)\in T''$, there is some wedge $L \in \mathcal{L}$ such that $s,t \in (L \cup \{d\}) \setminus \{x,y\}$. For each $L \in \mathcal{L}$, let $T''_L \subseteq T''$ denote the set of terminal pairs in $T''$ for which both terminals belong to $(L \cup \{d\}) \setminus \{x,y\}$. Thus, the sets $T''_L$ partition $T''$. For each $L \in \mathcal{L}$, compute a minimum Steiner forest $\bar{F}_L$ for $(G[L] \dagger \{x,y\},T''_L)$. Output the union $\bar{F}$ of $\bar{F}_L$ over all wedges $L$.

We first analyze the running time of the algorithm. Note that $\bar{F}_L$ can be computed in polynomial time for each $L \in \mathcal{L}$. If $L$ is a juicy wedge, then as argued in Lemma~\ref{l-lemon-support}(\ref{l-lemon-support-i}), $G[L] \dagger \{x,y\}$ has treewidth~$2$ and we apply Lemma~\ref{l-tw}. If $L$ is a seeded wedge or $\ell$-pulped wedge, then we apply Lemma~\ref{l-exponential}, taking $2^{O(\ell)}$ time. Hence, the algorithm runs in $2^{O(\ell)}n^{O(1)}$ time.

It remains to prove correctness of the algorithm. We start by showing correctness of the separation rule. That is, $F$ is a minimum Steiner forest for $(G',T')$ if and only if $F$ is a minimum Steiner forest for $(G',T'')$. Let $F$ be any minimum Steiner forest for $(G',T')$. Let $L, L' \in \mathcal{L}$. Suppose that there is a terminal pair $(s,t) \in T$ such that $s \in L \setminus \{x,y\}$ and $t \in L' \setminus\{x,y\}$. By the definition of $F$, there is a path in $F$ between $s$ and $t$. By the definition of a wedge, any such path must contain $d$. Thus, $F$ contains a path between $s$ and $d$ and between $t$ and $d$.  Hence, $F$ remains a Steiner forest for the set of terminal pairs after we replace $(s,t)$ by $(s,d)$ and $(t,d)$. The converse is straightforward.

Next, we argue that $\bar{F}$ is indeed a Steiner forest for $(G',T')$. By the preceding, it suffices to argue this for $(G',T'')$. Let $(s,t) \in T''$ be a terminal pair. By construction of $T''$, $s,t \in L$ for some wedge $L \in \mathcal{L}$. Hence, $s,t$ are connected in $\bar{F}_L$ and thus in $\bar{F}$. Hence, $\bar{F}$ is a Steiner forest for $(G',T')$.

Finally, we argue that $\bar{F}$ is a minimum Steiner forest for $(G',T')$. By the preceding, it suffices to argue this for $(G',T'')$. Let $F$ be any minimum Steiner forest for $(G',T'')$. Let $(s,t) \in T''$ be a terminal pair. By construction of $T''$, $s,t \in L$ for some wedge $L \in \mathcal{L}$. Hence, the restriction of $F$ to $G[L] \dagger \{x,y\}$ is a Steiner forest for $(G[L] \dagger \{x,y\},T''_L)$. These restrictions partition the edge set of $F$, and thus $|E(\bar{F})| \leq |E(F)|$ by definition of $\bar{F}_L$. Therefore, $\bar{F}$ is a minimum Steiner forest for $(G',T')$.

\bigskip\noindent\textsf{\textbf{Part~(\ref{l-l-lemon-support-ii}).}}
We assume that $\mathcal{L} \not= \emptyset$; otherwise, the algorithm simply returns the edge $xy$. We first show that we can reduce, in polynomial time, to a polynomial number of instances of the situation in Part~(\ref{l-l-lemon-support-i}). To start the reduction, we apply the following separation rule. Intuitively, the rule allows us to treat the wedges of the $\ell$-lemon independently. Let $L, L' \in \mathcal{L}$. Suppose that there is a terminal pair $(s,t) \in T'$ such that $s \in L \setminus \{x,y\}$ and $t \in L' \setminus\{x,y\}$. Replace the terminal pair $(s,t)$ by the terminal pairs $(s,x)$ and $(t,x)$. After applying the separation rule exhaustively, let $T'$ denote the resulting set of terminal pairs. Observe that, by construction, for each terminal pair $(s,t)\in T'$, there is some wedge $L \in \mathcal{L}$ such that $s,t \in L$. We now branch on all wedges of $\mathcal{L}$. Let $L$ denote the wedge of the current branch. Let $T'_L \subseteq T'$ denote the set of terminal pairs in $T'$ for which both terminals belong to $L$ plus the terminal pair $(x,y)$, if it is not present already, and compute a minimum Steiner forest $\bar{F}_L$ for $(G[L], T'_L)$. Compute a minimum Steiner forest $\hat{F}_L$ for $((G \setminus (L \setminus\{x,y\})) \dagger \{x,y\}, (T' \setminus T'_L) \dagger \{x,y\})$. Let $\tilde{F}_L = \bar{F}_L \cup \hat{F}_L$; note that $\hat{F}_L$ was computed in a graph where $x,y$ are identified and thus we perform this union on the edges to ensure that $\tilde{F}_L$ is indeed a subgraph of $G$. After performing all branches, output the minimum set $\tilde{F}_L$ over all wedges $L$.

We first analyze the running time of the algorithm. There are $n$ branches. Consider a branch, corresponding to some $L \in \mathcal{L}$. Note that $\hat{F}_L$ can be computed in polynomial time by using the algorithm of Part~(\ref{l-l-lemon-support-i}). Moreover, $\bar{F}_L$ can be computed in polynomial time as well: if $L$ is a juicy wedge, then $G[L]$ has treewidth at most~$2$ and we apply Lemma~\ref{l-tw}; otherwise, $L$ is a seeded wedge or $\ell$-pulped wedge and we apply Lemma~\ref{l-exponential}, taking $2^{O(\ell)}$ time. Hence, the algorithm runs in $2^{O(\ell)}n^{O(1)}$ time.

It remains to prove correctness of the reduction. We start by showing correctness of the separation rule. That is, $F$ is a minimum Steiner forest for $(G,T)$ over all Steiner forests for which the $\ell$-lemon is intertwined if and only if $F$ a minimum Steiner forest for $(G,T')$ over all Steiner forests for which the $\ell$-lemon is intertwined. Let $F$ be any minimum Steiner forest for $(G,T)$ over all Steiner forests for which the $\ell$-lemon is intertwined. Let $L, L' \in \mathcal{L}$. Suppose that there is a terminal pair $(s,t) \in T'$ such that $s \in L \setminus \{x,y\}$ and $t \in L' \setminus\{x,y\}$. By the definition of $F$, there is a path $P$ in $F$ between $s$ and $t$. By the definition of a wedge, $P$ must contain $x$ or $y$. In the former case, clearly, $F$ contains a path between $s$ and $x$ and between $t$ and $x$. Else, by the definition of $F$, there is a path in $F$ between $x$ and $y$. The concatenation of $P$ with this path yields that $F$ contains a path between $s$ and $x$ and between $t$ and $x$. Hence, $F$ remains a Steiner forest for the set of terminal pairs after we replace $(s,t)$ by $(s,x)$ and $(t,x)$ and $F$ is still intertwined with respect to the $\ell$-lemon. The converse is straightforward.

Next, we argue that the output is indeed a Steiner forest for $(G,T)$ that is intertwined with respect to the $\ell$-lemon. To this end, by the correctness of the separation rule, it suffices to show that for any wedge $L \in \mathcal{L}$, $\tilde{F}_L$ is a Steiner forest for $(G,T')$. Note that, by construction of $T'_L$, $\bar{F}_L$ contains a path $P$ between $x$ and $y$. Moreover, by the construction of $T'$, both terminals of each pair in $T'$ belong to the same wedge of $G$. Let $(s,t) \in T'$ be a terminal pair. If $s,t \in L$, then $s,t$ are connected in $\bar{F}_L$ and thus in $\tilde{F}_L$. If $s,t \in L'$ for some wedge $L' \not= L$, then there is a path in $\hat{F}_L$ between $s$ and $t$, and using $P$ if necessary, this can be extended to a path in $\tilde{F}_L$ between $s$ and $t$. Hence, $\tilde{F}_L$ is a Steiner forest for $(G,T')$. By the existence of $P$, $\tilde{F}_L$ is intertwined with respect to the $\ell$-lemon. 

Finally, we show that the output is a minimum Steiner forest for $(G,T)$ over all Steiner forests for which the $\ell$-lemon is intertwined. Let $F$ be any such forest. Let $P$ be a shortest path in $F$ between $x$ and $y$. Then $P$ contains only vertices of a single wedge $L \in \mathcal{L}$. Observe now that the restriction of $F$ to $G[L]$ is a Steiner forest $\bar{F}_L$ for $(G[L], T'_L)$ and the restriction of $F$ to $(G \setminus (L \setminus\{x,y\})) \dagger \{x,y\}$ is a Steiner forest for $((G \setminus (L \setminus\{x,y\})) \dagger \{x,y\}, (T \setminus T'_L) \dagger \{x,y\})$. These forests jointly partition the edge set of $F$. Hence, $|E(\tilde{F}_L)| \leq |E(F)|$ by definition of $\bar{F}_L$ and $\hat{F}_L$. Therefore, the output is indeed a minimum Steiner forest for $(G,T)$ over all Steiner forests for which the $\ell$-lemon is intertwined.

\bigskip\noindent\textsf{\textbf{Part~(\ref{l-l-lemon-support-iii}).}}
First, compute a minimum Steiner forest $\tilde{F}$ that is intertwined with respect to the $\ell$-lemon using Part~(\ref{l-l-lemon-support-ii}). Next, we seek to find a minimum Steiner forest that is bi-stemmed with respect to the $\ell$-lemon. We first apply Lemma~\ref{l-wedge-transform} to turn all seeded wedges into juicy wedges. By abuse of notation, we still denote the resulting graph by $G$ and the resulting $\ell$-lemon by $\mathcal{L}$. By inspection of the proof of Lemma~\ref{l-wedge-transform}, we note that still $V(x,y,\mathcal{L}) = V(G)$.

Next, we seek the intersection of some minimum Steiner forest with the graph induced by all $\ell$-pulped wedges in $\mathcal{L}$, by branching on all possibilities. Let $\mathcal{L}'$ denote the set of $\ell$-pulped wedges in $\mathcal{L}$; thus, $\mathcal{L}' \subseteq \mathcal{L}$ and $|\mathcal{L}'| \leq \ell$. Let $C = V(x,y,\mathcal{L}') \setminus \{x,y\}$; thus $|C| \leq \ell (\ell -2)$. Branch on all possible assignments of each vertex in $C$ to $x$, $y$, or neither; this assignment will correspond to whether the vertex will be in a component of the solution that contains $x$, $y$, or neither. There are $O(3^{|C|})=O(3^{\ell (\ell-2)})$ branches. For each branch $b$, let $X_b$ denote the set of vertices in $C$ assigned to $x$, $Y_b$ those assigned to $y$, and $Z_b$ the remaining vertices in $C$.

We first focus on $X_b$ and $Y_b$. We skip the branch if $G[X_b \cup \{x\}]$ or $G[Y_b \cup \{y\}]$ is not connected, or if there is a terminal pair $(s,t) \in T$ such that either $s \in X_b$ and $t \in C \setminus X_b$, or $s \in Y_b$ and $t \in C \setminus Y_b$. This is safe, as we look for a Steiner forest that is bi-stemmed for the $\ell$-lemon. Compute a spanning tree $S_{X_b}$ of $G[X_b \cup \{x\}]$ and a spanning tree $S_{Y_b}$ of $G[Y_b \cup \{y\}]$. Let $T'_b \subseteq T$ be the set of terminal pairs $(s,t)$ in $T$ for which $s, t \not\in Z_b$. Let $G'_b = (G-C) \cup S_{X_b} \cup S_{Y_b}$. Note that this graph has treewidth at most~$2$, because $G-C$ is a juicy citrus and thus has treewidth~$2$ by Proposition~\ref{p-juicy-citrus}. We compute a minimum Steiner forest $\bar{F}^{XY}_b$ of $(G'_b, T'_b)$. This takes polynomial time by application of Lemma~\ref{l-tw}.

We now focus on $Z_b$. We skip the branch if there is a terminal pair $(s,t) \in T$ for which $s \in Z_b$ and $t \not\in Z_b$. If we did not skip the branch, it holds for each $(s,t)\in T \setminus T'_b$ that $s,t \in Z_b$. Consider some $L \in \mathcal{L}'$ and some connected component $D$ of $G[L \cap Z_b]$. We skip the branch if there is a terminal pair $(s,t) \in T \setminus T'_b$ for which $s \in D$ and $t \not\in D$. Let $T_D \subseteq T \setminus T'_b$ be the set of terminal pairs $(s,t) \in T \setminus T'_b$ with $s,t \in D$. Compute a minimum Steiner forest $\bar{F}^D_b$ for $(G[D], T_D)$. Since $L$ is an $\ell$-pulped wedge, we apply Lemma~\ref{l-exponential}, taking $2^{O(l)}$ time.

Let $\bar{F}_b$ denote the union of $\bar{F}^{XY}_b$ and all forests $\bar{F}^D_b$. Verify that $\bar{F}_b$ is indeed a Steiner forest for $(G,T)$; otherwise, set $\bar{F}_b$ to any spanning tree of $G$, making it a trivial Steiner forest for $(G,T)$. Output the minimum of $\tilde{F}$ and $\bar{F}_b$ over all branches $b$.

The algorithm clearly runs in $2^{O(\ell^2)}n^{O(1)}$ time. It always outputs a Steiner forest for $(G,T)$. Hence, it remains to prove optimality of the algorithm. 

Let $F$ be any minimum Steiner forest for $(G,T)$. If $F$ is intertwined with respect to the $\ell$-lemon, then $\hat{F}$ (and the output) will be a minimum Steiner forest for $(G,T)$. Hence, we assume that $F$ is bi-stemmed with respect to the $\ell$-lemon. Let $X_F$ denote the set of vertices in $C$ that are connected in $F$ to $x$ and let $Y_F$ denote the set of vertices in $C$ that are connected in $F$ to $y$. Let $Z_F = C \setminus (X_F \cup Y_F)$ be the remaining vertices. Note that there is a branch $b^*$ that corresponds to $X_F,Y_F,Z_F$. 

Consider the first part of the algorithm that focuses on $X_F,Y_F$. We claim that the branch $b^*$ is not skipped. Indeed, by definition $G[X_F \cup \{x\}]$ and $G[Y_F \cup \{y\}]$ are connected. If there is a terminal pair $(s,t) \in T$ such that $s \in X_F$ and $t \in C$, then since there is no path between $x$ and $y$ in $F$, $t \in X_F$. Similarly, for a terminal pair $(s,t) \in T$ such that $s \in Y_F$ and $t \in C$, it holds that $t \in Y_F$. Thus, the branch is not skipped. For any $(s,t) \in T'_b$, $s,t \in V(G) \setminus Z_F$ by construction. By definition, $F$ connects all vertices in $X_F$ to $x$. As $F$ is bi-stemmed with respect to the $\ell$-lemon, any path in $F$ between a vertex in $X_F$ and $x$ does not contain $y$. Hence, by minimality, the vertices in $X_F \cup \{x\}$ induce in $F$ a spanning tree of $G[X_F \cup \{x\}]$. Similarly, the vertices in $Y_F \cup \{y\}$ induce in $F$ a spanning tree of $G[Y_F \cup \{y\}]$. Hence, without loss of generality, we may assume from now on that $S_{X_F}$ and $S_{Y_F}$ are a subgraph of $F$. Therefore, $F - Z_F$ is a Steiner forest for $(G'_b,T'_b)$. Since $\bar{F}_b$ is a minimum such Steiner forest, it follows that $|E(\bar{F}_b)| \leq |E(F-Z)|$.

Next, consider the second part of the algorithm that focuses on $Z_F$. The branch $b^*$ is not skipped, because if there is a terminal pair $(s,t) \in T$ for which $s \in Z_F$, then since $s$ is not connected to $x$ or $y$, it holds that $t \in Z_F$. Let $L \in \mathcal{L}'$ and $D$ some connected component of $G[L \cap Z_F]$. Again, the branch $b^*$ is not skipped, because if $s \in D$ for some $(s,t) \in T \setminus T'_b$, then there is a path from $s$ to $t$ that is fully contained in $D$, and thus $t \in D$. By definition of $Z_F$, no vertex of $D$ is connected in $F$ to $x$ or $y$. Hence, a subset of the trees in $F$ is contained in $G[D]$ and this subset forms a Steiner forest $F^D$ of $(G[D], T_D)$. Therefore, $|E(\hat{F}^D_b)| \leq |E(F^D)|$. As the sets $T_D$ partition $T \setminus T'_{b^*}$, $|E(\bar{F}_{b^*})| \leq |E(F)|$ and the algorithm indeed outputs a minimum Steiner forest for $(G,T)$.
\end{proof}

\subsection{Bushes of Citruses}
We now combine several citruses.
\defCitrusBush*

\noindent
We extend the definition by the notion of a \emph{juicy citrus bush}, which is a citrus bush that consists of only juicy citruses.
It is important to note that some of $C_1,\ldots,C_{|E(G')|}$ may be empty, particularly if $xy \in E(G)$ for some $x,y \in X$. Moreover, a lemon bush is trivially an $\ell$-lemon bush for any integer $\ell \geq 5$.

\thmLemonTWtwo*
\begin{proof}
Let $T$ be a set of terminal pairs in $G$. We first apply Lemma~\ref{l-wedge-transform}; let $\tilde{G}$ be the resulting subgraph of $G$. By the proof of Lemma~\ref{l-wedge-transform}, $\tilde{G}$ is a juicy citrus bush flowering from a graph $\tilde{G}'$ of treewidth at most~$2$. Consider any tree decomposition of $\tilde{G}'$ of minimum width. For each edge of $\tilde{G}'$, we mark one unique bag that contains both endpoints $x,y$ of the edge. Let $L$ be the juicy wedge corresponding to $xy$; in particular, it has ends $x,y$. As in Proposition~\ref{p-juicy-citrus}, there is a tree decomposition of $G_{L} \cup \{xy\}$ of width~$2$ with a special bag containing exactly $x$ and $y$. We attach this special bag to the marked bag. Then we obtain a tree decomposition of $\tilde{G}$ of width~$2$ and we can apply Lemma~\ref{l-tw}.
\end{proof}
Note that the algorithm of Theorem~\ref{t-lemon-tw2} does not need to know $G'$. It suffices to apply Lemma~\ref{l-wedge-transform}, verify that the resulting graph has treewidth~$2$ in polynomial time~\cite{VTL82}, and then apply Lemma~\ref{l-tw}.

We can prove slightly weaker versions of Theorem~\ref{t-lemon-tw2} for $\ell$-lemon bushes for any fixed integer $\ell$ in case it flowers from a path or cycle.

\thmPathbush*
\begin{proof}
Let $G'$ be a path. Let $G$ be an $\ell$-lemon bush flowering from $G'$ and let $T$ be a set of terminal pairs in $G$. We can apply Lemma~\ref{l-2con} to reduce to the case of a single $\ell$-lemon. For this $\ell$-lemon, we can apply Lemma~\ref{l-l-lemon-support}(\ref{l-l-lemon-support-iii}).
\end{proof}
We immediately generalize to the case that $G'$ is a cycle. For the sake of clarity, in the following theorem statement, by a cycle we mean a cycle of length at least~$3$. The case of a degenerate cycle corresponds to the case that the graph is a single $\ell$-lemon, covered by Lemma~\ref{l-l-lemon-support}(\ref{l-l-lemon-support-iii}).

\thmCyclebush*
\begin{proof}
Let $G'$ be a cycle. Let $G$ be an $\ell$-lemon bush flowering from $G'$ and let $T$ be a set of terminal pairs. Let $F$ be a minimum Steiner forest for $(G,T)$. Observe that $F$ does not contain a cycle, and thus at least one $\ell$-lemon of $G$ must be bi-stemmed. In our search for a minimum Steiner forest for $(G,T)$, we consider three cases, each considering a different number of bi-stemmed $\ell$-lemons.

\medskip
\noindent{\textsf{\textbf{Case 1:}}} \emph{A single bi-stemmed $\ell$-lemon.}
\par\smallskip\noindent
The intuition of our approach is to guess the bi-stemmed $\ell$-lemon and, after projecting terminal pairs appropriately, solve for this $\ell$-lemon and the remaining ones separately. Crucially, if our guess is correct, there is a path in the solution between the ends of the guessed $\ell$-lemon that avoids the guessed $\ell$-lemon (except the ends). We make use of this while projecting the terminal pairs and ensure this path indeed exists in our computed solution.

We branch over all edges in $G'$. Let $xy$ be the edge of $G'$ corresponding to the current branch. Let $\mathcal{L}_{xy}$ denote the $\ell$-lemon corresponding to $xy$ and let $C_{xy} = V(x,y,\mathcal{L})$. We create two new sets $T_{xy}$ and $T'_{xy}$ of terminals as follows. Let $(s,t) \in T$. If $s,t \in C_{xy}$, then add $(s,t)$ to $T_{xy}$. If $s,t \not\in C_{xy}$, then add $(s,t)$ to $T'_{xy}$. If $s \in C_{xy}$ and $t \not\in C_{xy}$, then add $(s,x)$ to $T_{xy}$ and $(t,x)$ to $T'_{xy}$. Finally, add $(x,y)$ to $T'_{xy}$. Compute a minimum Steiner forest $\hat{F}_{xy}$ for $(G[C_{xy}] \dagger \{x,y\}, T_{xy} \dagger \{x,y\})$ and compute a minimum Steiner forest $\hat{F}'_{xy}$ for $(G-(C_{xy} \setminus\{x,y\}) - \{xy\}, T'_{xy})$. This can be done in $2^{O(\ell^2)} n^{O(1)}$ time: $G[C_{xy}]$ is a single $\ell$-lemon and thus we can apply Lemma~\ref{l-l-lemon-support}(\ref{l-l-lemon-support-i}); $G-(C_{xy} \setminus\{x,y\}) - \{xy\}$ is an $\ell$-lemon bush flowering from a path and thus we can apply Theorem~\ref{t-pathbush}. Output the minimum $\hat{F}_{xy} \cup \hat{F}'_{xy}$ over all edges $xy$ of $G'$; we note that this union is performed over the edges, as there is a clear mapping from the edges of $G[C_{xy}] \dagger \{x,y\}$ to those of $G[C_{xy}]$.

The algorithm clearly runs in $2^{O(\ell^2)} n^{O(1)}$ time. Next, we argue that the algorithm outputs a Steiner forest for $(G,T)$. Consider any branch and let $xy$ be the edge of $G'$ of the branch. Let $(s,t) \in T$. If $s,t\not\in C_{xy}$, then $s$ and $t$ are connected by $\hat{F}'_{xy}$. If $s,t \in C_{xy}$, then $s$ and $t$ are connected by $\hat{F}_{xy}$. If $s \in C_{xy}$ and $t \not\in C_{xy}$, then there is a path between $t$ and $x$ in $\hat{F}'_{xy}$, as $(t,x) \in T'_{xy}$. There is also a path between $s$ and $x$ or $y$, as $(s,x) \in T_{xy}$; note that the path may indeed go to $y$, as $\hat{F}_{xy}$ considers $T_{xy} \dagger \{x,y\}$. There is also a path between $x$ and $y$ in $\hat{F}'_{xy}$, as $(x,y) \in T'_{xy}$. Hence, there is a path between $s$ and $t$ in $\hat{F}_{xy} \cup \hat{F}'_{xy}$. By the construction of $G[C_{xy}] \dagger \{x,y\}$, $\hat{F}_{xy}$ is bi-stemmed. Therefore, $\hat{F}_{xy} \cup \hat{F}'_{xy}$ is a Steiner forest for $(G,T)$.

Finally, we prove optimality. Let $F$ be a minimum Steiner forest for $(G,T)$ and suppose that $F$ is bi-stemmed with respect to exactly one $\ell$-lemon. Let this $\ell$-lemon correspond to the edge $xy$ of $G'$. Let $F_{xy}$ be the subgraph of $F$ in $G[C_{xy}]$ and let $F'_{xy}$ be the subgraph of $F$ in $G-(C_{xy} \setminus\{x,y\}) - \{xy\}$. We focus again on the corresponding edge sets. Note that $E(F_{xy})$ and $E(F'_{xy})$ partition $E(F)$. Moreover, as $F$ is only bi-stemmed with respect to the $\ell$-lemon corresponding to $xy$, $F$ contains a path $P$ between $x$ and $y$ that contains no internal vertex from $C_{xy}$ (and not the edge $xy$ if it exists). Thus, $P$ is in $F'_{xy}$.

We now first show that (the edge set of) $F_{xy}$ is a Steiner forest for $(G[C_{xy}] \dagger \{x,y\}, T_{xy} \dagger \{x,y\})$. To this end, we consider the two ways in which a terminal pair ended up in $T_{xy}$. First, let $(s,t) \in T$ such that $s,t\in C_{xy}$. Then $(s,t) \in T_{xy}$. Moreover, $F$ contains a path between $s$ and $t$ that either is fully contained in $G[C_{xy}]$ or goes from $s$ to $x$, via $P$ to $y$, and finally to $t$. Hence, $F_{xy}$ contains a path in $G[C_{xy}] \dagger \{x,y\}$ between $s$ and $t$. Second, let $(s,t) \in T$ such that $s \in C_{xy}$ and $t \not\in C_{xy}$. Then $(s,x) \in T_{xy}$. Moreover, either $F$ contains a path in $G[C_{xy}]$ between $s$ and $x$ or a path in $G[C_{xy}]$ between $s$ and $y$. In both cases, $F_{xy}$ contains a path in $G[C_{xy}] \dagger \{x,y\}$ between $s$ and $x$. Hence, $F_{xy}$ is indeed a Steiner forest for $(G[C_{xy}] \dagger \{x,y\}, T_{xy} \dagger \{x,y\})$. Therefore, $|E(\hat{F}_{xy})| \leq |E(F_{xy})|$.

We next show that $F'_{xy}$ is a Steiner forest for $(G-(C_{xy} \setminus\{x,y\}) - \{xy\}, T'_{xy})$. For simplicity, let $G'_{xy} = G-(C_{xy} \setminus\{x,y\}) - \{xy\}$. We consider the three ways in which a terminal pair ended up in $T'_{xy}$. First, let $(s,t) \in T$ such that $s,t \not\in C_{xy}$. Then $(s,t) \in T'_{xy}$. Moreover, since $F$ is bi-stemmed with respect to the $\ell$-lemon corresponding to $xy$ and $\{x,y\}$ is a cutset, $F'_{xy}$ contains a path between $s$ and $t$. Second, let $(s,t) \in T$ such that $s \in C_{xy}$ and $t \not\in C_{xy}$. Then $(t,x) \in T'_{xy}$. Moreover, either $F$ contains a path in $G'_{xy}$ between $t$ and $x$ or a path in $G'_{xy}$ between $t$ and $y$; in the latter case, the addition of $P$ shows that in fact there is also a path in $G'_{xy}$ between $t$ and $x$. Finally, the pair $(x,y)$ is satisfied by $P$. Hence, $F'_{xy}$ is a Steiner forest for $(G'_{xy}, T'_{xy})$. Therefore, $|E(\hat{F}'_{xy})| \leq |E(F'_{xy})|$.

In conclusion, $|E(\hat{F}_{xy})| + |E(\hat{F}'_{xy})| \leq |E(F_{xy})| + |E(F'_{xy})| = |E(F)|$. Hence, the Steiner forest returned in this case is a minimum Steiner forest for $(G,T)$.

\medskip
\noindent{\textsf{\textbf{Case 2:}}} \emph{Two bi-stemmed $\ell$-lemons.}
\par\smallskip\noindent
The intuition of our approach is to guess the bi-stemmed $\ell$-lemons and, after projecting terminal pairs appropriately, solve for these $\ell$-lemons and the remaining ones separately. Crucially, if our guess is correct, there are two paths in the solution between the ends of the guessed $\ell$-lemons that avoid the guessed $\ell$-lemons (except the ends) and each end of each $\ell$-lemon is the end of exactly one of these paths. We make use of this while projecting the terminal pairs and ensure these paths indeed exist in our computed solution.

We branch over all pairs of distinct edges in $G'$. Let $(e_1,e_2)$ be the pair of edges of $G'$ corresponding to the current branch. Let $e_1=x_1y_1$ and $e_2=x_2y_2$; without loss of generality, we assume that going around the cycle $G'$ (in an arbitrary direction that we define to be clockwise) we encounter $x_1,y_1,x_2,y_2$. For $i=1,2$, let $E'^i_{e_1e_2}$ denote the (possibly empty) set of edges of $G'$ that we encounter between $y_i$ and $x_{3-i}$ by traversing the cycle $G'$ clockwise. For $i=1,2$, let $\mathcal{L}_i$ be the $\ell$-lemon corresponding to $e_i$ and let $C^{i}_{e_1e_2} = V(x_i,y_i,\mathcal{L}_i)$  (note that $x_i,y_i \in C^{i}_{e_1e_2}$). We create three new sets $T_{e_1e_2}$, $T^1_{e_1e_2}$, and $T^2_{e_1e_2}$ of terminals as follows. Let $(s,t) \in T$. If $s$ is in an $\ell$-lemon corresponding to an edge in $E'^1_{e_1e_2}$ and $t$ is in an $\ell$-lemon corresponding to an edge in $E'^2_{e_1e_2}$, then we skip the branch. Now let $i=1,2$. There are five cases (which collapse to four):
\begin{itemize}
\item If $s,t \in C^{i}_{e_1e_2}$, or if $s \in C^i_{e_1e_2}$ and $t \in C^{3-i}_{e_1e_2}$, then add $(s,t)$ to $T_{e_1e_2}$. (These two cases have the same outcome.)
\item If $s \in C^i_{e_1e_2}$ and $t$ is in an $\ell$-lemon corresponding to an edge in $E'^i_{e_1e_2}$, then add $(s,y_i)$ to $T_{e_1e_2}$ and $(y_i,t)$ to $T^i_{e_1e_2}$. 
\item If $s \in C^i_{e_1e_2}$ and $t$ is in an $\ell$-lemon corresponding to an edge in $E'^{3-i}_{e_1e_2}$, then add $(s,x_i)$ to $T_{e_1e_2}$ and $(x_i,t)$ to $T^{3-i}_{e_1e_2}$. 
\item If $s,t$ are both in an $\ell$-lemon corresponding to an edge in $E^i_{e_1e_2}$ (note that these may be two different $\ell$-lemons), then add $(s,t)$ to $T^i_{e_1e_2}$.
\end{itemize}
Finally, add $(y_1,x_2)$ to $T^1_{e_1e_2}$ and $(y_2,x_1)$ to $T^2_{e_1e_2}$.

Next, we compute three Steiner forests:
\begin{itemize}
\item Let $G^\dagger_{e_1e_2}=(G[C^1_{e_1e_2}] \cup G[C^2_{e_1e_2}])\dagger \{y_1,x_2\} \dagger \{y_2,x_1\}$, meaning we both identify $y_1$ and $x_2$ and identify $y_2$ and $x_1$. Let $T^\dagger_{e_1e_2} = T_{e_1e_2} \dagger \{y_1,x_2\} \dagger \{y_2,x_1\}$. $\hat{F}^\dagger_{e_1e_2}$ is a minimum Steiner forest for $(G^\dagger_{e_1e_2}, T^\dagger_{e_1e_2})$;
\item Let $G^1_{e_1e_2}$ denote the subgraph of $G$ corresponding to the $\ell$-lemons corresponding to the edges in $E'^1_{e_1e_2}$. $\hat{F}^1_{e_1e_2}$ is a minimum Steiner forest for $(G^1_{e_1e_2},T^1_{e_1e_2})$;
\item Let $G^2_{e_1e_2}$ denote the subgraph of $G$ corresponding to the $\ell$-lemons corresponding to the edges in $E'^2_{e_1e_2}$. $\hat{F}^2_{e_1e_2}$ is a minimum Steiner forest for $(G^2_{e_1e_2},T^2_{e_1e_2})$.
\end{itemize}
The first can be computed in $2^{O(\ell^2)} n^{O(1)}$ time by Lemma~\ref{l-l-lemon-support}(\ref{l-l-lemon-support-iii}), as $G'_{e_1e_2}$ is a $2\ell$-lemon. The latter two graphs, $G^1_{e_1e_2}$ and $G^2_{e_1e_2}$, form an $\ell$-lemon bush flowering from a path, and thus the requested forests can be computed in $2^{O(\ell^2)} n^{O(1)}$ time by Theorem~\ref{t-pathbush}. Output the minimum $\hat{F}_{e_1e_2} = \hat{F}^\dagger_{e_1e_2} \cup \hat{F}^1_{e_1e_2} \cup \hat{F}^2_{e_1e_2}$ over all pairs $(e_1,e_2)$ of edges of $G'$; we note that this union is performed over the edges, as there is a clear mapping from the edges of $G^\dagger_{e_1e_2}$ to those of $G[C^1_{e_1e_2}]$ and $G[C^2_{e_1e_2}]$.

The algorithm clearly runs in $2^{O(\ell^2)} n^{O(1)}$ time. Next, we argue that, for any branch (corresponding to a pair of edges $(e_1,e_2)$) that is not skipped, the algorithm outputs a Steiner forest for $(G,T)$. Let $x$ respectively $y$ denote the vertex to which $y_1$ and $x_2$ respectively $y_2$ and $x_1$ were identified in $G^\dagger_{e_1e_2}$. Since $(y_1,x_2) \in T^1_{e_1e_2}$ respectively $(y_2,x_1) \in T^2_{e_1e_2}$, there exist paths $\hat{P}^1_{e_1e_2}$ respectively $\hat{P}^2_{e_1e_2}$ in $\hat{F}^1_{e_1e_2}$ respectively $\hat{F}^2_{e_1e_2}$, and thus both paths are in $\hat{F}_{e_1e_2}$.
Let $(s,t) \in T$ and let $i=1,2$. If $s$ is in an $\ell$-lemon corresponding to an edge in $E'^1_{e_1e_2}$ and $t$ is in an $\ell$-lemon corresponding to an edge in $E'^2_{e_1e_2}$, then we have skipped the branch. Let $i=1,2$. We consider the five remaining cases of the location of $s$ and $t$:
\begin{itemize}
\item If $s,t \in C^{i}_{e_1e_2}$, then as $(s,t) \in T_{e_1e_2}$, there is a path $Q$ in $\hat{F}^\dagger_{e_1e_2}$ between $s$ and $t$. If $Q$ contains only edges of $G[C^{i}_{e_1e_2}]$, then $Q$ also exists in $\hat{F}_{e_1e_2}$. If $Q$ contains edges of $G[C^{3-i}_{e_1e_2}]$, then $\hat{F}^\dagger_{e_1e_2}$ is intertwined with respect to the $\ell$-lemon corresponding to $e_{3-i}$. Adding $\hat{P}^1_{e_1e_2}$ and $\hat{P}^2_{e_1e_2}$ to $Q$ then yields a path in $\hat{F}_{e_1e_2}$ between $s$ and $t$.
\item If $s \in C^i_{e_1e_2}$ and $t \in C^{3-i}_{e_1e_2}$, then as $(s,t) \in T_{e_1e_2}$, there is a path $Q$ in $\hat{F}^\dagger_{e_1e_2}$ between $s$ and $t$. We can assume that $Q$ starts in $s$, goes via edges of $G[C^i_{e_1e_2}]$, then via edges of $G[C^{3-i}_{e_1e_2}]$, and reaches $t$. In particular, the edges of $G[C^i_{e_1e_2}]$ and $G[C^{3-i}_{e_1e_2}]$ do not appear mixed in $Q$. Hence, $Q$ reaches $x$ or $y$ during the transition from $G[C^i_{e_1e_2}]$ to $G[C^{3-i}_{e_1e_2}]$. In these cases, adding $\hat{P}^1_{e_1e_2}$ respectively $\hat{P}^2_{e_1e_2}$ to $Q$ yields a path in $\hat{F}_{e_1e_2}$ between $s$ and $t$.
\item If $s \in C^i_{e_1e_2}$ and $t$ is in an $\ell$-lemon corresponding to an edge in $E'^i_{e_1e_2}$, then as $(s,y_i) \in T_{e_1e_2}$, there is a path $Q$ in $\hat{F}^\dagger_{e_1e_2}$ between $s$ and (the vertex corresponding to) $y_i$. If $Q$ contains only edges of $G[C^{i}_{e_1e_2}]$, then $Q$ also exists in $\hat{F}_{e_1e_2}$. If $Q$ contains edges of $G[C^{3-i}_{e_1e_2}]$, then $\hat{F}^\dagger_{e_1e_2}$ is intertwined with respect to the $\ell$-lemon corresponding to $e_{3-i}$; adding $\hat{P}^1_{e_1e_2}$ and $\hat{P}^2_{e_1e_2}$ to $Q$ then yields a path in $\hat{F}_{e_1e_2}$ between $s$ and $y_i$. Since $(y_i,t) \in T^i_{e_1e_2}$, there is a path in $\hat{F}^i_{e_1e_2}$ between $y_i$ and $t$. Hence, there is a path in $\hat{F}_{e_1e_2}$ between $s$ and $t$.
\item If $s \in C^i_{e_1e_2}$ and $t$ is in an $\ell$-lemon corresponding to an edge in $E'^{3-i}_{e_1e_2}$, then the argument is similar as above.
\item If $s,t$ are both in an $\ell$-lemon corresponding to an edge in $E'^i_{e_1e_2}$, then there is a path in $\hat{F}^i_{e_1e_2}$, and thus in $\hat{F}_{e_1e_2}$, between $s$ and $t$.
\end{itemize}
By the construction of $G^\dagger_{e_1e_2}$, $\hat{F}^\dagger_{e_1e_2}$ may be intertwined with respect to the $\ell$-lemon corresponding to $e_1$, but then it will not be with respect to the $\ell$-lemon corresponding to $e_2$ and vice versa. Hence, $\hat{F}_{e_1e_2}$ is a forest. It follows that $\hat{F}_{e_1e_2}$ is a Steiner forest for $(G,T)$.

Finally, we prove optimality. Let $F$ be a minimum Steiner forest for $(G,T)$ and suppose that $F$ is bi-stemmed with respect to exactly two $\ell$-lemons. Let these $\ell$-lemons correspond to the edges $e_1,e_2$ of $G'$. As a first step, we note that the pair $(e_1,e_2)$ is not skipped. Indeed, suppose there is a pair $(s,t) \in T$ such that $s$ is in an $\ell$-lemon corresponding to an edge in $E'^1_{e_1e_2}$ and $t$ is in an $\ell$-lemon corresponding to an edge in $E'^2_{e_1e_2}$. As $F$ is bi-stemmed with respect to both the $\ell$-lemon corresponding to $e_1$ and $e_2$, there is no path in $F$ between $x_1$ and $y_1$ and no path between $x_2$ and $y_2$. Hence, there can be no path in $F$ between $s$ and $t$, a contradiction.

Let $F^\dagger_{e_1e_2}$ be the subgraph of $F$ in $G[C^1_{e_1e_2}] \cup G[C^2_{e_1e_2}]$, let $F^1_{e_1e_2}$ be the subgraph of $F$ in $G^1_{e_1e_2}$, and let $F^2_{e_1e_2}$ be the subgraph of $F$ in $G^2_{e_1e_2}$. We focus again on the corresponding edge sets, particularly for $F^\dagger_{e_1e_2}$. Note that $E(F^\dagger_{e_1e_2})$, $E(F^1_{e_1e_2})$, $E(F^2_{e_1e_2})$ partition $E(F)$. Moreover, as $F$ is bi-stemmed only with respect to the $\ell$-lemons corresponding to $e_1$ and $e_2$, $F$ contains a path $P^1_{e_1e_2}$ between $y_1$ and $x_2$ that contains no edges from $G[C^1_{e_1e_2}] \cup G[C^2_{e_1e_2}]$ and a path $P^2_{e_1e_2}$ between $y_2$ and $x_1$ that contains no edges from $G[C^1_{e_1e_2}] \cup G[C^2_{e_1e_2}]$. Thus, $P^1_{e_1e_2}$ is in $F^1_{e_1e_2}$ and $P^2_{e_1e_2}$ is in $F^2_{e_1e_2}$.

We now show that $F^\dagger_{e_1e_2}$ is a Steiner forest for $(G^\dagger_{e_1e_2},T^\dagger_{e_1e_2})$. To this end, we consider the four ways in which a terminal pair ended up in $T^\dagger_{e_1e_2}$. Let $(s,t) \in T$ and let $i=1,2$. 
\begin{itemize}
\item Suppose that $s,t \in C^i_{e_1e_2}$. Then $(s,t) \in T^\dagger_{e_1e_2}$. Since $F$ is bi-stemmed with respect to $e_{3-i}$, there is no path in $F$ between $x_{3-i}$ and $y_{3-i}$. Hence, there is a path in $F$ between $s$ and $t$ that is restricted to $G[C^i_{e_1e_2}]$ and thus is contained in $F^\dagger_{e_1e_2}$. 
\item Suppose that $s \in C^i_{e_1e_2}$ and $t \in C^{3-i}_{e_1e_2}$. Then $(s,t) \in T^\dagger_{e_1e_2}$. Since $F$ is bi-stemmed with respect to $e_1$ and $e_2$, there is no path in $F$ between $x_1$ and $y_1$ and between $x_2$ and $y_2$. Hence, either there is a path in $F$ between $s$ and $t$ that goes from $s$ to $y_i$, via $P^i_{e_1e_2}$ to $x_{3-i}$, to $t$; or there is a path in $F$ between $s$ and $t$ that goes from $s$ to $x_i$, via $P^{3-i}_{e_1e_2}$ to $y_{3-i}$, to $t$. Since $y_i$ and $x_{3-i}$ as well as $x_i$ and $y_{3-i}$ are identified in $G^\dagger_{e_1e_2}$, there is a path in $F^\dagger_{e_1e_2}$ between $s$ and $t$.
\item Suppose that $s \in C^i_{e_1e_2}$ and $t$ is in an $\ell$-lemon corresponding to an edge in $E'^i_{e_1e_2}$. Then $(s,y_i) \in T^\dagger_{e_1e_2}$. Since $F$ is bi-stemmed with respect to $e_i$ and $e_{3-i}$, there is no path in $F$ between $x_i$ and $y_i$ and between $x_{3-i}$ and $y_{3-i}$. Hence, the path in $F$ between $s$ and $t$ must go through $y_i$, and there is a path in $F^\dagger_{e_1e_2}$ between $s$ and $y_i$.
\item Suppose that $s \in C^i_{e_1e_2}$ and $t$ is in an $\ell$-lemon corresponding to an edge in $E'^{3-i}_{e_1e_2}$. Similar as before, we can show there is a path in $F^\dagger_{e_1e_2}$ between $s$ and $x_i$.
\end{itemize}
Hence, $F^\dagger_{e_1e_2}$ is indeed a Steiner forest for $(G^\dagger_{e_1e_2},T^\dagger_{e_1e_2})$. Therefore, $|E(\hat{F}^\dagger_{e_1e_2})| \leq |E(F^\dagger_{e_1e_2})|$.

We next show, for $i=1,2$, that $F^i_{e_1e_2}$ is a Steiner forest for $(G^i_{e_1e_2},T^i_{e_1e_2})$. To this end, we consider the three ways in which a terminal pair ended up in $T^i_{e_1e_2}$. Let $(s,t) \in T$.
\begin{itemize}
\item If $s \in C^i_{e_1e_2}$ and $t$ is in an $\ell$-lemon corresponding to an edge in $E^i_{e_1e_2}$, then $(y_i,t) \in T^i_{e_1e_2}$. Since $F$ is bi-stemmed with respect to $e_1$ and $e_{2}$, there is no path in $F$ between $x_1$ and $y_1$ and between $x_{2}$ and $y_{2}$. Hence, the path in $F$ between $s$ and $t$ must go through $y_i$, and there is a path in $F^i_{e_1e_2}$ between $y_i$ and $t$.
\item If $s \in C^{3-i}_{e_1e_2}$ and $t$ is in an $\ell$-lemon corresponding to an edge in $E^i_{e_1e_2}$, then a similar argument as before shows there is a path in $F^{i}_{e_1e_2}$ between $x_i$ and $t$.
\item If $s$ and $t$ are both in an $\ell$-lemon corresponding to an edge in $E^i_{e_1e_2}$, then $(s,t) \in T^i_{e_1e_2}$. Since $F$ is bi-stemmed with respect to $e_1$ and $e_{2}$, there is no path in $F$ between $x_1$ and $y_1$ and between $x_{2}$ and $y_{2}$. Hence, the path in $F$ between $s$ and $t$ must be in $G^i_{e_1e_2}$, and thus there is a path in $F^i_{e_1e_2}$ between $s$ and $t$.
\end{itemize}
The pair $(y_i,x_{3-i})$ is satisfied by $P^i_{e_1e_2}$. Hence, $F^i_{e_1e_2}$ is indeed a Steiner forest for $(G^i_{e_1e_2},T^i_{e_1e_2})$. Therefore, $|E(\hat{F}^i_{e_1e_2})| \leq |E(F^i_{e_1e_2})|$.

In conclusion, $|E(\hat{F}_{e_1e_2})| + |E(\hat{F}^1_{e_1e_2})| + |E(\hat{F}^2_{e_1e_2})| \leq |E(F_{e_1e_2})| + |E(F^1_{e_1e_2})| + |E(F^1_{e_1e_2})| = |E(F)|$. Hence, the Steiner forest returned in this case is a minimum Steiner forest for $(G,T)$.

\medskip
\noindent{\textsf{\textbf{Case 3:}}} \emph{More than two bi-stemmed $\ell$-lemons.}
\par\smallskip\noindent
The intuition of our approach is roughly the same as in the case of exactly two bi-stemmed $\ell$-lemons. We guess two bi-stemmed $\ell$-lemons that are closest on the cycle. After projecting terminal pairs appropriately, solve for these $\ell$-lemons and the remaining ones separately. Crucially, if our guess is correct, there is a path in the solution between an end of one guessed lemon and an end of the other that avoids the guessed $\ell$-lemons (except the ends), but in contrast to the previous case, this does not hold for the other pair of ends. We make use of this while projecting the terminal pairs and ensure this path indeed exist in our computed solution. 

We branch over all pairs of distinct edges in $G'$. Let $(e_1,e_2)$ be the pair of edges of $G'$ corresponding to the current branch. Let $e_1=x_1y_1$ and $e_2=x_2y_2$; without loss of generality, we assume that going around the cycle $G'$ (in an arbitrary direction that we define to be clockwise) we encounter $x_1,y_1,x_2,y_2$. For $i=1,2$, let $E'^i_{e_1e_2}$ denote the (possibly empty) set of edges of $G'$ that we encounter between $y_i$ and $x_{3-i}$ by traversing the cycle $G'$ clockwise. For $i=1,2$, let $\mathcal{L}_i$ be the $\ell$-lemon corresponding to $e_i$ and let $C^{i}_{e_1e_2} = V(x_i,y_i,\mathcal{L}_i)$  (note that $x_i,y_i \in C^{i}_{e_1e_2}$). We create three new sets $T_{e_1e_2}$, $T^1_{e_1e_2}$, and $T^2_{e_1e_2}$ of terminals as follows. Let $(s,t) \in T$. If $s$ is in an $\ell$-lemon corresponding to an edge in $E'^1_{e_1e_2}$ and $t$ is in an $\ell$-lemon corresponding to an edge in $E'^2_{e_1e_2}$, then we skip the branch. Now let $i=1,2$. There are five cases (which collapse to four outcomes):
\begin{itemize}
\item If $s,t \in C^{i}_{e_1e_2}$, or if $s \in C^i_{e_1e_2}$ and $t \in C^{3-i}_{e_1e_2}$, then add $(s,t)$ to $T_{e_1e_2}$. (These two cases have the same outcome.)
\item If $s \in C^i_{e_1e_2}$ and $t$ is in an $\ell$-lemon corresponding to an edge in $E'^i_{e_1e_2}$, then add $(s,y_i)$ to $T_{e_1e_2}$ and $(y_i,t)$ to $T^i_{e_1e_2}$. 
\item If $s \in C^i_{e_1e_2}$ and $t$ is in an $\ell$-lemon corresponding to an edge in $E'^{3-i}_{e_1e_2}$, then add $(s,x_i)$ to $T_{e_1e_2}$ and $(x_i,t)$ to $T^{3-i}_{e_1e_2}$. 
\item If $s,t$ are both in an $\ell$-lemon corresponding to an edge in $E^i_{e_1e_2}$ (note that these may be two different $\ell$-lemons), then add $(s,t)$ to $T^i_{e_1e_2}$.
\end{itemize}
Finally, add $(y_1,x_2)$ to $T^1_{e_1e_2}$.

Next, we compute three Steiner forests:
\begin{itemize}
\item Let $G^\dagger_{e_1e_2}=(G[C^1_{e_1e_2}] \cup G[C^2_{e_1e_2}])\dagger \{y_1,x_2\}$. Let $T^\dagger_{e_1e_2} = T_{e_1e_2} \dagger \{y_1,x_2\}$. $\hat{F}^\dagger_{e_1e_2}$ is a minimum Steiner forest for $(G^\dagger_{e_1e_2}, T^\dagger_{e_1e_2})$;
\item Let $G^1_{e_1e_2}$ denote the subgraph of $G$ corresponding to the $\ell$-lemons corresponding to the edges in $E'^1_{e_1e_2}$. $\hat{F}^1_{e_1e_2}$ is a minimum Steiner forest for $(G^1_{e_1e_2},T^1_{e_1e_2})$;
\item Let $G^2_{e_1e_2}$ denote the subgraph of $G$ corresponding to the $\ell$-lemons corresponding to the edges in $E'^2_{e_1e_2}$. $\hat{F}^2_{e_1e_2}$ is a minimum Steiner forest for $(G^2_{e_1e_2},T^2_{e_1e_2})$.
\end{itemize}
In all three cases, we obtain an $\ell$-lemon bush flowering from a path, and thus the requested forests can be computed in $2^{O(\ell^2)} n^{O(1)}$ time by Theorem~\ref{t-pathbush}. Let $\hat{F}_{e_1e_2}$ be the subgraph obtained from $\hat{F}^\dagger_{e_1e_2} \cup \hat{F}^1_{e_1e_2} \cup \hat{F}^2_{e_1e_2}$ (we note that this union is performed over the edges, as there is a clear mapping from the edges of $G^\dagger_{e_1e_2}$ to those of $G[C^1_{e_1e_2}]$ and $G[C^2_{e_1e_2}]$) by taking a spanning tree of each of its components. Output the minimum $\hat{F}_{e_1e_2}$ over all pairs $(e_1,e_2)$ of edges of $G'$.

The algorithm clearly runs in $2^{O(\ell^2)} n^{O(1)}$ time. Next, we argue that, for any branch (corresponding to a pair of edges $(e_1,e_2)$) that is not skipped, the algorithm outputs a Steiner forest for $(G,T)$. Let $z$ denote the vertex to which $y_1$ and $x_2$ were identified in $G^\dagger_{e_1e_2}$. Since $(y_1,x_2) \in T^1_{e_1e_2}$, there exists a path $\hat{P}_{e_1e_2}$ in $\hat{F}^1_{e_1e_2}$, and in $\hat{F}_{e_1e_2}$, between $y_1$ and $x_2$.
Let $(s,t) \in T$ and let $i=1,2$. If $s$ is in an $\ell$-lemon corresponding to an edge in $E'^1_{e_1e_2}$ and $t$ is in an $\ell$-lemon corresponding to an edge in $E'^2_{e_1e_2}$, then we have skipped the branch. Let $i=1,2$. We consider the five remaining cases of the location of $s$ and $t$:
\begin{itemize}
\item If $s,t \in C^{i}_{e_1e_2}$, then as $(s,t) \in T_{e_1e_2}$, there is a path $Q$ in $\hat{F}^\dagger_{e_1e_2}$, and in $\hat{F}_{e_1e_2}$, between $s$ and $t$.
\item If $s \in C^i_{e_1e_2}$ and $t \in C^{3-i}_{e_1e_2}$, then as $(s,t) \in T_{e_1e_2}$, there is a path $Q$ in $\hat{F}^\dagger_{e_1e_2}$ between $s$ and $t$. We can assume that $Q$ starts in $s$, goes via edges of $G[C^i_{e_1e_2}]$, then via edges of $G[C^{3-i}_{e_1e_2}]$, and reaches $t$. In particular, the edges of $G[C^i_{e_1e_2}]$ and $G[C^{3-i}_{e_1e_2}]$ do not appear mixed in $Q$. Hence, $Q$ reaches $z$ during the transition from $G[C^i_{e_1e_2}]$ to $G[C^{3-i}_{e_1e_2}]$. Adding $\hat{P}_{e_1e_2}$ to $Q$ in the place of $z$ yields a path in $\hat{F}_{e_1e_2}$ between $s$ and $t$.
\item If $s \in C^i_{e_1e_2}$ and $t$ is in an $\ell$-lemon corresponding to an edge in $E'^i_{e_1e_2}$, then as $(s,y_i) \in T_{e_1e_2}$, there is a path $Q$ in $\hat{F}^\dagger_{e_1e_2}$, and in $\hat{F}_{e_1e_2}$, between $s$ and (the vertex corresponding to) $y_i$. Since $(y_i,t) \in T^i_{e_1e_2}$, there is a path in $\hat{F}^i_{e_1e_2}$ between $y_i$ and $t$. Hence, there is a path in $\hat{F}_{e_1e_2}$ between $s$ and $t$.
\item If $s \in C^i_{e_1e_2}$ and $t$ is in an $\ell$-lemon corresponding to an edge in $E'^{3-i}_{e_1e_2}$, then the argument is similar as above.
\item If $s,t$ are both in an $\ell$-lemon corresponding to an edge in $E'^i_{e_1e_2}$, then there is a path in $\hat{F}^i_{e_1e_2}$, and thus in $\hat{F}_{e_1e_2}$, between $s$ and $t$.
\end{itemize}
It follows that $\hat{F}_{e_1e_2}$ is a Steiner forest for $(G,T)$.

Finally, we prove optimality. Let $F$ be a minimum Steiner forest for $(G,T)$ and suppose that $F$ is bi-stemmed with respect to more than two $\ell$-lemons. Consider two such $\ell$-lemons whose corresponding edges $e_1,e_2$ are closest on the cycle $G'$. We assume the indexing is such that there is no edge of $G'$ that is clockwise from $e_1$ to $e_2$ and corresponds to an $\ell$-lemon for which $F$ is bi-stemmed.

As a first step, we note that the pair $(e_1,e_2)$ is not skipped. Indeed, suppose there is a pair $(s,t) \in T$ such that $s$ is in an $\ell$-lemon corresponding to an edge in $E'^1_{e_1e_2}$ and $t$ is in an $\ell$-lemon corresponding to an edge in $E'^2_{e_1e_2}$. As $F$ is bi-stemmed with respect to both the $\ell$-lemon corresponding to $e_1$ and $e_2$, there is no path in $F$ between $x_1$ and $y_1$ and no path between $x_2$ and $y_2$. Hence, there can be no path in $F$ between $s$ and $t$, a contradiction.

Let $F^\dagger_{e_1e_2}$ be the subgraph of $F$ in $G[C^1_{e_1e_2}] \cup G[C^2_{e_1e_2}]$, let $F^1_{e_1e_2}$ be the subgraph of $F$ in $G^1_{e_1e_2}$, and let $F^2_{e_1e_2}$ be the subgraph of $F$ in $G^2_{e_1e_2}$. We focus again on the corresponding edge sets, particularly for $F^\dagger_{e_1e_2}$. Note that $E(F^\dagger_{e_1e_2})$, $E(F^1_{e_1e_2})$, $E(F^2_{e_1e_2})$ partition $E(F)$. Moreover, as $F$ is bi-stemmed with respect to the $\ell$-lemons corresponding to $e_1$ and $e_2$, and $(e_1,e_2)$ are a closest such pair, $F$ contains a path $P_{e_1e_2}$ between $y_1$ and $x_2$ that contains no edges from $G[C^1_{e_1e_2}] \cup G[C^2_{e_1e_2}]$. Thus, $P_{e_1e_2}$ is in $F^1_{e_1e_2}$.

We now show that $F^\dagger_{e_1e_2}$ is a Steiner forest for $(G^\dagger_{e_1e_2},T^\dagger_{e_1e_2})$. To this end, we consider the four ways in which a terminal pair ended up in $T^\dagger_{e_1e_2}$. Let $(s,t) \in T$ and let $i=1,2$. 
\begin{itemize}
\item Suppose that $s,t \in C^i_{e_1e_2}$. Then $(s,t) \in T^\dagger_{e_1e_2}$. Since $F$ is bi-stemmed with respect to $e_{3-i}$, there is no path in $F$ between $x_{3-i}$ and $y_{3-i}$. Hence, there is a path in $F$ between $s$ and $t$ that is restricted to $G[C^i_{e_1e_2}]$ and thus is contained in $F^\dagger_{e_1e_2}$. 
\item Suppose that $s \in C^i_{e_1e_2}$ and $t \in C^{3-i}_{e_1e_2}$. Assume $i=1$; we swap the roles of $s$ and $t$ otherwise. Then $(s,t) \in T^\dagger_{e_1e_2}$. Since $F$ is bi-stemmed with respect to $e_1$ and $e_2$, there is no path in $F$ between $x_1$ and $y_1$ and between $x_2$ and $y_2$. Moreover, as $F$ is bi-stemmed with respect the $\ell$-lemon corresponding to an edge of $G'$ that is not $e_1$ or $e_2$, and this edge does not appear clockwise from $e_1$ to $e_2$ by assumption, there is no path between $y_2$ and $x_1$ in $F$. Hence, there is a path in $F$ between $s$ and $t$ that goes from $s$ to $y_1$, via $P_{e_1e_2}$ to $x_{2}$, to $t$. Since $y_1$ and $x_{2}$ are identified in $G^\dagger_{e_1e_2}$, there is a path in $F^\dagger_{e_1e_2}$ between $s$ and $t$.
\item Suppose that $s \in C^i_{e_1e_2}$ and $t$ is in an $\ell$-lemon corresponding to an edge in $E'^i_{e_1e_2}$. Then $(s,y_i) \in T^\dagger_{e_1e_2}$. Since $F$ is bi-stemmed with respect to $e_i$ and $e_{3-i}$, there is no path in $F$ between $x_i$ and $y_i$ and between $x_{3-i}$ and $y_{3-i}$. Hence, the path in $F$ between $s$ and $t$ must go through $y_i$, and there is a path in $F^\dagger_{e_1e_2}$ between $s$ and $y_i$.
\item Suppose that $s \in C^i_{e_1e_2}$ and $t$ is in an $\ell$-lemon corresponding to an edge in $E'^{3-i}_{e_1e_2}$. Similar as before, we can show there is a path in $F^\dagger_{e_1e_2}$ between $s$ and $x_i$.
\end{itemize}
Hence, $F^\dagger_{e_1e_2}$ is indeed a Steiner forest for $(G^\dagger_{e_1e_2},T^\dagger_{e_1e_2})$. Therefore, $|E(\hat{F}^\dagger_{e_1e_2})| \leq |E(F^\dagger_{e_1e_2})|$.

We next show, for $i=1,2$, that $F^i_{e_1e_2}$ is a Steiner forest for $(G^i_{e_1e_2},T^i_{e_1e_2})$. To this end, we consider the three ways in which a terminal pair ended up in $T^i_{e_1e_2}$. Let $(s,t) \in T$.
\begin{itemize}
\item If $s \in C^i_{e_1e_2}$ and $t$ is in an $\ell$-lemon corresponding to an edge in $E^i_{e_1e_2}$, then $(y_i,t) \in T^i_{e_1e_2}$. Since $F$ is bi-stemmed with respect to $e_1$ and $e_{2}$, there is no path in $F$ between $x_1$ and $y_1$ and between $x_{2}$ and $y_{2}$. Hence, the path in $F$ between $s$ and $t$ must go through $y_i$, and there is a path in $F^i_{e_1e_2}$ between $y_i$ and $t$.
\item If $s \in C^{3-i}_{e_1e_2}$ and $t$ is in an $\ell$-lemon corresponding to an edge in $E^i_{e_1e_2}$, then a similar argument as before shows there is a path in $F^{i}_{e_1e_2}$ between $x_i$ and $t$.
\item If $s$ and $t$ are both in an $\ell$-lemon corresponding to an edge in $E^i_{e_1e_2}$, then $(s,t) \in T^i_{e_1e_2}$. Since $F$ is bi-stemmed with respect to $e_1$ and $e_{2}$, there is no path in $F$ between $x_1$ and $y_1$ and between $x_{2}$ and $y_{2}$. Hence, the path in $F$ between $s$ and $t$ must be in $G^i_{e_1e_2}$, and thus there is a path in $F^i_{e_1e_2}$ between $s$ and $t$.
\end{itemize}
Specifically for $i=1$, the pair $(y_1,x_2)$ is satisfied by $P_{e_1e_2}$. 
Hence, $F^i_{e_1e_2}$ is indeed a Steiner forest for $(G^i_{e_1e_2},T^i_{e_1e_2})$. Therefore, $|E(\hat{F}^i_{e_1e_2})| \leq |E(F^i_{e_1e_2})|$.

In conclusion, $|E(\hat{F}_{e_1e_2})| + |E(\hat{F}^1_{e_1e_2})| + |E(\hat{F}^2_{e_1e_2})| \leq |E(F_{e_1e_2})| + |E(F^1_{e_1e_2})| + |E(F^1_{e_1e_2})| = |E(F)|$. Hence, the Steiner forest returned in this case is a minimum Steiner forest for $(G,T)$.

\medskip\noindent
We now conclude the proof. We run each of the three described algorithms, and return the smallest Steiner forest for $(G,T)$ found. The analysis of the three cases shows that this yields a minimum Steiner forest for $(G,T)$. The running time of the algorithm is $2^{O(\ell^2)} n^{O(1)}$: in each case, there are $O(n^2)$ branches, following by applications of Theorem~\ref{t-pathbush} and Lemma~\ref{l-l-lemon-support}.
\end{proof}

\noindent
It is important to note that Theorem~\ref{t-pathbush} and Theorem~\ref{t-cyclebush} assume that the stem set, and thus the graph $G'$, is given. This will be the case for all applications of these two theorems in this paper.

\subsection{Entangled Bushes of Citruses}

Next, we consider citrus bushes with stem set of bounded size. In fact, our algorithm holds for a more general case.

\defEntangledBush*

\noindent
We extend the definition by the notion of an \emph{entangled juicy citrus bush}, which is an entangled citrus bush that consists of only juicy citruses.
Observe that if, for an entangled citrus bush with stem set $X$, every citrus has an empty wedge set, then $X$ forms a vertex cover of $G$. Moreover, if some vertex of $Y$ is adjacent to exactly two vertices of $X$, then it forms a juicy wedge and can alternatively be seen as part of the citrus bush.

\begin{lemma} \label{l-find-l-lemon-tangle}
For any integers $\ell \geq 5, k$, we can decide in $n^{O(k)}$ time whether a graph $G$ is an entangled $\ell$-lemon bush with a stem set of size at most~$k$, and if so, return an associated stem set and tangle set.
\end{lemma}
\begin{proof}
Branch on all sets $X \subseteq V(G)$ of size at most $k$. Consider some branch with set $X$. Let $Y$ be the set of vertices $y$ in $G$ for which $N(y) \subseteq X$ and either $|N(y)| = 1$ or $|N(y)| \geq 3$. If there exists a connected component of $G-(X \cup Y)$ that is adjacent to at least three vertices in $X$ or exactly one vertex of $X$, then we reject the branch. For any $u,v \in X$, suppose $G-(X \cup Y)$ has at least one connected component with neighbours $u,v$. Consider the connected components $D_1,\ldots,D_p$ for some $p \geq 1$ in $G-(X \cup Y)$ with neighbours $u,v \in X$. Suppose that $D_1,\ldots,D_q$ for some $q \leq p$ have more than two vertices and more than treewidth~$2$. We can check this in polynomial time~\cite{VTL82}. Verify that $q \leq \ell$ and $D_1,\ldots,D_q$ each have at most $\ell-2$ vertices; otherwise, we reject the branch. If we did not reject the branch, then $G$ is an entangled $\ell$-lemon bush with stem set $X$ and tangle set $Y$. It is also clear that if $G$ is an entangled $\ell$-lemon bush with stem set $\tilde{X}$ and tangle set $\tilde{Y}$, then the branch where $\tilde{X}$ is considered is not rejected, even though the set $Y$ constructed in that branch may not fully correspond to $\tilde{Y}$ (in particular, if $\tilde{Y}$ contains vertices of degree~$2$).
\end{proof}

\thmSolvetangle*
\begin{proof}
We start by defining some notation. Let $X$ be a stem set of $G$ and $Y$ a tangle set; note that we can find $X,Y$ in $n^{O(k)}$ time by Lemma~\ref{l-find-l-lemon-tangle}. Let $G' = G-Y$. Since $G'$ is an $\ell$-lemon bush, let $G''$ be the graph such that $G'$ flowers from $G''$ with stem set $X$. Note that $G''$ can be easily found in polynomial time from $G'$ by inspection of the adjacency of the components of $G'-X$ and the edges of $G[X]$.

The intuition of the proof is to enumerate all possible patterns of a solution and find a minimum solution for each possible pattern. The pattern describes which vertices of $X$ are in the same component of a solution and which $\ell$-lemons must have an intertwined solution. At the same time, the pattern will prescribe exactly which vertices in $Y$ will connect the different components in $X$.

Let $T$ be a set of terminal pairs in $G$. If there is a vertex $x \in X$ for which there exists a pair $(x,t) \in T$ for some $t \in V(G)$, add a vertex $x'$ to $G$ (and to $Y$) and replace $(x,t)$ by $(x',t)$. Note that this operation is safe and increases the size of any solution by exactly~$1$. By performing this operation exhaustively, we ensure that no vertex of $X$ is in a terminal pair.

We define the following auxiliary graph. Let $H$ be the graph on vertex set $X \cup Y$ and the following edges: for any $u,v \in X \cup Y$, $uv \in E(H)$ if and only if $uv \in E(G)$ or $uv \in E(G'')$. Now branch on all subgraphs (\emph{patterns}) $P$ of $H$ such that $P$ is a forest, $X \subseteq V(P)$, and every vertex of $Y$ has degree at least~$2$ in $P$. For any such $P$, we say that a Steiner forest $F$ of $(G,T)$ \emph{adheres} to $P$ if the following hold:
\begin{itemize}
\item if $uv \in E(P) \cap E(G'')$ (note that $u,v \in X$), then $F$ is intertwined with respect to the $\ell$-lemon corresponding to $uv$;
\item if $uv \in E(G'') \setminus E(P)$ (note that $u,v \in X$), then $F$ may be bi-stemmed with respect to the $\ell$-lemon corresponding to $uv$;
\item if $y \in Y \cap V(P)$, then $N_P(y) = N_F(y)$ (note that $N_P(y),N_F(y) \subseteq X$);
\item if $y \in Y \setminus V(P)$, then $y$ has degree at most~$1$ in $F$;
\end{itemize}
We note that the above conditions do not imply that $X \subseteq V(F)$, even though $X \subseteq V(P)$. We only impose the latter condition to simplify some discussions below.

We claim that for any minimum Steiner forest $F$ for $(G,T)$, there exists a pattern $P_F$ such that $F$ adheres to $P_F$. Indeed, we can define $P_F$ as follows. Add all $x \in X$ to $P_F$. For any $y$ for which $|N_F(y)| > 1$, add $y$ to $P_F$ and add the edges from $y$ to $N_F(y)$ to $P_F$. For any $uv \in E(G'')$, add $uv$ to $P_F$ if $F$ is intertwined with respect to the $\ell$-lemon corresponding to $uv$. This subgraph $P_F$ of $H$ satisfies all conditions of a pattern. To argue this, it suffices to show that $P_F$ is a forest; the other conditions are trivial. It is straightforward to observe that if $P_F$ contains a cycle, then there is a sequence of intertwined $\ell$-lemons and edges between $X$ and $Y$ that jointly induce a cycle in $F$, a contradiction to the minimality of $F$.

Now fix a pattern $P$. We aim to find a minimum Steiner forest for $(G,T)$ that adheres to $P$. For any $y \in Y \cap V(P)$, remove any incident edge not in $E(P)$. This is safe, as any such edge will not be part of the solution we seek in this branch. 
Consider the vertex sets of the components of $P$. These yield a partition of $X$. Any element of this partition is called a \emph{strong subset} of $X$ (with respect to $P$). 
Note that the vertices of a strong subset do not necessarily induced a connected subgraph of $G''$; also, a strong subset may be a singleton, as $X \subseteq V(P)$.
We call an $\ell$-lemon \emph{strong} if its corresponding edge in $E(G'')$ has both endpoints in a single strong subset (we may then say that the $\ell$-lemon is strong for that subset). We call an $\ell$-lemon \emph{delicate} if it is not strong. Observe that for any delicate $\ell$-lemon, the corresponding edge $uv \in E(G'')$ is not in $P$, but the converse is not necessarily true.

We now provide some rejection rules for this branch. In the following, it may help the intuition of the reader to imagine that each strong subset has been identified to form a single vertex (although we do not explicitly do so). Suppose there is a school $S$ such that one of the following holds:
\begin{enumerate}
\item\label{kr-tangle-1} $S \subseteq Y$ and there is not a strong subset $A$ such that $S \subseteq N(A)$;
\item\label{kr-tangle-2} there are two strong subsets $A,B$ such that there is a strong $\ell$-lemon for $A$ and a strong $\ell$-lemon for $B$ that both contain a terminal in $S$;
\item\label{kr-tangle-3} there are two delicate $\ell$-lemons with ends in four distinct strong subsets for which both $\ell$-lemons contains a terminal in $S$;
\item\label{kr-tangle-4} there are three delicate $\ell$-lemons $\mathcal{L}_1, \mathcal{L}_2, \mathcal{L}_3$ and three strong subsets $A_1,A_2,A_3$, such that each of the three $\ell$-lemons contains a terminal in $S$ and $\mathcal{L}_i$ has its ends in $A_{i-1}$ and $A_{i}$ for $i=1,2,3$, where $A_0=A_3$;
\item\label{kr-tangle-5} there is a strong subset $A$ with a strong $\ell$-lemon for $A$ and there is a delicate $\ell$-lemon whose ends are not in $A$, such that both $\ell$-lemons contain a terminal in $S$.
\end{enumerate}
In each case, we reject the branch. It can be readily verified in each case, any Steiner forest for $(G,T)$ would not adhere to $P$, because two strong subsets will have to be connected for which this is not prescribed by $P$. We can apply the rejection rules in polynomial time.

If we did not reject until this point, we may conclude that for all schools $S$ either:
\begin{itemize}
\item $S \subseteq Y$ and there is a strong subset $A$ such that $S \subseteq N(A)$;
\item $S \not\subseteq Y$ and there is a strong subset $A$ such that any terminal in $S$ that is in an $\ell$-lemon is in a strong $\ell$-lemon for $A$ or in a delicate $\ell$-lemon for which one of its ends is in $A$.
\end{itemize}
Note that if there is, for example, a single delicate $\ell$-lemon that contains a terminal from a particular school, then the strong subset $A$ in the second point may not be unique. Also, the strong subset in the second point may not be unique.

We now provide three reduction rules. Let $S$ be any school such that $(Y \cap S) \setminus V(P) \not=\emptyset$ and $S \not\subseteq Y$. Let $y \in (Y \cap S) \setminus V(P)$ and let $A$ be a strong subset such that any terminal in $S$ that is in an $\ell$-lemon is in a strong $\ell$-lemon for $A$ or in a delicate $\ell$-lemon for which one of its ends is in $A$. Then the rules are:
\begin{enumerate}
\item\label{rr-tangle-1} If there is a strong $\ell$-lemon for $A$ that contains a terminal in $S$, then remove all edges incident to $y$ except a single edge incident to a vertex in $A$; if no such edge exists, then reject the branch. 
\item\label{rr-tangle-2} If there are two distinct strong subsets $B,C$, both distinct from $A$, such that some terminal in $S$ is in a delicate $\ell$-lemon with ends in $A,B$ and some terminal in $S$ is in a delicate $\ell$-lemon with ends in $A,C$, then remove all edges incident to $y$ except a single edge incident to a vertex in $A$; if no such edge exists, then reject the branch.
\item\label{rr-tangle-3} If there is a strong subset $B \not= A$ such that some terminal in $S$ is in a delicate $\ell$-lemon with ends in $A$ and $B$, then remove all edges incident to $y$ except a single edge (if one exists) incident to a vertex in $A$ and a single edge (if one exists) incident to a vertex in $B$; if both of these types of edges do not exist, then reject the branch.
\end{enumerate}
We execute the reduction rules in order: first Rule~1 until there are no schools and vertices $y$ to which it applies, then Rule~2 and Rule~3 under the same condition. Since any vertex is part of a single school, it is not important to check for a possible second application of one of the rules for the same vertex in $Y$. It is clear that the rules are safe when finding a Steiner forest for $(G,T)$ that adheres to $P$. We can apply the rules in polynomial time.

After the reduction rules, we note that for every $y \in Y \setminus V(P)$, at least one of the following is true: $y$ is not a terminal, the school that contains $y$ is a subset of $Y$, or $y$ has degree at most~$2$. Each vertex of $y \in Y \setminus V(P)$ that has degree~$2$ is removed from $Y$ and treated as a juicy wedge in some $\ell$-lemon with ends in $X$ (part of an existing or new $\ell$-lemon). If some such vertex is removed, we restart the branch (starting with the rejection rules) with the new set $Y$ and possibly new edge set for $G''$. Since $|Y|$ is reduced, this happens at most $n$ times. We note that now, for every $y \in Y \setminus V(P)$, at least one of the following is true: $y$ is not a terminal, the school that contains $y$ is a subset of $Y$, or $y$ has degree~$1$.

We now describe how to find a minimum Steiner forest $\hat{F}_P$ that adheres to $P$. To start, for a strong subset $A$, let $Y(A)$ denote the set of vertices $y \in Y$ for which there is a school $S$ that contains $y$ and contains a terminal in a strong or delicate $\ell$-lemon with an end in $A$. Any such $y$ has degree~$1$ by the preceding, and thus $Y(A) \cap Y(B) = \emptyset$ for any two strong subsets $A,B$. Now:
\begin{itemize}
\item for every $y \in Y \cap V(P)$, add $y$ and all edges incident to $y$ in $P$ to $\hat{F}_P$;
\item for every terminal $y \in Y \setminus V(P)$: 
\begin{itemize}
\item if $y$ belongs to a school $S$ such that $S \subseteq Y$, then there is a strong subset $A$ such that $S \subseteq N(A)$. As there may be multiple such strong subsets, fix a strong subset $A(S)$ for the school. We add $y$ and any edge from $y$ to $A$ to $\hat{F}_P$;
\item if $y$ has degree~$1$, then we add $y$ and its incident edge to $\hat{F}_P$;
\end{itemize}
\item for every $uv \in E(G'') \cap E(P)$, let $\mathcal{L}_{uv}$ denote the corresponding $\ell$-lemon. Let $T_{uv}$ be the set of pairs of vertices formed as follows: for every $(s,t) \in T$, if $s,t \in V(u,v,\mathcal{L}_{uv})$, then add $(s,t)$ to $T_{uv}$; if $s \in V(u,v,\mathcal{L}_{uv})$, and $t \not\in V(u,v,\mathcal{L}_{uv})$, then add $(s,u)$ to $T_{uv}$. Now find a minimum Steiner forest for $(G[V(u,v,\mathcal{L}_{uv})], T_{uv})$ for which the $\ell$-lemon is intertwined and add its vertices and edges to $\hat{F}_P$;
\item for every $uv \in E(G'') \setminus E(P)$ for which the corresponding $\ell$-lemon $\mathcal{L}_{uv}$ is strong, let $T_{uv}$ be the set of pairs of vertices formed as follows: for every $(s,t) \in T$, if $s,t \in V(u,v,\mathcal{L}_{uv})$, then add $(s,t)$ to $T_{uv}$; if $s \in V(u,v,\mathcal{L}_{uv})$, and $t \not\in V(u,v,\mathcal{L}_{uv})$, then add $(s,u)$ to $T_{uv}$. Now find a minimum Steiner forest for $(G[V(u,v,\mathcal{L}_{uv})] \dagger \{u,v\}, T_{uv}\dagger \{u,v\})$ and add its vertices and edges to $\hat{F}_P$. Note that since there is a clear mapping between the edges of $G[V(u,v,\mathcal{L}_{uv})] \dagger \{u,v\}$ and of $G[V(u,v,\mathcal{L}_{uv})]$, this is well defined;
\item for every two strong subsets $A,B$, we simultaneously treat all $uv \in E(G'') \setminus E(P)$ for which the corresponding $\ell$-lemon $\mathcal{L}_{uv}$ is delicate and $u \in A$, $v \in B$. There are at most $k^2$ such lemons. We create a graph $G_{AB}$ by adding each of these $\ell$-lemons separately, then identifying the ends of the $\ell$-lemons in $A$ to a vertex $a$ and identifying the ends in $B$ to a vertex $b$, and removing the edge $ab$ if it exists. Hence, $G_{AB}$ contains a $k^2\ell$-lemon with wedge set $\mathcal{L}_{ab}$ and ends $a,b$ such that $V(a,b,\mathcal{L}_{ab}) = V(G_{AB})$, where $\mathcal{L}_{ab}$ is the union of the wedge sets of the considered $\ell$-lemons. 
Let $T_{AB}$ be the set of pairs of vertices formed as follows: for every $(s,t) \in T$, if $s,t \in V(a,b,\mathcal{L}_{ab})$, then add $(s,t)$ to $T_{AB}$; if $s \in V(a,b,\mathcal{L}_{ab})$, and $t$ is in $Y(A)$ or in a strong or delicate $\ell$-lemon with an end in $A$ and the other end not in $B$, then add $(s,a)$ to $T_{AB}$; if $s \in V(a,b,\mathcal{L}_{ab})$, and $t$ is in $Y(B)$ or in a strong or delicate $\ell$-lemon with an end in $B$ and the other end not in $A$, then add $(s,b)$ to $T_{AB}$. Now find a minimum Steiner forest for $(G_{AB}, T_{AB})$ and add its vertices and edges to $\hat{F}_P$. Note that since there is a clear mapping between the edges of $G_{AB}$ and of the different $\ell$-lemons, this is well defined.
\end{itemize}
By a precise inspection of the running time analysis of Lemma~\ref{l-l-lemon-support}, we note that $\hat{F}_P$ can be computed in $2^{O(k^2\ell^2)}n^{O(1)}$ time by application of Lemma~\ref{l-l-lemon-support}. Finally, we output the minimum over all $P$ of the subgraph of $\hat{F}_P$ obtained by taking an arbitrary spanning tree of each component, that form a Steiner forest for $(G,T)$.

We first analyse the running time of the algorithm. Since $|X| = k$ and any pattern is a forest where the vertices of $Y$ have degree at least~$2$, any pattern contains at most $k-1$ vertices and edges from $Y$. Note that there are most $|Y| \leq n$ vertices and $|X|\cdot |Y| \leq nk$ edges to choose from. Hence, they can be enumerated in $(n(k+1))^{O(k)}$ time. The applications of Lemma~\ref{l-l-lemon-support} take $2^{O(k^2\ell^2)}n^{O(1)}$ time. Hence, the total running time is $2^{O(k^2\ell^2)} \cdot (n(k+1))^{O(k)}$.

We now claim that the output of the algorithm is a minimum Steiner forest for $(G,T)$. Let $F$ be a minimum Steiner forest and let $P_F$ denote the corresponding pattern described earlier. We will show that $\hat{F}_{P_F}$ is a minimum Steiner forest for $(G,T)$, by slowly transforming $F$ into $\hat{F}_{P_F}$, without increasing its size.

As we showed earlier, the five rejection rules are safe and the three reduction rules are safe. Hence, we may assume that the branch with $P_F$ is not rejected. For Reduction Rule~\ref{rr-tangle-1} and~\ref{rr-tangle-2}, we note that $F$ must contain an edge from $y$ to a vertex in $A$ and that these edges are equivalent. Hence, we may replace the current edge in $F$ incident to $y$ and replace it by the edge that remained after application of the rule. The resulting subgraph is still a minimum Steiner forest for $(G,T)$, and by abuse of notation, we still call it $F$. For the sake of brevity, we shorten this argument to ``without loss of generality, we assume that the edge that remained after application of the rule is the one incident to $y$ in $F$'', and apply similar shortenings in the remainder. Similarly, for Reduction Rule~\ref{rr-tangle-3}, we note that $F$ most contain an edge from $y$ to a vertex in $A$ or to a vertex in $B$; without loss of generality, we assume that one of the edges that remained after the application of the rule is the edge incident to $y$ in $F$. Now consider the construction of $\hat{F}_{P_F}$:
\begin{itemize}
\item for any $y \in Y \cap V(P)$, $N_F(y) = N_{P_F}(y)$ by definition and $N_{\hat{F}_{P_F}}(y) = N_{P_F}(y)$ by construction.
\item for every non-terminal $y \in Y \setminus V(P_F)$, no incident edge is in $F$: such a $y$ would have degree~$1$ in $F$, contradicting minimality.
\item for any terminal $y \in Y \setminus V(P)$:
\begin{itemize}
\item if $y$ belongs to a school $S$ such that $S \subseteq Y$, then there is a strong subset $A$ such that $S \subseteq N_F(A)$. There are also $|S|$ edges incident to $S$, exactly one per vertex of $S$. At the same time, since $S \subseteq Y$, for $A(S)$, we can change $F$ to use $|S|$ edges between $A(S)$ and $S$, one per vertex of $S$, to obtain an equivalent solution. Hence, we may assume we add exactly the same edge from $y$ to $A(S)$ in $F$ as in $\hat{F}_{P_F}$.
\item if $y$ has degree~$1$, then trivially we add exactly the same edge from $y$ to $A(S)$ in $F$ as in $\hat{F}_{P_F}$.
\end{itemize}
\item for any $uv \in E(P_F)$, note that $F$ is intertwined with respect to $\mathcal{L}_{uv}$ by definition. Hence, $F$ contains a path $Q$ from $u$ to $v$ using only vertices of $V(u,v,\mathcal{L}_{uv})$. Note that, by the existence of $Q$, for any $(s,t) \in T$ such that $s,t \in V(u,v,\mathcal{L}_{uv})$, $F$ contains a path between $s$ and $t$ using only vertices in $V(u,v,\mathcal{L}_{uv})$. Moreover, for any $(s,t) \in T$ such that $s \in V(u,v,\mathcal{L}_{uv})$ and $t \not\in V(u,v,\mathcal{L}_{uv})$, note that $F$ contains a path between $s$ and $u$ or $v$ using only vertices of $V(u,v,\mathcal{L}_{uv})$, and using $Q$, a path between $s$ and $u$ using only vertices of $V(u,v,\mathcal{L}_{uv})$. Hence, $E(F) \cap E(G[V(u,v,\mathcal{L}_{uv})])$ is a Steiner forest for $G[V(u,v,\mathcal{L}_{uv})]$ and the constructed set $T_{uv}$ that is intertwined with respect to the citrus. Since $E(\hat{F}_{P_F}) \cap E(G[V(u,v,\mathcal{L}_{uv})])$ is a minimum such Steiner forest, which in particular is intertwined, we may assume without loss of generality that $E(F) \cap E(G[V(u,v,\mathcal{L}_{uv})]) = E(\hat{F}_{P_F}) \cap E(G[V(u,v,\mathcal{L}_{uv})])$.
\item for any $uv \in E(G'') \setminus E(P_F)$ for which $\mathcal{L}_{uv}$ is strong, note that $F$ contains a path $Q$ from $u$ to $v$ that does have an internal vertex in $V(u,v,\mathcal{L}_{uv})$. Note that, for any $(s,t) \in T$ such that $s,t \in V(u,v,\mathcal{L}_{uv})$, $F$ contains a path between $s$ and $t$ using only vertices in $V(u,v,\mathcal{L}_{uv}) \cup V(Q)$. Moreover, for any $(s,t) \in T$ such that $s \in V(u,v,\mathcal{L}_{uv})$ and $t \not\in V(u,v,\mathcal{L}_{uv})$, note that $F$ contains a path between $s$ and $u$ or $v$. Hence, $E(F) \cap E(G[V(u,v,\mathcal{L}_{uv})] \dagger\{u,v\})$ is a Steiner forest for $G[V(u,v,\mathcal{L}_{uv})\dagger\{u,v\}]$ and the constructed set $T_{uv} \dagger\{u,v\}$. Since $E(\hat{F}_{P_F}) \cap E(G[V(u,v,\mathcal{L}_{uv})]\dagger\{u,v\})$ is a minimum such Steiner forest, by the existence of $Q$ we may assume without loss of generality that $E(F) \cap E(G[V(u,v,\mathcal{L}_{uv})]) = E(\hat{F}_{P_F}) \cap E(G[V(u,v,\mathcal{L}_{uv})])$.
\item for any $uv \in E(G'') \setminus E(P_F)$ for which $\mathcal{L}_{uv}$ is delicate, note that there is no path in $F$ between $u$ and $v$. In particular, if $uv \in E(G)$, then $uv \not\in E(F)$. Hence, for any $(s,t) \in T$ for which $s,t \in V(u,v,\mathcal{L}_{uv})$, there is a path between $s$ and $t$ in $F$ that only uses vertices of $V(u,v,\mathcal{L}_{uv})$. Now consider a pair $(s,t) \in T$ such that $s \in V(u,v,\mathcal{L}_{uv})$ and $t \not\in V(u,v,\mathcal{L}_{uv})$. Since there is a path between $s$ and $t$ in $F$, $t$ must be either: in a strong $\ell$-lemon with an end in $A$ or $B$, in a delicate $\ell$-lemon with an end in $A$ and an end not in $B$ (or vice versa), in $Y(A)$ or $Y(B)$, or in a delicate $\ell$-lemon with an end $u' \in A$ and an end $v' \in B$ (for which $\{u',v'\}\not= \{u,v\}$). In the first three cases, it is clear that $F$ contains a path between $s$ and $u$ or between $s$ and $v$ respectively using only vertices of $V(u,v,\mathcal{L}_{uv})$. In the fourth and final case, which path exists is not so clear. However, if $F$ contains a path between $s$ and $u$, then the path in $F$ between $s$ and $t$ contains $u$ and $u'$; similarly, if $F$ contains a path between $s$ and $v$, then the path in $F$ between $s$ and $t$ contains $v$ and $v'$. Hence, in $E(F) \cap E(G_{AB})$, there is a path from $s$ to $t$ using only vertices of $V(G_{AB})$. In particular, these arguments show that $E(F) \cap E(G_{AB})$ is a Steiner forest for $(G_{AB},T_{AB})$. Since $E(\hat{F}_{P_F}) \cap E(G_{AB})$ is a minimum such Steiner forest, and there is a path in $F$ between any pair of vertices in $A$ and between any pair of vertices in $B$, we may assume without loss of generality that $E(F) \cap E(G_{AB}) = E(\hat{F}_{P_F}) \cap E(G_{AB})$. In particular, $E(F) \cap E(G[V(u,v,\mathcal{L}_{uv})]) = E(\hat{F}_{P_F}) \cap E(G[V(u,v,\mathcal{L}_{uv})])$.
\end{itemize}
From this analysis, it follows that $\hat{F}_{P_F}$ is a minimum Steiner forest for $(G,T)$.
\end{proof}

\section{\fl{Polynomial-Time Results}}%The Full Proofs of Section~\ref{s-poly}}
\label{s-fp-poly}

\subsection{Forbidding a Path $\mathbf{P_{11}}$}

We first prove the following lemma.

\begin{restatable}{lemma}{lemPIOfree}
\label{l-p10free}
\stf{} is polynomial-time solvable for $P_{10}$-subgraph-free graphs.
\end{restatable}

\begin{proof}
    Let $(G,T,k)$ be an instance of {\sc Steiner Forest}, where $G$ is a $P_{10}$-subgraph-free graph, $T$ a terminal set, and $k$ an integer. We apply Lemma~\ref{l-2con} and assume that $G$ is $2$-connected. If $G$ is $P_9$-subgraph-free we apply Lemma~\ref{t-pol}, therefore we may assume in the following that $G$ has $P_9$ as a subgraph. Note that such a path can be found in polynomial time applying Lemma~\ref{l-longestpath}.
    Let $P = (v_1, \dots , v_9)$ be a $P_9$ which is a subgraph of $G$.

    If $G-V(P)$ consists only of isolated vertices, then $V(P)$ is a vertex cover of $G$ of size~$9$. Hence, we conclude by applying Theorem~\ref{t-vc}. 
    Thus, in the following we may assume there is at least one component $D$ of $G-V(P)$ of size at least~$2$.

    \begin{claim}\label{c-p10-comp-adjacencies}
        Let $D$ be any component of $G-V(P)$ of size at least~2. Then:
        \begin{enumerate}[i]
            \item $N(D) \subseteq \{v_3, \ldots, v_7\}$. \label{i-p10-d-nadj-1289}
            \item $D$ is adjacent to exactly 2 vertices of $P$.
            \label{i-p10-d-adj-2}
            \item The vertices in $N(D)$ are at distance at least~$3$ on $P$. 
            \label{i-p10-d-adj-dist-3}
            \item $D$ is a star, and the centre of $D$ is adjacent to $v_3$ or $v_7$.
            \label{i-p10-d-star}
            \item $D \cup N(D)$ is a juicy or seeded wedge.\label{i-p10-wedge}
            
        \end{enumerate}
    \end{claim}

    \begin{claimproof}
        We first prove item~\ref{i-p10-d-nadj-1289}: Note first that $N(v_1) \subseteq P$, since otherwise we can extend $P$ by a neighbour of $v_1$ and obtain a $P_{10}$. 
        Suppose now that $D$ is adjacent to $v_2$. Let $x,y \in D$, such that $v_2x, xy \in E(G)$. Then, $(y,x,v_2, \dots , v_9)$ is a $P_{10}$, a contradiction. Therefore, no such component exists. Symmetric arguments apply to $v_9$ and $v_8$. Item~\ref{i-p10-d-nadj-1289} follows.
        
        For item~\ref{i-p10-d-adj-2}, suppose for a contradiction that $D$ is adjacent to 3 or more vertices on~$P$. From item~\ref{i-p10-d-nadj-1289}, these must include 3 of $\{v_3,\ldots,v_7\}$. 
        Applying Lemma~\ref{l-adjacencytopath}, we obtain that $N(D) \cap P = \{v_3, v_5, v_7\}$. Let $x \in D$ with  $v_5 \in N(x)$. By Lemma~\ref{l-adjacencytopath} there does not exist any vertex $y \ne x$ in $D$ adjacent to $v_3$ or $v_7$. Consequently, $N(v_3)\cap D = N(v_7)\cap D = \{x\}$. Applying Lemma~\ref{l-adjacencytopath} again we obtain that no vertex $y \ne x$ in $D$ is adjacent to $v_5$. Consequently, either $|D| = 1$, contradicting our premise, or $x$ is a cut-vertex of $G$ and so $G$ cannot be 2-connected, a contradiction by Lemma~\ref{l-2con}. Item~\ref{i-p10-d-adj-2} follows.

        For item~\ref{i-p10-d-adj-dist-3}, let $N(D)=\{u,v\}$. Since $D$ does not contain a cut-vertex, $u$ and $v$ collectively have at least two neighbours in $D$. Applying Lemma~\ref{l-adjacencytopath} $u$ and $v$ have distance at least~3 on~$P$, and item~\ref{i-p10-d-adj-dist-3} follows.

        For item~\ref{i-p10-d-star}. 
        Note that by items~\ref{i-p10-d-nadj-1289}-\ref{i-p10-d-adj-dist-3}, $D$ is adjacent to at least one of~$v_3$ and $v_7$. Without loss of generality, we may assume that $D$ is adjacent to $v_3$.
        Let $x$ be a neighbour of $v_3$ in~$D$, and $y$ a neighbour of $x$ in $D$. Suppose $y$ has a neighbour $z \ne x$ in $D$ then we find a $P_{10}$ $(z,y,x,v_3,\ldots ,v_9)$, a contradiction. Thus, $D$ is a star with centre $x$. Item~\ref{i-p10-d-star} follows.

        Lastly, we prove item~\ref{i-p10-wedge}. If $|D|=2$ then $D\cup N(D)$ is immediately a seeded wedge with ends $v_3,v_6$, or $v_3, v_7$, or $v_4, v_7$. Now suppose $|D|\ge 3$. Again call $x$ the centre of the star $D$ and let $y$ and $y'$ be two neighbours of $x$ in $D$.    
        Suppose for contradiction that $y$ is adjacent to $v_3$ or $v_7$. 
        Now either $(y',x,y,v_3 ,\ldots ,v_9)$ or $(y',x,y,v_7 ,\ldots ,v_1)$ is a $P_{10}$, a contradiction. 
        From item~\ref{i-p10-d-star} $x$ neighbours either $v_3$ or $v_7$, say $v_3$ by symmetry. 
        Thus, applying item~\ref{i-p10-d-adj-2} and Lemma~\ref{l-2con} we obtain that $N(y)=\{x,v_6\}$ and this holds for all neighbours of $x$ in $D$. 
        Now $D \cup N(D)= D \cup \{v_3, v\}$ is a juicy wedge witnessed by the path-decomposition of width 2 with bags $\{v_3,v_6,x\},$ $\{v_6, x, y_1\}, \{v_6, x, y_2\}, \ldots$ for successive neighbours $y_i$ of $x$ in $D$. Item~\ref{i-p10-wedge} follows.
    \end{claimproof}

    \noindent
    Let $D$ be a component of $G-V(P)$. By Claim~\ref{c-p10-comp-adjacencies}, $D$ is a star with some vertex $x$ as its centre, and without loss of generality (by symmetry with $v_7$) $x \in N(v_3)$. Also, $x$ has a neighbour $y$ which is adjacent to either $v_6$ or $v_7$.

    \medskip
    \noindent
    \textbf{Case 1:} $y$ is adjacent to $v_6$.\\
        Applying Claim~\ref{c-p10-comp-adjacencies}, $D\cup \{v_3, v_6\}$ is a juicy or seeded wedge.
        Suppose there is another component $D'$ of $G-V(P)$ of size at least~$2$. 
        Applying Claim~\ref{c-p10-comp-adjacencies} again, we obtain that $D'$ is a star adjacent to two vertices of $P$, including either $v_3$ or $v_7$.
        \begin{itemize}
            \item If $D'$ is adjacent to $v_3$ and $v_6$, then by Claim~\ref{c-p10-comp-adjacencies} $D'\cup\{v_3,v_6\}$ is a juicy or seeded wedge. 
            \item If $D'$ is adjacent to $v_3$ and $v_7$, we find a path of length at least~$11$ (and hence a $P_{10}$) as follows: $(v_9, v_8, v_7, D' , v_3 , \dots , v_6 , y, x)$, a contradiction.
            \item If $D'$ is adjacent to $v_4$ and $v_7$, we find a path of length at least~$13$ (and hence a $P_{10}$) as follows: $(v_1, v_2, v_3, x, y,v_6,v_5,v_4, D' , v_7 , v_8 , v_9)$, a contradiction.
        \end{itemize}
        It follows that for any component $D'$ of $G-V(P)$, $D' \cup \{v_3, v_6\}$ is a juicy or seeded wedge with ends $v_3$ and $v_6$. Thus all components of $G-V(P)$ of size at least~$2$ are wedges of a lemon with ends $v_3,v_6$; all other components of $G-V(P)$ are an independent set (tangle), and $G$ is an entangled lemon bush with a stem set of size~$9$ (the vertices of $P$), so we conclude by applying Theorem~\ref{t-solvetangle}. 
    \medskip
    
    \noindent
    \textbf{Case 2:} $y$ is adjacent to $v_7$.\\
    Again by Claim~\ref{c-p10-comp-adjacencies}, $D \cup \{v_3, v_7\}$ is a juicy or seeded wedge. 
    Suppose there is another component $D'$ of $G-V(P)$ of size at least~$2$. Applying Claim~\ref{c-p10-comp-adjacencies} again, we are left with the following three cases to consider:
    \begin{itemize}
        \item If $D'$ is adjacent to $v_3$ and $v_6$, this is covered in Case 1 above. 
        \item If $D'$ is adjacent to $v_3$ and $v_7$, $D' \cup \{v_3, v_7\}$ is a juicy or seeded wedge by Claim~\ref{c-p10-comp-adjacencies}.
        \item If $D'$ is adjacent to $v_4$ and $v_7$, we find a path on 11 vertices (and hence a $P_{10}$) as follows: $(v_1, v_2, v_3, x, y,v_7, v_6, v_5,v_4, D')$. 
    \end{itemize}
        It follows that for any component $D'$ of $G-V(P)$, $D' \cup \{v_3, v_7\}$ is a juicy or seeded wedge with ends $v_3$ and $v_7$. Thus all components of $G-V(P)$ of size at least~$2$ are wedges of a lemon with ends $v_3,v_7$, and
        all other components of $G-V(P)$ are an independent set (tangle), so $G$ is (again) an entangled lemon bush with a stem set of size~$9$ (the vertices of~$P$), and we again conclude by applying Theorem~\ref{t-solvetangle}. 
        
        \medskip
        \noindent        
        This concludes the proof.
\end{proof}

\noindent
We can now prove Theorem~\ref{t-p11free}, which we restate below.

\thmPIIfree*
\begin{proof}
    Let $(G,T,k)$ be an instance of {\sc Steiner Forest}, where $k$ is some integer, $T$ a terminal set, and $G$ is a $P_{11}$-subgraph-free graph. From Lemma~\ref{l-2con}, we assume $G$ is $2$-connected. Further, if $G$ is $P_{10}$-subgraph-free, then we can solve {\sc Steiner Forest} on $(G,T,k)$ in polynomial time from Lemma~\ref{l-p10free}. That is, we may assume $G$ contains some $P_{10}$ subgraph, $P = (v_1, \dots , v_{10})$. Such a path can be found in polynomial time applying Lemma~\ref{l-longestpath}.

    If $G-V(P)$ contains only isolated vertices, then $V(P)$ is a vertex cover of $G$ of size~$10$ and we conclude by applying Theorem~\ref{t-vc}.

    \begin{figure}
        \centering
        \includegraphics[width=0.7\linewidth, page=4]{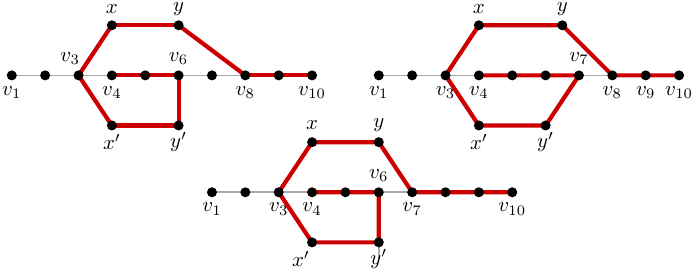}
        \caption{Occurrences of $P_{11}$ in the proof of Theorem~\ref{t-p11free}.}
        \label{fig:p11-contradiction-3-678}
    \end{figure}

    \begin{figure}
        \centering
        \includegraphics[width=\linewidth, page=5]{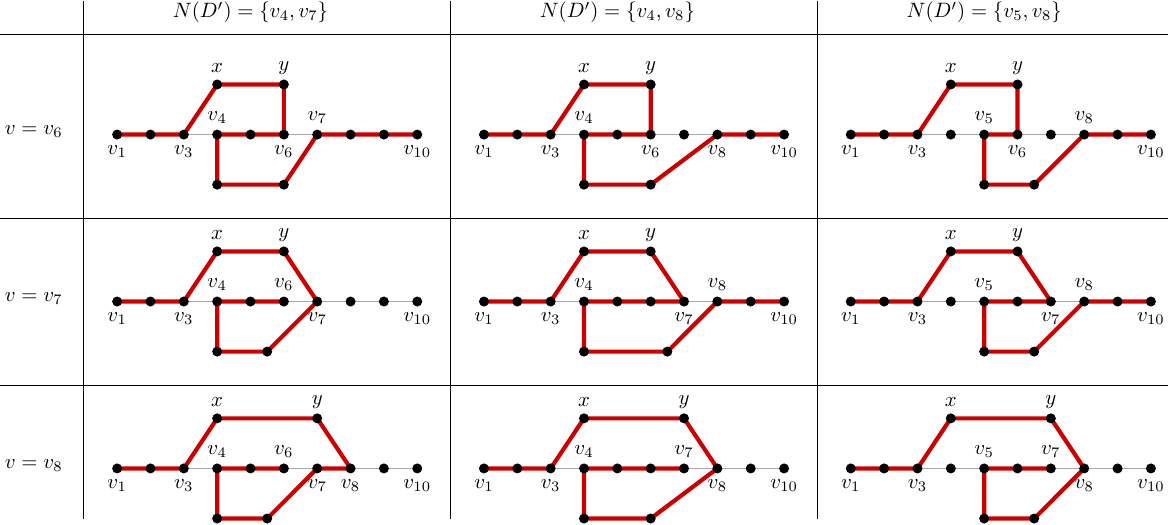}
        \caption{Occurrences of $P_{11}$ in the proof of Theorem~\ref{t-p11free}.}
        \label{fig:p11-contradiction-grid}
    \end{figure}
    
    \begin{claim}\label{c-p11-38-comps}
        Let $D$ be any component of $G-V(P)$ of size at least~2 such that $D$ is adjacent to $v_3$ (respectively $v_8$). Then all the following hold:
        \begin{enumerate}[i]
            \item $G[D]$ is a star. \label{i-p11-star}
            \item There exists an edge $xy$ with $x,y \in D$ such that $xv_3$ (resp.~$xv_8$) is an edge and $yv$ is an edge for some $v\in \{v_6,v_7,v_8\}$  (resp.~$v\in \{v_3,v_4,v_5\}$). \label{i-p11-xy} 
            \item If $|N(D)|=2$, then $D \cup \{N(D)\}$ is a juicy or seeded wedge.\label{i-p11-wedge}
        \end{enumerate}
    \end{claim}
    \begin{claimproof}
        By symmetry, we only prove the claim where $D$ is adjacent to $v_3$.
        
        We begin by proving item~\ref{i-p11-star}. Let $x \in D$ be a neighbour of $v_3$ and let $y \in D$ be a neighbour of $x$.
        If $y$ has some other neighbour $x' \in D$, then $(x',y,x,v_3,\ldots,v_{10})$ is a $P_{11}$ in $G$. Hence, $N(y) \subseteq P \cup \{x\}$. This implies that $G[D]$ is a star with $x$ at the centre. Item~\ref{i-p11-star} follows.        
        
        Now for item~\ref{i-p11-xy}. Applying item~\ref{i-p11-star} let $x$ be the centre of the star $D$ and $y$ some neighbour of $x$ in $D$.
        Given that $G$ is $2$-connected, there is a path from every vertex in $D \setminus \{x\}$ to $P$ which does not contain $x$. Since $G[D]$ is a star, it follows that every vertex in $D$ has some neighbour in $P$. Further, again due to the $2$-connectedness, we know that some vertex in $D$ is adjacent to some vertex $v \in P \setminus \{v_3\}$.  
        After potentially relabelling $x$ and $y$, we may assume this is $y$. 
        From Lemma~\ref{l-adjacencytopath}, we obtain that $v_4,v_5 \notin N(y)$, that is $v \in \{v_6,v_7,v_8\}$. Item~\ref{i-p11-xy} follows.
        
        Now we show item~\ref{i-p11-wedge}. Applying item~\ref{i-p11-xy}, let $xy$ be an edge with $x,y \in D$ such that $xv_3$ is an edge and $yv$ is an edge for some $v\in \{v_6,v_7,v_8\}$.
        By the premise of item~\ref{i-p11-wedge}, $|N(D)|=2$, so $N(D)=\{v_3, v\}$. If $|D|=2$ then $D \cup N(D)$ is a seeded wedge and we are done. 
        If $|D| \geq 3$, then $x$ has some other neighbour $y' \in D$. Since $x$ is the centre of the star $G[D]$ and $N(D)=\{v_3,v\}$, we know that $N(y') \subseteq \{v_3,v, x\}$. Moreover, $y'$ is not adjacent to $v_3$ since otherwise $G$ contains the $P_{11}$ $(y, x, y', v_3, \ldots, v_{10})$, a contradiction; that is, $N(y')\subseteq\{v,x\}$. 
        As $G$ is $2$-connected, it follows that $N(y') = \{v, x\}$. Note that this holds for every vertex in $D \setminus \{x\}$.
        Now $D \cup N(D)= D \cup \{v_3, v\}$ is a juicy wedge witnessed by the path-decomposition of width~$2$ with bags $\{v_3,v,x\},$ $\{v, x, y_1\}, \{v, x, y_2\}, \ldots$ for successive neighbours $y_i$ of $x$ in $D$.
        Item~\ref{i-p11-wedge} follows.
    \end{claimproof}

    \noindent
    We can leverage this structure to derive a polynomial-time algorithm when it arises:
    
    \begin{claim}\label{c-p11-38-poly}
        If there exists a component $D$ of size at least~$2$ in $G-V(P)$ such that $v_3 \in N(D)$ or $v_8 \in N(D)$, then for any $T,k$, the instance $(G,T,k)$ of \stf{} can be solved in polynomial time. 
    \end{claim}
    \begin{claimproof}
        Consider such a component $D$, and assume by symmetry that $D$ is adjacent to $v_3$. Applying Claim~\ref{c-p11-38-comps}.\ref{i-p11-xy}, let vertices $x,y \in D$ with $xv_3, yv \in E(G)$ for some $v \in \{v_6, v_7,v_8\}$.
        
       Note that $x$ is a vertex cover of $G[D]$ from Claim~\ref{c-p11-38-comps} item~\ref{i-p11-star}. If $V(P) \cup \{x\}$ is a vertex cover of $G$, then we conclude by applying Theorem~\ref{t-vc}. Otherwise, there must exist some other component in $G-V(P)$ of size at least~2. We now deal with cases where such a component exists. 
       First suppose that there exists a second component $D'$ of $G-V(P)$ of size at least~2 such that $N(D') \ne \{v_3, v\}$. By Lemma~\ref{l-2con}, $|N(D')|\ge 2$. 
       \begin{enumerate}
           \item If $v_3 \in N(D')$ then, applying the reasoning above, there are vertices $x', y' \in D'$ such that $x'$ is adjacent to $v_3$ and $y'$ is adjacent to some $v' \in \{v_6, v_7, v_8\}$. We have that $v\ne v'$.
           If $v=v_8$ and $v'=v_7$ then we find a $P_{11}$ as shown in Figure~\ref{fig:p11-contradiction-3-678} (top left). If $v=v_8$ and $v'=v_7$ then we find a $P_{11}$ as shown in Figure~\ref{fig:p11-contradiction-3-678} (top right). If $v=v_7$ and $v'=v_6$ then we find a $P_{11}$ as shown in Figure~\ref{fig:p11-contradiction-3-678} (bottom). The cases where $v=v_i$ and $v'=v_j$ for $j>i$ are symmetric. Consequently, we find that $v_3 \notin N(D')$.
           \item If $v_3 \notin N(D')$ then applying Lemma~\ref{l-adjacencytopath} we have 3 cases: $N(D')=\{v_4,v_7\}$, $N(D')=\{v_4,v_8\}$, or $N(D')=\{v_5,v_8\}$. We also again have the 3 possibilities for $v$. Figure~\ref{fig:p11-contradiction-grid} shows that in any case, we can extend the path $P$ and find a $P_{11}$ subgraph, a contradiction.
       \end{enumerate}

        \noindent
        Thus, we obtain that $N(D') = \{v_3, v\}$ for every component $D' \ne D$ in $G-V(P)$ of size at least~$2$, and that there exists at least one such component $D'$. Then by symmetry we also have that $N(D) = \{v_3, v\}$; that is, we can at this point exclude the possibility that $|N(D)| \ge 3$. Applying Claim~\ref{c-p11-38-comps}.\ref{i-p11-wedge}, we have that for every component $D$ of $G-V(P)$ of size at least~$2$, $D \cup N(D)$ is a juicy or seeded wedge. 

        We let $G'' = G[V(P)] \cup \{v_3v\}$. Then, $G' = G[V(P) \cup D]$ is a lemon bush flowering from~$G''$ with stem set $V(P)$.
        Further, all vertices in $V(G) - (V(P)  \cup D)$ are adjacent only to vertices in $P$ and are thus a tangle set. Hence, we obtain that $G$ is an entangled lemon bush with $V(P)$ its stem set of size~$10$, and we conclude by applying Theorem~\ref{t-solvetangle}.
    \end{claimproof}
    
    \noindent
    It remains to deal with the case where no component of size at least~2 is adjacent to either $v_3$ or $v_8$. 
    \begin{claim}\label{c-p11-47-wedge}
        If no component $D$ of $G-V(P)$ of size at least~2 is adjacent to either $v_3$ or $v_8$, then for any $T,k$, the instance $(G,T,k)$ of \stf{} can be solved in polynomial time.
    \end{claim}
    \begin{claimproof}
        Recall that no component of size at least~$2$ is adjacent to any of $v_1, v_2, v_9, v_{10}$. We will show that any component $D$ of $G-V(P)$ of size at least~2 with $N(D) \subseteq \{v_4, \ldots, v_7\}$ satisfies that $D \cup N(D)$ is a juicy or seeded wedge with ends $v_4$ and $v_7$. 
        
        First note that $N(D)=\{v_4,v_7\}$ by applying 2-connectivity together with Lemma~\ref{l-adjacencytopath}. If $|D|=2$ then $D \cup N(D)$ is immediately a seeded wedge. We now consider the case where $|D|\ge 3$.
        Let $x,y \in D$ be two vertices such that $xv_4 \in E(G)$ and $yv_7\in E(G)$.
    
        If there is some path from $x$ to $y$ in $G[D]$ with length at least~$2$, then $(v_1, \ldots , v_4, x, \ldots, y,$ $v_7, \dots, v_{10})$ is a $P_{11}$ in $G$. It follows that $x$ and $y$ are adjacent and there is no other path from $x$ to $y$ in $G[D]$.        
        Note that the same holds for any pair of vertices in $D$ such that one is adjacent to $v_4$ and the other to~$v_7$.
    
        Suppose that $x$ has another neighbour $y'$. As $G$ is $2$-connected and $N(D) = \{v_4,v_7\}$, there is some path in $G[D \cup \{v_4,v_7\}]$ from $y'$ to either $v_4$ or $v_7$ which does not contain $x$. Further, this path does not contain $y$, since this would entail the existence of a path from $x$ to $y$ of length at least~2, a contradiction.  
    
        Suppose first there is a path from $y'$ to $v_4$ using only vertices in $D \cup \{v_4\} \setminus \{x,y\}$.
        Regardless of the length of this path, we find a $P_{11}$ contained in $(v_1, \dots , v_4, \dots, y', x, y, v_7, \dots, v_{10})$, a contradiction.
        If there is a path of length at least~$2$ from $y'$ to $v_7$ using only vertices in
        $D \cup \{v_7\} \setminus \{x,y\}$,
        we find a $P_{11}$ contained in $(y, x, y', \dots, v_7, \dots, v_1)$, a contradiction.
        It follows that $y'$ is adjacent to $v_7$. Thus, all neighbours of $x$ in $D$ are adjacent to $v_7$ and, by symmetry, all neighbours of~$y$ in $D$ are adjacent to $v_4$.
        
        Suppose for contradiction that $x$ has some neighbour $y' \in D$ and $y$ has some neighbour $x' \in D$.
        Given that $x'$ is adjacent to $v_4$ and $y'$ is adjacent to $v_7$, we find the path $(v_1, \dots, v_4, x', y, x, y',v_7, \dots, \,v_{10})$ of length 12 and hence a $P_{11}$, a contradiction. 
        It follows that only one of $x$ and $y$ can have additional neighbours and thus, $G[D]$ is a star with either $x$ or $y$ at the centre. If the centre is $x$, then all of the leaves of this star are adjacent to $v_7$, otherwise all the leaves are adjacent to $v_4$. 
        If $x$ is the centre then $G[D \cup N(D)]$ admits the path-decomposition of width~$2$ with bags $\{v_4,v_7,x\},$ $\{v_7, x, y_1\}, \{v_7, x, y_2\}, \ldots$ for successive neighbours $y_i$ of $x$ in $D$. A symmetric decomposition exists if $y$ is the centre. Thus, $D \cup N(D)$ is a juicy wedge.
        
        Now we let $G'' = G[V(P)] \cup \{v_4v_7\}$. Then, applying Claim~\ref{c-p11-47-wedge}, $G' = G[V(P) \cup D]$ is a lemon bush flowering from~$G''$ with stem set $V(P)$.
        Further, all vertices in $V(G) - (V(P)  \cup D)$ are adjacent only to vertices in $P$ and are thus a tangle set. Hence, we obtain that $G$ is an entangled lemon bush with $V(P)$ its stem set of size~$10$, and we can apply Theorem~\ref{t-solvetangle}. This concludes the proof of the claim.
    \end{claimproof}

    \noindent
    Thus all cases can be solved in polynomial time: we know there is at least one component of $G-V(P)$ of size at least~$2$. If it is adjacent to one of $v_3$ and $v_8$ Claim~\ref{c-p11-38-poly} applies, otherwise Claim~\ref{c-p11-47-wedge} applies. This concludes the proof.
    \end{proof}

\subsection{Forbidding a Subdivided Claw $\mathbf{S_{1,3,6}}$}

In this section we prove Theorem~\ref{t-s136}, which we restate below.

\thmSIEbfree*
\begin{proof}
    Let $(G,T,k)$ be an instance of \stf, where $G$ is a $S_{1,3,6}$-subgraph-free graph.
    From Lemma~\ref{l-2con}, we may assume that $G$ is $2$-connected.
    Let $P = (v_1, \dots, v_r)$ be a longest path in $G$, which can be found in polynomial time by Lemma~\ref{l-longestpath}.
    If $G$ is $P_{11}$-subgraph-free, then, applying Theorem~\ref{t-p11free}, \stf{} can be solved in polynomial time on $(G,T,k)$.
    Therefore, we now assume $r\geq 11$.

\begin{figure}
    \centering
    \includegraphics[width=0.8\linewidth,page=7]{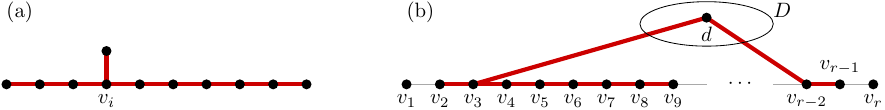}
    \caption{Occurrences of $S_{1,3,6}$ in the proof of Theorem~\ref{t-s136}.}
    \label{fig:s136-i}
\end{figure}

    \begin{claim}\label{clm-s334-P}
        If $r \geq 12$, then $V(G) = V(P)$.
    \end{claim}
    \begin{claimproof}
        Suppose that $|V(P)| \geq 12$. If some vertex $v_i \in \{v_4,\dots, v_{r-3}\}$ has some neighbour in $G-V(P)$, then we find $S_{1,3,6}$, see Figure~\ref{fig:s136-i}\,a), a contradiction. Likewise, as $P$ is a longest path, the vertices $v_1$ and $v_{r}$ do not have any neighbour in $G-V(P)$. It follows that no vertex in $P - \{v_2,v_3, v_{r-2},v_{r-1}\}$ has a neighbour in $G-V(P)$.

        Suppose there is a component $D$ of $G-V(P)$. 
        Since we assume $G$ to be $2$-connected, and $v_1,v_4,\dots, v_{r-3}, v_{r}$ have no neighbour in $G-V(P)$, $D$ has at least two neighbours $v_i, v_j \in V(P)$ with $i, j \in \{2,3,r-2,r-1\}$. 
        By Lemma~\ref{l-adjacencytopath}, without loss of generality, $i \in \{2,3\}$ and $j \in \{r-2, r-1\}$.
        This results in a $S_{1,3,6}$ subgraph, see Figure~\ref{fig:s136-i}\,b), a contradiction. Hence, $V(G) = V(P)$.
    \end{claimproof}

\begin{figure}
    \centering
    \includegraphics[width=0.9\linewidth,page=8]{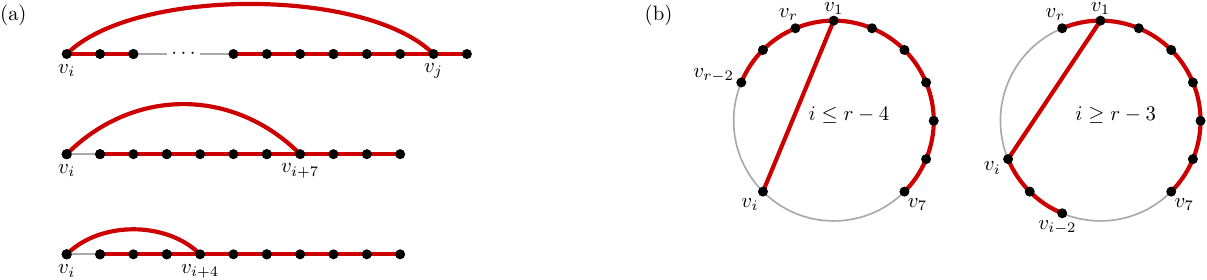}
    \caption{Occurrences of $S_{1,3,6}$ in the proof of Theorem~\ref{t-s136}.}    \label{fig:s136-ii}
\end{figure}

    \noindent
    We now consider the case where $r\geq 13$.

    \begin{claim}\label{clm-s334-cyc}
        If $r\geq 13$, then $G$ is a cycle of length $r$.
    \end{claim}
    \begin{claimproof}
        We first show that $G$ contains some cycle of length $r$. Suppose, for a contradiction, $G$ contains no cycle of length $r$.

        We first show that $N(v_1) \subseteq \{v_2,v_3,v_4\}$ and $N(v_2) \subseteq \{v_1,v_3,v_4,v_5\}$, that is, $v_1$ is not adjacent to any vertex in $\{v_5, \ldots, v_r\}$ and $v_2$ is not adjacent to any vertex in $\{v_6, \ldots, v_r\}$. As $G$ does not contain a cycle of length $r$, $v_1$ is not adjacent to $v_r$. 
        
        We now claim that for $v_i \in \{v_1,v_2\}$, if $v_i$ is adjacent to some vertex $v_j$, for $j \in \{{i+4,} \ldots, v_{r-1}\}$, then we find some $S_{1,3,6}$ subgraph. As illustrated in Figure~\ref{fig:s136-ii}\,a), we distinguish between the cases where $j \in \{i+9, \ldots, r-1\}$ (top), $j \in \{i+7,i+8\}$ (middle) and $j \in \{i+4, \ldots, i+6\}$ (bottom), but in each case we find a $S_{1,3,6}$ subgraph. It follows that $N(v_1) \subseteq \{v_2,v_3,v_4\}$ and, as the edge $v_1v_{r-1}$ is symmetric to the edge $v_2v_{r}$, $N(v_2) \subseteq \{v_1,v_3,v_4, v_5\}$. 
        We will also use this reasoning more generally to reason about the neighbourhood of the first and second vertex of a path of length $r-1$.

        We consider again $v_1$. As $v_2$ is not a cut vertex and $N(v_1) \subseteq \{v_2,v_3,v_4\}$, at least one of the edges $v_1v_3$ or $v_1v_4$ must exist.
        
        First suppose that $G$ does not contain the edge $v_1v_4$. As $G$ is $2$-connected, $N(v_1) = \{v_2,v_3\}$. 
        As $v_2$ is the end of the path $(v_2, v_1, v_3, v_4, \ldots, v_r)$ of length $r-1$, it follows that $N(v_2) \subseteq \{v_1, v_3, v_4\}$. As $v_3$ is not a cut vertex, it follows that $N(v_2) = \{v_1,v_3,v_4\}$ 
        and so $v_3$ is the end vertex of the path $(v_3, v_2, v_1, v_4, \ldots, v_r)$ of length $r-1$. Applying the same argument, we find that $N(v_3) = \{v_1,v_2,v_4\}$. As $N(v_1) = \{v_2,v_3\}$, $N(v_2) = \{v_1,v_3,v_4\}$ and $N(v_3) = \{v_1,v_2,v_4\}$, it follows that $v_4$ is a cut vertex, a contradiction. Thus, $G$ contains the edge $v_1v_4$.

        Given that $v_1$ is adjacent to $v_4$, it follows that $v_3$ is the end of the path $(v_3,v_2,v_1,v_4, \ldots,v_r)$ of length $r-1$ and so $N(v_3) \subseteq \{v_1,v_2,v_4\}$. Recall that $N(v_2) \subseteq \{v_1,v_3,v_4,v_5\}$. As $v_4$ is not a cut vertex, $v_2$ is adjacent to $v_5$. Given that $v_4$ is the second vertex of the path $(v_1,v_4,v_3,v_2,v_5, \ldots, v_r)$ of length $r-1$, using the same reasoning as for $v_2$, $N(v_4) \subseteq \{v_1,v_2,v_3,v_5\}$. Note that this implies that $v_5$ is a cut vertex, a contradiction. It follows that $G$ contains some cycle of length $r$.

        Without loss of generality, this cycle consists of the edges of $P$ alongside the edge $v_1v_r$. We now claim that $G$ is a cycle of length $r$. Suppose, for a contradiction, there is some edge between a pair of non-consecutive vertices along this cycle. Without loss of generality this is the edge $v_1v_i$ for $i \geq 8$. We further distinguish between the cases where $i \leq r-4$ and $i \geq r-3$. In both cases we find some $S_{1,3,6}$ as shown in Figure~\ref{fig:s136-ii}\,b).
    \end{claimproof}

    \noindent
    By Claim~\ref{clm-s334-cyc}, if $r \geq 13$, then $G$ is a cycle of length $r$. It follows that $G$ has treewidth~$2$ and so applying Lemma~\ref{l-tw}, we can solve \stf{} in polynomial time on $(G,T,k)$.

    Suppose now $r =12$, then by Claim~\ref{clm-s334-P}, $V(G) = V(P)$. As $G$ has size~$12$, we can solve \stf{} in polynomial time on $(G,T,k)$.

    Hence, we may assume that $r = 11$.
    We now claim that $G$ is an entangled lemon bush with a stemset of size~$11$. 

\begin{figure}
    \centering
    \includegraphics[width=0.8\linewidth,page=9]{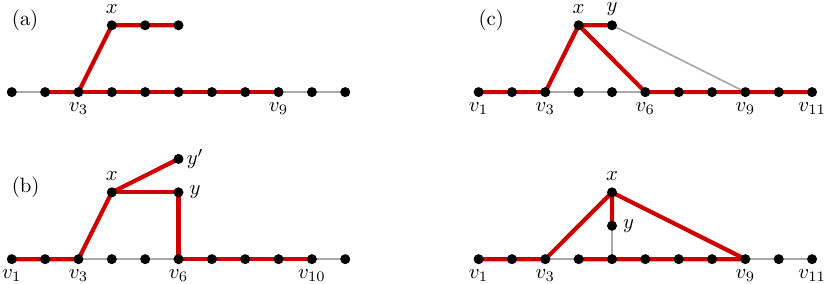}
    \caption{Occurrences of $S_{1,3,6}$ in the proof of Theorem~\ref{t-s136}.}    \label{fig:s136-iii}
\end{figure}

    \begin{claim}\label{clm-s334-lemon-bush}
        If $D$ is some connected component of $G - V(P)$, then $D$ has size at most~$2$. If $D$ has size~$2$ then it is adjacent to exactly two vertices of $P$.
    \end{claim}
    \begin{claimproof}
        Let $D$ be some component of $G - V(P)$ of size at least~$2$. As $G$ is $2$-connected, there exist distinct vertices $x, y$ in $D$ and distinct vertices $v_i, v_j$ in $P$ such that $x$ is adjacent to $v_i$ and $y$ is adjacent to $v_j$. As $P$ is a longest path, $i, j \notin \{1,2, r-1, r\}$. That is, $i,j \in \{3, \ldots, 9\}$.
        
        If $i \in \{4,5\}$, then we find $S_{1,3,6}$ as shown in Figure~\ref{fig:s136-i}\,a), a contradiction. 
        It follows that $i \notin \{4,5\}$ and, symmetrically, $i \notin \{7,8\}$. 
        That is, $i, j \in \{3,6,9\}$. Given that $i \neq j$ and $v_3$ and $v_9$ are symmetric, without loss of generality $i=3$. That is, $x$ is adjacent to $v_3$.

        We now claim that $D$ has size at most~$2$. If $x$ is the end of some path of length~$2$ in~$D$, then as $x$ is adjacent to $v_3$, $G$ contains a $S_{1,3,6}$ subgraph as shown in Figure~\ref{fig:s136-iii}\,a). It follows, if $D$ has size at least~$3$, then $D$ is a star with centre $x$, i.e., there is some path $(y,x,y')$ in~$D$. Recall that $y$ is adjacent to some vertex $v_j$ in $P$.
        As $y$ is the end vertex of some path of length $2$, using that same argument as for $x$, it follows that $y$ is not adjacent to $v_3$ or symmetrically $v_9$. That is $y$ is adjacent to $v_6$. However, now we find that $S_{1,3,6}$ as shown in Figure~\ref{fig:s136-iii}\,b), a contradiction. It follows that $D$ has size at most~$2$.

        It remains to show that if $D$ has size~$2$ then it is adjacent to exactly two vertices of~$P$. As $D$ has size~$2$, $V(D) =  \{x,y\}$.  Recall that there exist distinct vertices $v_i, v_j$ in $P$ such that $x$ is adjacent to $v_i$ and $y$ is adjacent to $v_j$. It follows that $D$ is adjacent to at least two vertices of~$P$. We also recall that $N(D) \subseteq \{v_3,v_6,v_9\}$ and, without loss of generality, $x$ is adjacent to $v_3$. We now claim that $|N(D)| =2$. Suppose, for contradiction, that $N(D) = \{v_3,v_6,v_9\}$. That, is either $v_3, v_6 \in N(x)$ and $v_9 \in N(y)$, or, $v_3, v_9 \in N(x)$ and $v_6 \in N(y)$. In both cases we find some $S_{1,3,6}$, as shown in Figure~\ref{fig:s136-iii}. This is a contradiction, that is, $N(D) \subsetneq \{v_3,v_6,v_9\}$ and so $|N(D)| =2$, thus concluding our proof.
    \end{claimproof}

    \noindent
    It now follows by Claim~\ref{clm-s334-lemon-bush}, that $G$ is an entangled lemon bush with stem set $P$. As $P$ has size~$11$, applying Theorem~\ref{t-solvetangle}, we can solve \stf{} in polynomial time on $(G,T,k)$. Thus concluding the proof of this theorem.
\end{proof}

\subsection{Forbidding Subdivided Claws $\mathbf{S_{2,2,7}}$, $\mathbf{S_{2,3,5}}$, $\mathbf{S_{2,4,4}}$}

Here and in the following, when we consider a path $(v_1, \dots, v_r)$, we denote by $v_{-i}$, for $i \in \{1,\dots, r\}$ the vertex $v_{r-i+1}$, that is $v_r = v_{-1}$. 

We start with the proof of Lemma~\ref{l-s2ab}, which we restate below.

\lemSTwoABFree*
\begin{proof}
    Let $G$ be a $2$-connected $S_{2,a,b}$-subgraph-free graph. Let $P = (v_1, \dots , v_r)$, for some integer $r\geq 3(a+b)$, be a longest path on at least~$3(a+b)$ vertices in $G$, such that $V(P)$ is a vertex cover. 
    Note that in the following we consider all jumps with respect to $P$. Further, note that no jump has length more than~$2$ since $V(P)$ is a vertex cover.
    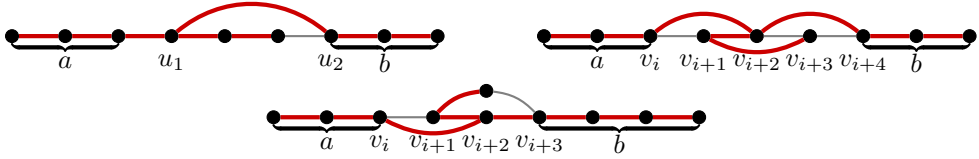
\begin{figure}[!ht]
    \centering
    \input{figures/s2ab}
    \caption{Forbidden jumps due to occurrences of $S_{2,a,b}$ in the proof of Lemma~\ref{l-s2ab} (this figure is a copy of Figure~\ref{f-s2ab-start}).}
    \label{f-s2ab}
\end{figure}

    Suppose there is a jump between two vertices $u_1, u_2 \in \{v_{a+1}, \dots, v_{-(a+1)}\}$ of distance at least~$3$ on~$P$. Then we find $S_{2,a,b}$ as a subgraph, a contradiction, see Figure~\ref{f-s2ab} (top left). 
    Since $P$ is a longest path, we can apply Lemma~\ref{l-adjacencytopath} and obtain that the neighbours of a vertex in $G-V(P)$ on $P$ have distance at least~$2$.
    Together, this implies that any vertex $v$ in $G-V(P)$ has at most~$2$ neighbours in \{$v_{a+1}, \dots, v_{-(a+1)}\}$, and that if $v$ has two such neighbours then they are at distance exactly~$2$ on $P$.
    
    Suppose now there are three jumps, one from $v_i$ to $v_{i+2}$, one from $v_{i+1}$ to $v_{i+3}$, and one from $v_{i+2}$ to $v_{i+4}$, where $v_i \in \{v_{a+1}, \dots, v_{-(a+5)}\}$. Then, we find $S_{2,a,b}$ with $v_{i+2}$ as centre, see Figure~\ref{f-s2ab} (top right), a contradiction.
    Consider now two jumps, one from $v_i$ to $v_{i+2}$ and the other from $v_{i+1}$ to $v_{i+3}$, where $v_i \in \{v_{a+1}, \dots, v_{-(a+4)}\}$. If one of the jumps has length at least~$2$ we find again $S_{2,a,b}$, see Figure~\ref{f-s2ab} (bottom), a contradiction. 
    
    Let $G' = G-\{v_1, \dots, v_{a}, v_{-a}, \dots, v_{-1}\}$, that is, the graph without the ends of $P$. Note that $G'$ might not be $2$-connected. We remove all isolated vertices and iteratively remove all vertices of degree~$1$ from $G'$, unless they are on $P$.
    Let $u_1, u_2 \in V(G')$ be two vertices on $P$ which are connected via a jump of length~$2$. By the observations above, they are at distance~$2$ on $P$. Further, the vertex on $P$ which is at distance~$1$ from both $u_1, u_2$ is not contained in any jump. Hence, $\{u_1, u_2\}$ is a cutset of $G'$ and $u_1, u_2$ are the ends of a lemon in $G'$.

    Suppose now there are vertices $u_1, u_2 \in V(G')$ which are connected via a jump of length~$1$, that is, via an edge, and which are not also connected by a jump of length~$2$. 
    By the observations above, $u_1$ and $u_2$ have distance~$2$ on $P$, that is, $u_1 = v_i$, $u_2 = v_{i+2}$ for some $v_i \in \{v_{a+1}, \dots, v_{-(a+3)}\}$. 
    If there is no jump from $v_{i+1}$ then $v_i$ and $v_{i+2}$ are immediately the ends of a juicy lemon in $G'$ consisting of a single wedge of size 3. If there is a jump from $v_{i+1}$ then applying the properties above this is a jump of length $1$ (i.e., the jump is a single edge) and the jump is to either $v_{i-1}$ or $v_{i+3}$.
    By the observations above, we get that $\{v_{i-1}, v_{i+2}\}$ is a cutset of $G'$ (in the former case) or that $\{v_i, v_{i+3}\}$ is a cutset of $G'$ (in the latter case). In either case, the vertices of the cutset are the ends of a seeded lemon in $G'$.
    We obtain that $G'$ is a lemon bush flowering from a path.

    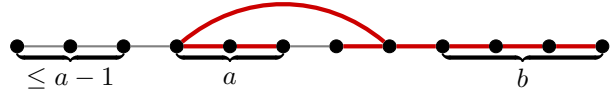
\begin{figure}
        \centering
        \input{figures/s2ab-second}
        \caption{An occurrence of $S_{2,a,b}$ at the end of a path in the proof of Lemma~\ref{l-s2ab}.}
        \label{f-s2ab-secondcase}
    \end{figure}

    We consider now $G$, that is, we add back the ends of the path and all other vertices we deleted.
    Consider a jump starting at $u \in \{v_1, \dots, v_{a}\}$. It can only reach $v_{-a}, \dots, v_{-1}$ or a vertex at distance at most~$a+1$ from~$u$, otherwise, we find $S_{2,a,b}$ as a subgraph, see Figure~\ref{f-s2ab-secondcase}.

    Suppose there is no jump from a vertex in $\{v_1, \dots, v_a\}$ to a vertex in $\{v_{-a}, \dots, v_{-1}\}$. Then, since $|V(P)| \geq 3(a+b)$ and since $G'$ is a lemon bush flowering from a path, we find a cut vertex in $G$, a contradiction to $G$ being $2$-connected.
    
    Hence, there is an edge or a path from $\{v_1, \dots, v_a\}$ to $\{v_{-a}, \dots, v_{-1}\}$. In this case we obtain that $G$ has a cycle of length at least~$a + 3b+2$. By application of Lemma~\ref{l-longestpath}, we find a longest cycle $C$ of $G$ which has length at least~$a+3b+2$. Now all observations for $v_{a+1}, \dots, v_{-(a+1)}$ on $P$ hold everywhere on the cycle. Hence, we may apply the same arguments as for $G'$ and we obtain that $C$ is a lemon bush flowering from a cycle.
\end{proof}

\noindent
We are now ready to prove Theorems~\ref{t-s227}--\ref{t-s244}.

\thmSTwoTwoSevenFree*
\begin{proof}
    Let $(G,T,k)$ be an instance of \stf, where $G$ is a $S_{2,2,7}$-subgraph-free graph, $T$ a terminal set, and $k$ an integer.
    If $G$ is $P_{11}$-subgraph-free, we apply Theorem~\ref{t-p11free}.
    Hence, we may assume that $G$ has a $P_{11}$ subgraph. Let $P$ be a longest path in $G$. Note that $|V(P)| \geq 11$ and $P$ can be found in polynomial time by applying Lemma~\ref{l-longestpath}.
    \begin{claim} $V(P)$ is a vertex cover of $G$.
    \end{claim}
    \begin{claimproof}
    Suppose for a contradiction that there is a component $D$ of $G-V(P)$ of size at least~$2$. 
    $D$ is not adjacent to any of $v_1,v_2, v_{-1}, v_{-2}$, since otherwise we find a longer path in $G$, a contradiction.
    We distinguish four cases based on the length of $P$.

    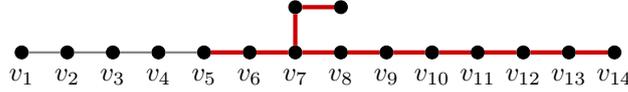
\begin{figure}
        \centering
        \input{figures/s227}
        \caption{An occurrence of $S_{2,4,4}$ in the proof of Theorem~\ref{t-s227}.}
        \label{f-s227}
    \end{figure}
    
    If $|P| \geq 14$, then we find a $S_{2,2,7}$-subgraph, regardless of the neighbours of $D$ in $\{v_3, \dots, v_{-3}\}$, see Figure~\ref{f-s227}. Hence, there is no such component $D$ and $V(P)$ is a vertex cover of $G$.

    If $|P| = 13$, we find $S_{2,2,7}$ as a subgraph as in the case before unless $N(D)=\{v_7\}$. However, since $D$ can only be adjacent to $v_7$, we get that $v_7$ is a cut vertex, a contradiction to the $2$-connectivity of $G$, which we have from Lemma~\ref{l-2con}. Hence, $V(P)$ is a vertex cover.
    
    If $|P| = 12$, similar to the case where $|P| = 13$, we get that $D$ can only be adjacent to $v_6, v_7$. 
    Applying Lemma~\ref{l-adjacencytopath}, we find that $|N(D)|=1$. That is, $N(D)$ consists exactly of a cut-vertex of $G$, which contradicts 2-connectivity of $G$.

    Hence, $V(P)$ is a vertex cover.
    
    If $|P| = 11$, using again the same arguments, $D$ can only be adjacent to $v_5, v_6, v_7$. 
    Applying Lemma~\ref{l-adjacencytopath}, we find that there is a unique vertex $x\in D$ with neighbours in $V(P)$, and so $x$ is a cut-vertex of $G$, which contradicts 2-connectivity of $G$. 
    Hence, $V(P)$ is a vertex cover.
    
    So in every case we get that $V(P)$ is a vertex cover. This proves the claim.
    \end{claimproof}

    \noindent
    If $|P| \leq 26$, the vertex cover number is bounded and we conclude by applying Theorem~\ref{t-vc}.
    Else, we may assume that $|P| \geq 27 = 3*(2+7)$.
    This allows us to apply Lemma~\ref{l-s2ab} and we obtain that $G$ is a lemon bush flowering from a cycle. Since a cycle has treewidth~$2$, we may apply Theorem~\ref{t-lemon-tw2} to solve \stf{} for the instance $(G,T,k)$ in polynomial time.
    This concludes the proof.
\end{proof}

\thmSTwoThreeFiveFree*
\begin{proof}
    Let $(G,T,k)$ be an instance of \stf, where $G$ is a $S_{2,3,5}$-subgraph-free graph, $T$ a terminal set, and $k$ an integer.
    By Lemma~\ref{l-2con}, we may assume that $G$ is $2$-connected.
    If $G$ is $P_{11}$-subgraph-free, we apply Theorem~\ref{t-p11free}. Hence, we may assume that $G$ has a $P_{11}$ subgraph. Let $P$ be a longest path in $G$. Note that $|V(P)| \geq 11$ and $P$ can be found in polynomial time by Lemma~\ref{l-longestpath}.

    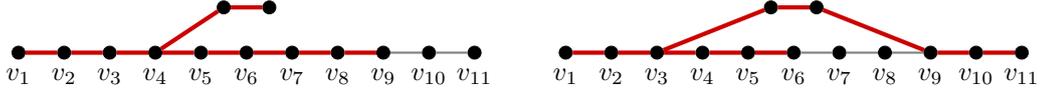
\begin{figure}
        \centering
        \input{figures/s235}
        \caption{Occurrences of $S_{2,3,5}$ in the proof of Theorem~\ref{t-s235}.}
        \label{f-s235}
    \end{figure}
    
    \begin{claim} 
        $V(P)$ is a vertex cover.
    \end{claim}
    
    \begin{claimproof}
    Suppose for a contradiction that there is a component $D$ of size at least~$2$ in $G-V(P)$. Then, $D$ is not connected to $v_1,v_2,v_{-2},v_{-1}$, since otherwise there is a longer path in $G$, a contradiction.
    If $D$ is adjacent to any vertex in $\{v_4, \dots, v_{-4}\}$, we find $S_{2,3,5}$ with the neighbour of $D$ as centre, see Figure~\ref{f-s235} (left), a contradiction.
    Hence, $D$ can be adjacent to only $v_3, v_{-3}$ on $P$.
    Recall that due to the $2$-connectivity of $G$, $D$ has at least two neighbours on $P$.
    Thus, $D$ is adjacent to both $v_3, v_{-3}$.
    We find $S_{2,3,5}$ with $v_3$ as centre, see Figure~\ref{f-s235} (right), a contradiction.
    It follows that no component of $G-V(P)$ can have size at least~$2$, so $V(P)$ is a vertex cover.
    \end{claimproof}

    \noindent
    If $|P| \leq 23$, the vertex cover number is bounded and we conclude by applying Theorem~\ref{t-vc}.
    Hence, we may assume that $|P| \geq 24 = 3*(3+5)$. This allows us to apply Lemma~\ref{l-s2ab} and we obtain that $G$ is a lemon bush flowering from a cycle. Since a cycle has treewidth~$2$, we may apply Theorem~\ref{t-lemon-tw2} to solve \stf{} for the instance $(G,T,k)$ in polynomial time.
\end{proof}

\thmSTwoFourFourFree*
\begin{proof}
    Let $(G,T,k)$ be an instance of \stf, where $G$ is a $S_{2,4,4}$-subgraph-free graph, $T$ a terminal set, and $k$ an integer. If $G$ is $P_{11}$-subgraph-free, we apply Theorem~\ref{t-p11free}.
    Let $P$ be a longest path of $G$ which can be found in polynomial time due to Lemma~\ref{l-longestpath}. Note that $|P| \geq 11$.

    \begin{figure}
        \centering
        \input{figures/s244-longp}
        \caption{Occurrences of $S_{2,4,4}$ in the proof of Theorem~\ref{t-s244}.}
        \label{f-s244-longp}
    \end{figure}
    
    We first consider the case where $|P| \geq 12$.
    \begin{claim}
        $V(P)$ is a vertex cover of $G$.
    \end{claim}
    \begin{claimproof}
    Suppose for a contradiction that there is a component $D$ of size at least~$2$ in $G-V(P)$. Note first that $D$ cannot be adjacent to any of $v_1, v_2, v_{-1}, v_{-2}$, since otherwise we find a longer path.
    If $D$ is adjacent to some vertex in $\{v_5, \dots, v_{-5}\}$, we find $S_{2,4,4}$ as a subgraph in~$G$, see Figure~\ref{f-s244-longp} (top left), a contradiction.
    Hence, $D$ is adjacent to a vertex in $\{v_3,v_4, v_{-3}, v_{-4}\}$.
    Recall that due to the $2$-connectivity of $G$, $D$ has two neighbours on $P$.
    Suppose first that $D$ is adjacent to $v_3$, then by Lemma~\ref{l-adjacencytopath}, $D$ is not adjacent to $v_4$, so it is adjacent to $v_{-3}$ or $v_{-4}$. In either case, we find $S_{2,4,4}$ with $v_3$ as centre, see Figure~\ref{f-s244-longp} (top right), a contradiction.
    Similarly, if $D$ is adjacent to $v_4$, it has to be adjacent to $v_{-3}$ or $v_{-4}$ and we find again $S_{2,4,4}$ with centre~$v_4$, see Figure~\ref{f-s244-longp} (bottom), a contradiction.
    It follows that $V(P)$ is a vertex cover of~$G$.
    \end{claimproof}

    \noindent
    If $|P| \leq 23$, the vertex cover number is bounded and we conclude by applying Theorem~\ref{t-vc}.
    Hence, we may assume that $|P| \geq 24 = 3*(4+4)$.
    This allows us to apply Lemma~\ref{l-s2ab} and we obtain that $G$ is a lemon bush flowering from a cycle. Since a cycle has treewidth~$2$, we may apply Theorem~\ref{t-lemon-tw2} to solve \stf{} for the instance $(G,T,k)$ in polynomial time in the case where we find a path of length at least~$12$ in $G$.

    \begin{figure}
        \centering
        \input{figures/s244-concomp}
        \caption{Occurrences of $S_{2,4,4}$ in the proof of Theorem~\ref{t-s244}.}
        \label{f-s244-concomp}
    \end{figure}
    
    \medskip
    \noindent
    We now consider the case where $P = P_{11}$ is a longest path in $G$.
    If all components of $G-V(P)$ have size~$1$, we get that $V(P)$ is a vertex cover of bounded size and we conclude by applying Theorem~\ref{t-vc}.
    Hence, we may assume that there is a component~$D$ of size at least~$2$ in $G-V(P)$. 

    \begin{claim}\label{cl-s244-adj48}
        Every component of $G-V(P)$ of size at least~$2$ is adjacent to exactly $v_4$ and~$v_8$.
    \end{claim}
    \begin{claimproof}
    Let $D$ be a component of $G-V(P)$ of size at least~$2$.
    Note first that $D$ cannot be adjacent to any of $v_1, v_2, v_{10}, v_{11}$, since otherwise there is a longer path in $G$.
    If $D$ is adjacent to some vertex in $\{v_5, v_6,v_7\}$, we find $S_{2,4,4}$ as a subgraph in $G$, see Figure~\ref{f-s244-concomp} (top left), a contradiction.
    Hence, $D$ is adjacent only to vertices in $\{v_3,v_4, v_{8}, v_{9}\}$.
    
    Recall that due to the $2$-connectivity of $G$, $D$ has at least two neighbours on $P$.
    Further, due to Lemma~\ref{l-adjacencytopath}, $D$ cannot be adjacent to $v_3$ and $v_4$ or to $v_8$ and $v_9$ at the same time.
    Suppose first that $D$ is adjacent to $v_3$, then it is adjacent to $v_{8}$ or $v_{9}$. In either case, we find $S_{2,4,4}$ with $v_3$ as centre, see Figure~\ref{f-s244-concomp} (top right), a contradiction.
    Similarly, if $D$ is adjacent to $v_4$, it has to be adjacent to $v_{8}$ or $v_{9}$. If $D$ is adjacent to $v_4$ and $v_9$ we find again $S_{2,4,4}$ with centre $v_4$, see Figure~\ref{f-s244-concomp} (bottom), a contradiction.
    Thus, we may assume that every component $D$ of $G-V(P)$ of size at least~$2$ is adjacent to exactly $v_4$ and $v_8$.
    \end{claimproof}

    \begin{claim}\label{cl-s244-comps-p4free}
        The components of $G-V(P)$ are $P_4$-subgraph-free. 
    \end{claim}
    \begin{claimproof}
    Suppose for a contradiction that there is a component $D$ in $G-V(P)$ which contains a $P_4$ $(u_1,u_2,u_3,u_4)$ as a subgraph. Since $D$ has size more than~$2$, by Claim~\ref{cl-s244-adj48}, the neighbours of $D$ on~$P$ are exactly $v_4$ and $v_8$.
    If $u_1$ is adjacent to $v_4$, we find a longer path $(u_4,u_3,u_2,u_1,v_4, \dots,  v_{11})$, a contradiction.
    Hence, $u_1$ is not adjacent to $v_4$. The cases where $u_1$ is adjacent to $v_8$ and $u_4$ is adjacent to $P$ follow by symmetry.  
    Thus, some other vertex in~$D$ is adjacent to $P$. We may assume without loss of generality that $u_2$ is adjacent to $v_4$.

    Since $G$ is $2$-connected, there is a path from $u_1$ to $u_3$ not using $u_2$. If the path contains a vertex from $P$, then let $w$ be the vertex in $D$ adjacent to $P$. We get that $(w , \dots , u_1 , \dots , u_3)$ is a $P_4$ in $D$ whose end is adjacent to $P$, a contradiction. Hence, the path contains no vertex from~$P$. If the path contains $u_4$, we find a cycle of length at least~$4$ in $D$ and thus, $u_2$ is the endvertex of a $P_4$ in $D$. Thus, $u_2$ may not be adjacent to a vertex on~$P$, a contradiction. Hence, the path reaches $u_3$ not via $u_4$. Then, $(u_4, u_3 , \dots , u_1, u_2)$ is another path of length at least~$4$ ending at $u_2$, again a contradiction. Thus we obtain a contradiction in all cases, and the claim follows.
    \end{claimproof}

    \noindent
    Note that Claim~\ref{cl-s244-comps-p4free} together with Observation~\ref{o-structure-p4sfree} implies that all components of $G-V(P)$ are stars, triangles, edges or single vertices.

    We now analyse the structure of stars and triangles.
    Observe that if a leaf of a star or a triangle is adjacent to both $v_4$ and $v_8$, we find $S_{2,4,4}$, see Figure~\ref{f-s244-neighboursstar}, a contradiction.
    Hence, no leaf of a star and no vertex of a triangle is adjacent to both $v_4$ and $v_8$.

    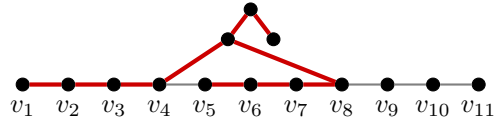
\begin{figure}
        \centering
        \input{figures/s244-neighboursstar}
        \caption{An occurrence of $S_{2,4,4}$ in the proof of Theorem~\ref{t-s244} if the leaf of a star is adjacent to both $v_4$ and $v_8$.}
        \label{f-s244-neighboursstar}
    \end{figure}

    Consider now a triangle $D$ in $G-V(P)$, consisting of $u_1, u_2 , u_3$. We know that each vertex of the triangle is adjacent to at most one of~$v_4$ and $v_8$. 

    \begin{claim}\label{claim-S244-notriangles}
        We can reduce to a case where $D \cup \{v_4, v_8\}$ is a juicy or seeded wedge. 
    \end{claim}
    \begin{claimproof}
        By the $2$-connectivity of $G$ we know that at least one of $u_1, u_2, u_3$, say $u_1$, is adjacent to~$v_4$ and at least one, say $u_2$, to $v_8$.
        We distinguish between the case where $u_3$ is adjacent to none of $v_4, v_8$ and the case where $u_3$ is adjacent to one of them. Recall that any vertex in the triangle is adjacent to at most one of them, hence these are all cases to consider.

        Suppose first that $u_3$ is not adjacent to any of $v_4, v_8$. 
        Then, $D\cup\{v_4,v_8\}$ is a juicy wedge with ends $v_4, v_8$ 
        witnessed by the path-decomposition $\{v_4,v_8,u_1\},\{v_8,u_1,u_2\},\{u_1,u_2,u_3\}$.

        Consider now the case where $u_3$ has a neighbour in $\{v_4, v_8\}$. We assume that it is $v_4$, the other case follows by symmetry.
        Note first that if either of $u_1, u_3$, say $u_1$, is not a terminal we may reduce to solving \stf{} on $G-u_1$ by applying Lemma~\ref{l-benjamins}.

        We may therefore assume that both $u_1$ and $u_3$ are terminals.
        We will show that $u_1u_3$ can safely be deleted by proving that any optimal solution $F$ which uses this edge entails the existence of a solution $F'$ of equal weight which does not.
        First suppose $u_1v_4 \in F$. If $u_1u_3 \in F$ then deleting $u_1u_3$ and introducing $u_3v_4$ is a solution of equal cost. A symmetric argument holds if $u_3v_4 \in F$.
        Now suppose $u_1v_4,u_3v_4\notin F$. Applying Lemma~\ref{l-adj-terminals} we may assume that $u_1$ and $u_3$ must connect to some vertex outside $D$ in $F$, and therefore they are both in the same component of $F$ as $u_2$. If $u_1u_3 \in F$, then deleting $u_1u_3$ from $F$ and introducing $u_1u_2$ or $u_3u_2$ (whichever was not already in $F$) yields a solution $F'$ of equal weight. The claim follows.
    \end{claimproof}

    \noindent   
    We now consider the star components. 
    \begin{claim}\label{cl-s244-starwedge}
        Let $D$ be a component of $G-V(P)$ which is a star. Then, $D \cup \{v_4, v_8\}$ is a juicy wedge with ends $v_4, v_8$. 
    \end{claim}
    \begin{claimproof}
        Recall that by definition $D \cup \{v_4, v_8\}$ is a juicy wedge if $G[D \cup \{v_4, v_8\}] \cup \{v_4v_8\}$ has treewidth~$2$. We show this by giving a tree decomposition.
        We make a bag $B$ containing the vertices $v_4$ and $v_8$. Let $c$ be the centre of the star. We add a bag $B_c$ containing $v_4, v_8, c$ and make it adjacent to $B$.
        Now, recall that every leaf is adjacent to only one of $v_4$ and~$v_8$. For a leaf $u$ adjacent to $v_4$, we add a bag containing $u, v_4, c$ and make it adjacent to $B_c$. Similarly, for a leaf $u$ adjacent to $v_8$, we add a bag containing $u, v_8, c$ and make it adjacent to~$B_c$.
        Hence, $G[D \cup \{v_4, v_8\}] \cup \{v_4v_8\}$ has treewidth~$2$ and thus, $D \cup \{v_4, v_8\}$ is a juicy wedge. 
    \end{claimproof}

    \noindent
    Now for any component $D$ of $G-V(P)$ which is adjacent to $v_4$ and $v_8$, it holds that $D \cup \{v_4, v_8\}$ is a juicy or seeded wedge:
    \begin{itemize}
        \item If $D$ is a star this holds by Claim~\ref{cl-s244-starwedge}.
        \item If $D$ is a triangle then applying Claim~\ref{claim-S244-notriangles}, it can be transformed to satisfy that $D \cup \{v_4, v_8\}$ is a juicy or seeded wedge.
        \item If $|D|=2$ then definitionally $D \cup \{v_4, v_8\}$ is a seeded wedge.
        \item If $|D|=1$ then $D \cup \{v_4, v_8\}$ is a (trivially) juicy wedge.
    \end{itemize}
    Now the set of all components of $G-V(P)$ adjacent to $v_4$ and $v_8$ is a lemon with ends $v_4$ and~$v_8$. We call its vesicle set $D$.

    We let $G'' = G[V(P)] \cup \{v_4v_8\}$. Then, $G' = G[V(P) \cup D]$ is a lemon bush flowering from~$G''$ with stem set $V(P)$.
    Further, all vertices in $V(G) - (V(P)  \cup D)$ are adjacent only to vertices in $P$ and are thus a tangle set. Hence, we obtain that $G$ is an entangled lemon bush with $V(P)$ its stem set of size~$11$, and we conclude by applying Theorem~\ref{t-solvetangle}.
    \end{proof}

\subsection{Forbidding a Subdivided Claw $\mathbf{S_{3,3,4}}$}

We first prove Lemma~\ref{l-s3ab}, which we restate below.

\lemSThreeABFree*
\begin{proof}
    Let $P = (p_1, \ldots, p_\ell)$ be some longest path in $G$ where $\ell \geq 3b+2a-4$.
    We first reason about the \textit{middle} of $P$.    
    Let $P'= P - \{p_1,\ldots,p_a,p_{\ell-a+1},\ldots,p_\ell\}$ and let $G'$ be the graph obtained from $G$ by removing the vertices $\{p_1,\ldots,p_a,p_{\ell-a+1},\ldots,p_\ell\}$ together with every connected component $C$ of $G - V(P)$ such that $N(C) \not\subseteq V(P')$. It follows that, if $C$ is some connected component of $G' - V(P')$, then $C$ had no neighbours in $P-V(P')$. Note that the following arguments for $G'$ only use that $\ell \geq 3b$.

    We now claim that $G'$ is a citrus bush flowering from a path such that every citrus is either a lemon or has size at most~$6$. To prove this, we explicitly construct a stemset $S \subseteq V(P')$ and verify that it satisfies the necessary properties.

    Towards defining our stemset, we say a vertex $v$ separates $P'$ if either $v \in \{p_{a+1}, p_{\ell-a}\}$ or $p_{a+1}$ and $p_{\ell-a}$ are contained in distinct components in $G'-v$. As $P'$ is a path, any vertex that separates $P'$ must also be contained in $P'$. We now define our stemset $S = \{s_1, \ldots, s_{|S|}\} \subseteq V(P')$ to be the set of vertices that separate $P'$. Further, we order these such that $s_1 := p_{a+1}$, $s_{|S|} := p_{\ell-a}$ and for every $i \in \{2, \ldots, {|S|-1}\}$, the vertex $s_i$ lies on that subpath of $P'$ between $s_{i-1}$ and $s_{i+1}$. Note that by definition every component in $G' - S$ is adjacent to at most  two vertices in $S$. 

    We now show that $S$ is indeed a stemset. To this end, we show that, for every connected component $C$ of $G' - S$, there exists an index $i$ such that  $V(C) \cup \{s_i, s_{i+1}\}$ is a wedge with ends $\{s_i, s_{i+1}\}$. Moreover, either $V(C) \leq 2$, in which case this wedge is juicy or seeded, or this is the unique wedge with ends $s_i$, $s_{i+1}$ and $V(C) \cup \{s_i,s_{i+1}\}$ has size at most~$6$.

    The following claims will useful for our reasoning.

    \begin{claim}\label{clm-far-disjoint-path}
        If there exist indices $i,j \in \{a+1,\ldots,\ell-a\}$ such that $j \geq i +4$ and there exists some path $Q$ between $p_i$ and $p_j$ such that $V(Q) \cap P' = \{p_i, p_j\}$, then $G$ contains $S_{3,a,b}$.        
    \end{claim}
    \begin{claimproof}
        Suppose there exist indices $i$, $j$ and a path $Q$ as described in the claim statement. In what follows, we assume without loss of generality that $Q$ is the single edge $p_ip_j$. Indeed, if $Q$ has length at least~$2$, then, since $V(Q) \cap P' = \{p_i, p_j\}$ our arguments still hold after replacing the edge $p_ip_j$ by the path $Q$. We now show that $G$ contains $S_{3,a,b}$ as a subgraph.

        \begin{figure}
            \centering
            \includegraphics[width=0.4\linewidth, page=11]{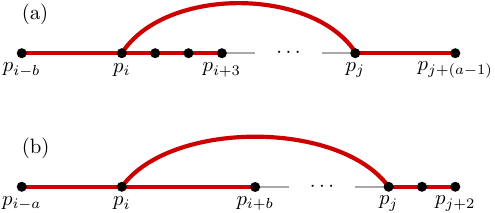}
            \caption{Occurrences of $S_{3,a,b}$ in the proof of Claim~\ref{clm-far-disjoint-path} of Lemma~\ref{l-s3ab}.}
            \label{fig:far-disjoint-path}
        \end{figure}
    
    	We consider two cases according to position of $p_i$ and $p_j$ along the path $P$. If $i \ge b+1$, then we find that $S_{3,a,b}$ shown in Figure~\ref{fig:far-disjoint-path} a). By symmetry, the same holds if $j \le \ell-b$. It remains to consider the case where $i \leq b$ and $j \geq \ell-b+1$. Given that $\ell \geq 3b$, in this case $j \geq i+b+1$. In this case we find that $S_{3,a,b}$ shown in Figure~\ref{fig:far-disjoint-path} b), thus concluding the proof of this claim.
    \end{claimproof}

    \noindent
    Before reasoning about the components of $G' - S$ we first show the following properties of $G' - V(P')$.

    \begin{claim} \label{claim-component-adjacencies}
        Either $G$ contains $S_{3,a,b}$ or for every connected component $C \subseteq G' - V(P')$, $|N(V(C))| = 2$ and $|V(C)| \leq 2$. If $C$ has size~$2$, then $C$ contains two distinct vertices that are adjacent to two distinct vertices of $p'$.
    \end{claim}
    \begin{claimproof}
        Let $C$ be some connected component of $G' - V(P')$. 
        
        We first claim that either $G$ contains $S_{3,a,b}$ as a subgraph, or, $C$ is adjacent to exactly $2$ vertices in $P'$. Suppose that $C$ is adjacent to at least~$3$ vertices in $P'$, that is, there exist indices $i$, $j$, $k$ such that $p_i, p_j, p_k \in N(V(C))$. By Lemma~\ref{l-adjacencytopath}, $j \geq i+2$ and $k \geq j+2$. It follows that $k-i \geq 4$. As the path from $p_i$ to $p_k$ via $C$ is disjoint from $P'$, by Claim~\ref{clm-far-disjoint-path}, $G$ contains $S_{3,a,b}$. That is, either $G$ contains $S_{3,a,b}$ or $C$ is adjacent to at most~$2$ vertices in $P'$. 

        We now claim that $C$ is adjacent to at least~$2$ vertices in $P'$. Recall that $G$ was $2$-connected. It follows that $C$ is adjacent to two distinct vertices in $P$ and if $C$ has size at least two then $P$ is adjacent to two distinct vertices in $C$. By definition of $G'$, $C$ is not adjacent to any vertices in $P - V(P')$. That is $C$ is adjacent to at least~$2$ vertices in $P'$. It follows that $C$ is adjacent to exactly two vertices in $P'$, call these $p_i$ and $p_j$. Further, if $C$ has size at least~$2$, then there exists distinct vertices $x$, $y$ such that $x$ is adjacent to $p_i$ and $y$ is adjacent to $p_j$.

        \begin{figure}
            \centering
            \includegraphics[width=0.4\linewidth, page=12]{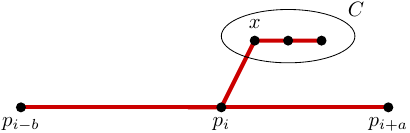}
            \caption{An occurrence of $S_{3,a,b}$ in the proof of Claim~\ref{claim-component-adjacencies} of Lemma~\ref{l-s3ab}.}
            \label{fig:s3ab-component-adjacencies}
        \end{figure}

        We now claim that if $C$ has size at least~$3$, then $G$ contains $S_{3,a,b}$. Suppose $C$ has size at least~$3$. Consider those vertices $x$, $y$ such that $x$ is adjacent to $p_i$ and $y$ is adjacent to $p_j$. If $x$ is the end of some path of length $2$ in $C$, then we find some $S_{3,a,b}$ as shown in Figure~\ref{fig:s3ab-component-adjacencies}. Note, that in the case depicted $i \geq b+1$. If $i \leq b$, then as $\ell \geq 3b$ we find that $i \leq \ell-b$ and so there is a symmetric $S_{3,a,b}$ found. 
        
        We now claim that either $x$ or $y$ is the end of some path of length $2$ in $C$. As $C$ is a connected component, there is a path from $x$ to $y$ in $C$. If this path has length at least~$2$, then we find that both $x$ and $y$ are the end vertex of a path of length $2$ in $C$. That is, we now assume that this path consists of the single edge $xy$. By assumption, $C$ has size at least~$3$ meaning there is some vertex $y' \in V(C)$ such that $y' \neq x \neq y$. Without loss, $y'$ is adjacent to $y$. It follows that $x$ is the end vertex of the path $(x,y,y')$ of length $2$ in $C$ and so $G$ contains $S_{3,a,b}$, thus concluding the proof of this claim.
    \end{claimproof}

    We are now ready to show that $S$ is indeed a stemset. Recall that by construction of $S$ every component in $G' - S$ is adjacent to exactly two vertices in $S$. From this we define the set $\mathcal{L}_i$ for every $i \in \{1, \ldots, |S|-1\}$. For every $L \in \mathcal{L}_i$, $L$ is a set of vertices such that $L = V(C) \cup \{s_i, s_{i+1}\}$ where $C$ is some connected component in $G'-S$ such that $N(V(C)) = \{s_i, s_{i+1}\}$. We let $V_i$ denote the vertices of these components, i.e., $V_i = \bigcup_{L \in \mathcal{L}_i} L$. 

    Note that if $|L| = 1$, then the graph $G'[L] \cup \{s_i,s_{i+1}\}$ has treedepth at most~$2$ and so $L$ is a juicy wedge. If $|L| = 2$, then by definition $L$ is a seeded wedge. It follows that if every $L \in \mathcal{L}_i$ has size at most~$4$, then $\mathcal{L}_i$ is a lemon.

    \begin{figure}
        \centering
        \includegraphics[width=\linewidth, page=16]{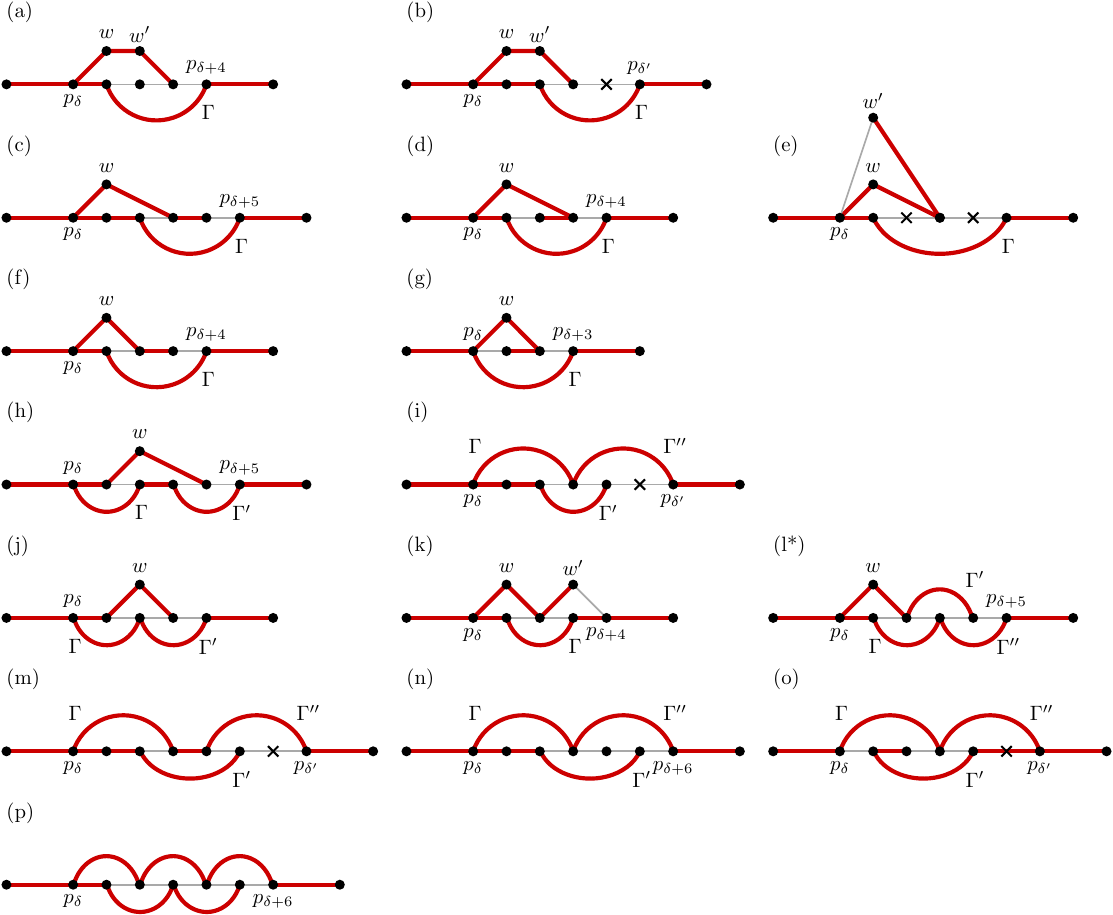}
        \caption{Occurrences of $S_{3,a,b}$ in the proof of  Claim~\ref{clm-G'-citrus} of Lemma~\ref{l-s3ab}.}
        \label{fig-S334-analysis}
    \end{figure}

    \begin{claim}\label{clm-G'-citrus}
        Either $G$ contains $S_{3,a,b}$ or, for every $i \in \{1, \ldots, |S|-1\}$, $\mathcal{L}_i$ is a citrus with ends $\{s_i, s_{i+1}\}$ such that either $\mathcal{L}_i$ is a lemon or $|V_i| \leq 6$.
    \end{claim}
    \begin{claimproof}
        Where $G$ contains $S_{3,a,b}$, we will often highlight to the reader which of those cases from Figure~\ref{fig-S334-analysis} we are in. 
        
        We consider it useful to explain how Fig. 17 presents different cases. We shall use 17b as an illustrative example. Here, the cross denotes a position where a vertex may or may not appear without affecting the existence of our forbidden subgraph. That is, Fig. 17b is actually representing two cases, one in which there is a vertex between $p_{\delta+3}$ and $p_{\delta'}$ and another in which there is no such vertex and $p_{\delta+3}$ and $p_{\delta'}$ are adjacent.  
        
        To ensure that we can find that $S_{3,a,b}$ illustrated, when applying these cases we ensure that $\delta, \delta' \in \{a+1, \ldots, \ell-a\}$. This implies that in $G$ both the paths $(p_1,\ldots,p_{\delta})$ and $(p_{\delta'},\ldots,p_{\ell})$ have length at least $a$. Further, combining the fact that $\ell \geq 3b$ with the fact that in each case $\delta' \leq \delta+6$, at least one of the paths $(p_1,\ldots,p_{\delta})$ and $(p_{\delta'},\ldots,p_{\ell})$ have length at least $b$. We further ensure that $\Gamma$ contains only the vertices $p_{\delta}$, $p_{\delta'}$ from $P'$ and $w, w' \notin V(P')$. The other diagrams in Figure~\ref{fig-S334-analysis} can be interpreted similarly.

        We will frequently use that, by definition, $s_i$ and $s_{i+1}$ separate $P'$. Recall that a vertex $v$ separates $P'$ if either $v \in \{p_{a+1}, p_{\ell-a}\}$ or $P'$ is disconnected in $G'- \{v\}$. Further, by definition of $S$, no vertex on the subpath of $P'$ between $s_i$ and $s_{i+1}$ separates $P'$. 

        Due to this, we use the following observation repeatedly. If some vertex $p_r$ does not separate $P'$, then there exist indices $r'$, $r''$ such that $r'< r < r''$ and a path $Q$ from $p_{r'}$ to $p_{r''}$ such that $V(Q) \cap V(P') = \{p_{r'},p_{r''}\}$. We call such a path a bypass path for $p_r$ and we say that it has ends $p_{r'}$ and $p_{r''}$. In all cases we consider that bypass path that maximises $r''$. If $r''-r' \geq 4$, then by Claim~\ref{clm-far-disjoint-path}, $G$ contains $S_{3,a,b}$, that is, we assume $r''-r' \leq 3$ and so $r'' \leq r+3$.

        We are now ready to prove our claim. Consider some $i \in \{1, \ldots, |S|-1\}$, we distinguish between the following cases depending on the components of $G'[V_i \setminus V(P')]$. Note in each case we assume that we are not contained in any previous case.

        \medskip
        
        \noindent \textbf{Case $1$:} \emph{$G'[V_i \setminus V(P')]$ contains some component $C$ of size at least~$2$.}
        \medskip
    \noindent
    In this case we will show that either $G$ contains some $S_{3,a,b}$ or $\mathcal{L}_i$ is a lemon with ends $p_j$ and $p_{j+3}$.
                
    By applying Claim~\ref{claim-component-adjacencies} together with Lemma~\ref{l-adjacencytopath}, we conclude that either $G$ contains $S_{3,a,b}$, or the following holds: the component $C$ consists of exactly two vertices, call these ${x,y}$; the neighbourhood of $C$ has size $2$; and there is some $j \in \{a+1,\ldots,\ell-(a+3)\}$ such that $x$ is adjacent to $p_j$ and $y$ is adjacent to $p_{j+3}$.
    
    We are now ready to show that either $G$ contains some $S_{3,a,b}$ or $\mathcal{L}_i$ is a lemon with ends $p_j$ and $p_{j+3}$. Suppose $G$ is $S_{3,a,b}$-subgraph-free. We first show that $\mathcal{L}_i$ is a citrus with ends $p_j$, $p_{j+3}$. Towards this, we claim that both $p_j$ and $p_{j+3}$ separate $P'$. 
        
    Suppose for contradiction that $p_{j+3}$ does not separate $P'$. That is, there is some bypass path for $p_{j+3}$. By definition, this either has end vertices $(p_{j+1}, p_{j+4})$, $(p_{j+2}, p_{j+5})$ or $(p_{j+2}, p_{j+4})$. As $N(C) = \{p_j,p_{j+3}\}$, this bypass path does not contain either $x$ or $y$. That is, if it has ends $(p_{j+1}, p_{j+4})$ then we find that $S_{3,a,b}$ depicted in Figure~\ref{fig-S334-analysis}\,(a) and if it has either the ends $(p_{j+2}, p_{j+5})$ or $(p_{j+2}, p_{j+4})$, then we find that $S_{3,a,b}$ depicted in Figure~\ref{fig-S334-analysis}\,(b). This contradicts the fact that $G$ is $S_{3,a,b}$-subgraph-free and so $p_{j+3}$ separates $P'$. Symmetrically the same holds for $p_{j}$ and so $p_{j}$ separates $P'$. Note that, as there is a path from $p_{j}$ to $p_{j+3}$ via $C$, neither $p_{j+1}$ or $p_{j+2}$ separate $P'$. It follows that $\mathcal{L}_i$ is a citrus with ends $p_j$, $p_{j+3}$, i.e., $p_j = s_i$ and $p_{j+3} = s_{i+1}$.

    We now claim that $\mathcal{L}_i$ is a lemon. We do this by showing that every wedge in $\mathcal{L}_i$ has size at most~$4$. Let us consider some $L \in \mathcal{L}_i$. By definition of $\mathcal{L}_i$, $L = V(D) \cup \{p_j, p_{j+3}\}$ where $D$ is some connected component in $G' -S$ such that $N(V(D)) = \{p_j, p_{j+3}\}$. If $D$ does not contain any vertex from $P'$, then $D$ is also a connected component in $G' - V(P')$. By Claim~\ref{claim-component-adjacencies}, it follows that $D$ has size at most~$2$ and so $L$ has size at most~$4$ as we sought to show.

    We now consider the case where $D$ contains some vertex from $P'$. As $D$ is some connected component in $G' -\{p_j, p_{j+3}\}$ and $p_j$ and $p_{j+3}$ separate $P'$, it follows that $V(D) \cap V(P') = \{p_{j+1}, p_{j+2}\}$. Note that, this also implies that either $L$ has size at most~$4$, as desired, or $D$ contains some vertex $v \notin V(P')$. We claim, if $D$ contains such a vertex $v$, then $G$ contains some $S_{3,a,b}$. As $G$ is $S_{3,a,b}$-subgraph free, this will allow us to conclude that $L$ has size at most $4$.
            
    Suppose that $D$ contains some vertex $v \notin V(P')$. As $D$ is a connected component and $V(D) \cap V(P') = \{p_{j+1}, p_{j+2}\}$, $v$ is adjacent to at least one of $p_{j+1}$ and $p_{j+2}$. It then follows from Claim~\ref{claim-component-adjacencies} and Lemma~\ref{l-adjacencytopath} that either $N(v) = \{p_j, p_{j+2}\}$ or $N(v) = \{p_{j+1}, p_{j+3}\}$. In both cases we find that $S_{3,a,b}$ depicted in Figure~\ref{fig-S334-analysis}\,(g), where $w=v$ and $\Gamma$ is that path via $C$. As $G$ is $S_{3,a,b}$-subgraph free, it follows that $D$ does not contain any such vertex $v$. Hence, if $D$ contains some vertex from $P'$, then $V(D) = \{p_{j+1}, p_{j+2}\}$ and so $L$ has size at most $4$, as we sought to show. Recall from the previous paragraph, if $D$ does not contain some vertex from $P'$, then $L$ has size at most $4$.    
    That is, either $G$ contains $S_{3,a,b}$ or every wedge $L \in \mathcal{L}_i$ has size at most~$4$, and so $\mathcal{L}_i$ is a lemon with ends $s_i$ and $s_{i+1}$ as we sought to prove.
        
    \medskip
    
    \noindent \textbf{Case $2$:} There exist distinct vertices $x,y \in V_i \setminus V(P')$ such that $N(x) = N(y)$.
    
    \medskip
    \noindent
    As we are not in Case~$1$, both $x$ and $y$ are isolated in $G'-V(P')$. Combining this with Claim~\ref{claim-component-adjacencies} and Lemma~\ref{l-adjacencytopath}, either $N(x) = N(y) = \{p_{j},p_{j+2}\}$, for some $j \in \{a+1,\ldots,\ell-(a+2)\}$, or $N(x) = N(y) = \{p_{j},p_{j+3}\}$, for some $j \in \{a+1,\ldots,\ell-(a+3)\}$. In this case we will show that either $G$ contains some $S_{3,a,b}$ or $\mathcal{L}_i$ is a lemon whose ends correspond to the vertices of $N(x)$.
    
    Suppose that $G$ is $S_{3,a,b}$-subgraph-free and let $j'$ be that index such that $N(x) = N(y) = \{p_{j},p_{j'}\}$. We will now show that $\mathcal{L}_i$ is a lemon with ends $p_{j}$ and $p_{j'}$. Towards this, we claim that both $p_j$ and $p_{j'}$ separate $P'$.

    Suppose, for the sake of contradiction, $p_{j'}$ does not separate $P'$. It follows that there is some bypass path for $p_{j'}$ that either has end vertices $(p_{j'-1},p_{j'+2})$, $(p_{j'-2},p_{j'+1})$ or $(p_{j'-1},p_{j'+1})$. We will now show that in each of these cases we find some $S_{3,a,b}$, thus contradicting that $S_{3,a,b}$-subgraph-free. 

    First suppose that $j' = j+3$. If the bypass path for $p_{j+3}$ has end vertices $(p_{j+2},p_{j+5})$ then we find that $S_{3,a,b}$ depicted in Figure~\ref{fig-S334-analysis}\,(c), if it has end vertices $(p_{j+1},p_{j+4})$ then we find that $S_{3,a,b}$ depicted in Figure~\ref{fig-S334-analysis}\,(d), otherwise, it must have end vertices $(p_{j+2},p_{j+4})$ in which case we find that $S_{3,a,b}$ depicted in Figure~\ref{fig-S334-analysis}\,(e). That is, if $j' = j+3$ and there is some bypass path for $p_{j+3}$, then $G$ contains $S_{3,a,b}$, a contradiction. It follows that if $j' = j+3$, then $p_{j'}$ and symmetrically $p_j$ separate $P'$. 

    This leaves the case where $j' = j+2$. If the bypass path for $p_{j+2}$ has end vertices $(p_{j+1},p_{j+4})$, then we find that $S_{3,a,b}$ depicted in Figure~\ref{fig-S334-analysis}\,(f), if it has end vertices $(p_{j},p_{j+3})$, then we find that $S_{3,a,b}$ depicted in Figure~\ref{fig-S334-analysis}\,(g),  otherwise, it must have end vertices $(p_{j+1},p_{j+3})$ in which case we find that $S_{3,a,b}$ depicted in Figure~\ref{fig-S334-analysis}\,(e). That is, if $j' = j+2$ and there is some bypass path for $p_{j+2}$, then $G$ contains $S_{3,a,b}$, a contradiction. It follows that if $j' = j+2$, then $p_{j'}$ and symmetrically $p_j$ separate $P'$.

    Combining the above two paragraphs, both $p_j$ and $p_{j'}$ separate $P'$. It follows that $\mathcal{L}_i$ is a citrus with ends $p_j$ and $p_{j'}$. We are now ready to show that $\mathcal{L}_i$ is indeed a $\mathcal{L}_i$ is a lemon. We do this by showing that every wedge in $\mathcal{L}_i$ has size at most~$4$. Recall that every $L \in \mathcal{L}_i$ is in the form $V(D) \cup \{p_j, p_{j'}\}$ where $D$ is some connected component of $G'-S$ which is adjacent to $p_j$ and $p_{j'}$. 

    First suppose that $j' = j+2$. If $D$ contains some vertex $v \notin V(P')$, then as we are not in Case~$1$, $v$ is isolated in $G'-V(P')$. Combining Lemma~\ref{l-adjacencytopath} with the fact that $p_j$ and $p_{j+2}$ separate $P'$, we find that $N(v) = \{p_j, p_{j+2}\}$. As $v$ is isolated in $G' - \{p_j, p_{j+2}\}$, it follows that $V(D) = \{v\}$ and so $L = \{v, p_j, p_{j+2}\}$. It follows, if $D$ contains some vertex $v \notin V(P')$, then $L$ has size $3$. Else, if $D$ contains only vertices from $P'$, then as $p_j$ and $p_{j+2}$ both separate $P'$, we find that $V(D) = \{p_{j+1}\}$ and so $L = \{p_j, p_{j+1}, p_{j+2}\}$. In each case $L$ has size $3$ meaning if $j' = j+2$ then $\mathcal{L}_i$ is a lemon. 

    This leaves the case where $j' = j+3$. If $D$ contains no vertex from $P'$, then $D$ is a connected component of $G'- V(P')$. As we are not in Case~$1$, this implies that $D$ has size~$1$ and so $L$ has size~$3$, as desired. That is, we now assume that $V(D) \cap V(P') \neq \emptyset$. As $p_j$ and $p_{j+3}$ both separate $P'$, $V(D) \cap V(P') = \{p_{j+1}, p_{j+2}\}$. It follows that either $L$ has size~$4$, as desired, or $D$ contains some vertex $v \notin V(P')$. We now show that if $D$ contains such a vertex $v$, then $G$ contains $S_{3,a,b}$. This contradicts that $G$ is $S_{3,a,b}$-subgraph free and so allows us to conclude that $L$ has size at most~$4$.
    
    As $v$ is contained in the same connected component as the vertices $p_{j+1}$ and $p_{j+2}$ in $G' - S$, $v$ must be adjacent to at least one of $p_{j+1}$ and $p_{j+2}$. Combining this with Claim~\ref{claim-component-adjacencies} and Lemma~\ref{l-adjacencytopath}, we find that either $N(v) = \{p_j,p_{j+2}\}$ or  $N(v) = \{p_{j+1},p_{j+3}\}$. In both cases we find that $S_{3,a,b}$ depicted in Figure~\ref{fig-S334-analysis}\,(g), where $w=v$ and $\Gamma$ is that path via $x$ (or $y$). As $G$ is $S_{3,a,b}$-subgraph free, $D$ contains no such vertex $v$ and so $L$ has size at most $4$, as claimed. This concludes the case where $j' = j+3$ and so also Case~$2$.

    That is, either $G$ contains $S_{3,a,b}$ or every wedge $L \in \mathcal{L}_i$ has size at most~$4$. It follows that $\mathcal{L}_i$ is a lemon with ends $p_j$ and $p_{j'}$ as we sought to prove.
    
\medskip

\noindent \textbf{Case $3$:} There exists some vertex $x \in V_i \setminus V(P')$ such that for some $j \in \{a+1, \ldots, \ell-(a+3)\}$, $N(x) = \{p_j, p_{j+3}\}$.

\medskip
\noindent
In this case, we will show that either $G$ contains some $S_{3,a,b}$ or $V_i$ has size at most $6$. Suppose that $G$ is $S_{3,a,b}$-subgraph-free.

We first claim that if both $p_j$ and $p_{j+3}$ separate $P'$, then $V_i= \{x, p_j, p_{j+1}, p_{j+2}, p_{j+3}\}$. Suppose, for contradiction, there is some vertex $y \in  V_i \setminus \{x, p_j, p_{j+1}, p_{j+2}, p_{j+3}\}$. As both $p_j$ and $p_{j+3}$ separate $P'$, $y \notin V(P')$. As we are not in Case~$1$, $y$ is isolated in $G'- V(P')$. Combining this with Claim~\ref{claim-component-adjacencies} and Lemma~\ref{l-adjacencytopath}, either $N(y)= \{p_j, p_{j+3}\}$, $N(y)= \{p_j, p_{j+2}\}$ or $N(y)= \{p_{j+1}, p_{j+3}\}$. As we are not in Case~$2$, $N(y)\neq \{p_j, p_{j+3}\}$. If either $N(y)= \{p_j, p_{j+2}\}$ or $N(y)= \{p_{j+1}, p_{j+3}\}$ then we find that $S_{3,a,b}$ depicted in Figure~\ref{fig-S334-analysis}\,(g), where $w=y$ and $\Gamma$ is that path via $x$. This contradicts that $G$ is $S_{3,a,b}$-subgraph-free, hence, $V_i$ contains no such vertex $y$ and so if both $p_j$ and $p_{j+3}$ separate $P'$, then $V_i= \{x, p_j, p_{j+1}, p_{j+2}, p_{j+3}\}$. As $V_i$ has size $5$ our claim holds in this case.

We now assume that at least one of $p_j$ and $p_{j+3}$ do not separate $P'$. Suppose that $p_{j+3}$ does not separate $P'$, that is, there is some bypass path for $p_{j+3}$. We claim that this bypass path consists of the single edge $p_{j+2}p_{j+4}$. The bypass path for $p_{j+3}$ must either have end vertices $(p_{j+2}, p_{j+5})$, $(p_{j+1}, p_{j+4})$ or $(p_{j+2}, p_{j+4})$. In the first case we find that $S_{3,a,b}$ depicted in Figure~\ref{fig-S334-analysis}\,(c), in the second case we find that $S_{3,a,b}$ depicted in Figure~\ref{fig-S334-analysis}\,(d). In both cases this contradicts that $G$ is $S_{3,a,b}$-subgraph free, hence, the bypass path for $p_{j+3}$ must have end vertices $(p_{j+2}, p_{j+4})$. If this path has length at least $2$, then by Claim~\ref{claim-component-adjacencies}, this path has length exactly $2$ and so consists of a single intermediate vertex $y \notin V(P')$ which is adjacent to $p_{j+2}$ and $p_{j+4}$. We then find that $S_{3,a,b}$ depicted in Figure~\ref{fig-S334-analysis}\,(f), where $w =y$ and $\Gamma$ is that path via $x$. As $G$ is $S_{3,a,b}$-subgraph free, the bypass path for $p_{j+3}$ consists of the single edge $p_{j+2}p_{j+4}$.

We now claim that both $p_j$ and $p_{j+4}$ separate $P'$. Suppose that $p_j$ does not separate $P'$. By symmetry with $p_{j+3}$, the bypass path for $p_j$ must consist of the single edge $p_{j-1}p_{j+1}$. We now find that $S_{3,a,b}$ depicted in Figure~\ref{fig-S334-analysis}\,(h), a contradiction. Hence, $p_{j}$ separates $P'$. Suppose now that $p_{j+4}$ does not separate $P'$, that is, there is some bypass path for $p_{j+4}$. Note that, as we showed there was no bypass path for $p_{j+3}$ with end vertices $(p_{j+2}, p_{j+4})$ there also cannot be any bypass path for $p_{j+4}$ with end vertices $(p_{j+2}, p_{j+4})$. It follows that the bypass path for $p_{j+4}$ must either have end vertices $(p_{j+3}, p_{j+6})$ or $(p_{j+3}, p_{j+5})$. In both cases we find that $S_{3,a,b}$ depicted in Figure~\ref{fig-S334-analysis}\,(i), where $\delta = j$ and the path $\Gamma$ corresponds to that path via the vertex $x$. Again this contradicts that $G$ is $S_{3,a,b}$-subgraph free. It follows that $p_{j+4}$ separates $P'$ and so both $p_j$ and $p_{j+4}$ separate $P'$, as claimed.

We now claim that $V_i= \{x, p_j, p_{j+1}, p_{j+2}, p_{j+3}, p_{j+4}\}$. Suppose, for contradiction, there is some vertex $y \in  V_i \setminus \{x, p_j, p_{j+1}, p_{j+2}, p_{j+3}, p_{j+4}\}$. Recall that there is no bypass path for $p_{j+3}$, of length at least~$2$. If $y$ were adjacent to $p_{j+4}$ alongside some vertex in $p_j, p_{j+1}, p_{j+2}, p_{j+3}$, then this would result in such a bypass path. It follows that $N(y) \subseteq \{p_j, p_{j+1}, p_{j+2}, p_{j+3}\}$. Combining this with Claim~\ref{claim-component-adjacencies} and Lemma~\ref{l-adjacencytopath}, either $N(y)= \{p_j, p_{j+3}\}$, $N(y)= \{p_j, p_{j+2}\}$ or $N(y)= \{p_{j+1}, p_{j+3}\}$. As we are not in Case~$2$, $N(y)\neq \{p_j, p_{j+3}\}$. If either $N(y)= \{p_j, p_{j+2}\}$ or $N(y)= \{p_{j+1}, p_{j+3}\}$ then we find that $S_{3,a,b}$ depicted in Figure~\ref{fig-S334-analysis}\,(g), where $w=y$, a contradiction. It follows that there is no such vertex $y$ and so $V_i= \{x, p_j, p_{j+1}, p_{j+2}, p_{j+3}, p_{j+4}\}$. As again $V_i$ has size at most $6$, this concludes Case~$3$. 
\medskip
        
\noindent \textbf{Case $4$:} There exists some vertex $x \in V_i \setminus V(P')$ such that for some $j \in \{a+1, \ldots, \ell-(a+2)\}$, $N(x) = \{p_j, p_{j+2}\}$.

\medskip\noindent
    In this case, we will again show that either $G$ contains some $S_{3,a,b}$ or $|V_i| \leq 6$.

    Combining Claim~\ref{claim-component-adjacencies} and Lemma~\ref{l-adjacencytopath} with the fact that we are not in either Case~$1$ or Case~$3$, for every vertex $v \in V_i \setminus V(P')$, there is some  $j' \in \{a+1, \ldots, \ell-(a+2)\}$ such that $N(v) =\{p_{j'},p_{j'+2}\}$. As we are not in Case~$2$, for any vertex $v \neq x$ in $V_i \setminus V(P')$, we find that $N(v) \neq N(x)$. We now claim that either $G$ contains some $S_{3,a,b}$ or $|V_i| \leq 6$. Suppose that $G$ is $S_{3,a,b}$-subgraph free.

    If both $p_j$ and $p_{j+2}$ separate $P'$, then from Claim~\ref{claim-component-adjacencies} and Lemma~\ref{l-adjacencytopath}, every vertex in $V_i \setminus V(P')$ must be adjacent to $p_j$ and $p_{j+2}$. As we are not in Case~$2$, it follows that $V_i = \{x,p_j, p_{j+1}, p_{j+2}\}$. That is, if both $p_j$ and $p_{j+2}$ separate $P'$, then our claim holds. Let us now assume that at least one of $p_j$ does not $p_{j+2}$ separate $P'$, without loss of generality this is $p_{j+2}$.

    We now claim that the bypass path for $p_{j+2}$ has ends $(p_{j+1}, p_{j+3})$. If the bypass path for $p_{j+2}$ has end vertices $(p_{j+1}, p_{j+4})$, then we find that $S_{3,a,b}$ depicted in Figure~\ref{fig-S334-analysis}\,(f), and if it has end vertices $(p_j, p_{j+3})$ then we find that $S_{3,a,b}$ depicted in Figure~\ref{fig-S334-analysis}\,(g). In both cases this contradicts that $G$ is $S_{3,a,b}$-subgraph free. It follows that the bypass path for $p_{j+2}$ must have ends $(p_{j+1}, p_{j+3})$. 

    We now claim that if $p_{j+3}$ separates $P'$, then $V_i$ has size at most $6$. We first show that $p_j$ separates $P'$. Note that by symmetry with $p_{j+2}$, any bypass path for $p_{j+2}$ must have ends $(p_{j-1}, p_{j+1})$. However, if such a bypass path exists, then we find that $S_{3,a,b}$ depicted in Figure~\ref{fig-S334-analysis}\,(j), a contradiction. Hence, $p_j$ separates $P'$. If $p_{j+3}$ separates $P'$, then combining Claim~\ref{claim-component-adjacencies} and Lemma~\ref{l-adjacencytopath} with the fact that we are not in Cases~$1$ or $3$, for every vertex in $y \in V_i \setminus V(P')$, either $N(y)=\{p_{j+1}, p_{j+3}\}$ or $N(y)=\{p_{j+1}, p_{j+3}\}$. As we are not in Case~$2$, it follows there are at most $2$ vertices in $V_i \setminus V(P')$ (as each has a distinct neighbourhood). It follows that $V_i$ has size at most $6$ and our claim holds. That is, we now assume that $p_{j+3}$ does not separate $P'$.

    Let us consider the bypass path for $p_{j+3}$. Recall that, the bypass path for $p_{j+2}$ (and so also $p_{j+3}$) cannot have ends $(p_{j+1}, p_{j+4})$. That is, the bypass path for $p_{j+3}$ must either have ends $(p_{j+2}, p_{j+5})$ or $(p_{j+2}, p_{j+4})$. In the first case we find a $S_{3,a,b}$, symmetrical to that depicted in Figure~\ref{fig-S334-analysis}\,(i). Here $\delta = j+5$, $\delta' = j$ and the path $\Gamma''$ this bypass path corresponds to that path from $p_{j}$ to $p_{j+2}$ via $x$. The $\Gamma$ and $\Gamma'$ then correspond to those bypass paths for $p_{j+3}$ and $p_{j+2}$ respectively. As $G$ is $S_{3,a,b}$-subgraph free, the bypass path for $p_{j+3}$ has ends $(p_{j+2}, p_{j+4})$. 

    We now claim that $p_{j+4}$ separates $P'$. Consider some bypass path for $p_{j+4}$. Given the bypass path for $p_{j+3}$ (and so also $p_{j+4}$) does not have ends $(p_{j+2}, p_{j+5})$, the bypass path for $p_{j+4}$ must either have ends $(p_{j+3}, p_{j+6})$ or $(p_{j+3}, p_{j+5})$. If this is $(p_{j+3}, p_{j+6})$, then we again find a $S_{3,a,b}$ symmetric to that depicted in Figure~\ref{fig-S334-analysis}\,(i), this time $\delta = j+6$ $\delta' = j+1$. If the bypass path for $p_{j+4}$ has ends $(p_{j+3}, p_{j+5})$, then we find that $S_{3,a,b}$ depicted in Figure~\ref{fig-S334-analysis}\,(l). In both cases this contradicts that $G$ is $S_{3,a,b}$-subgraph free, that is, $p_{j+4}$ separates $P'$.

    We finally show that given both $p_{j}$ and $p_{j+4}$ separate $P'$, $V_i$ has size at most $6$. 
    
    As we are not in Cases~$1$, $2$ or $3$ for every vertex in $y \in V_i \setminus V(P')$, either $N(y)=\{p_{j}, p_{j+2}\}$, $N(y)=\{p_{j+1}, p_{j+3}\}$ or $N(y)=\{p_{j+2}, p_{j+4}\}$. If there is some $y \in V_i \setminus V(P')$ such that $N(y)=\{p_{j+1}, p_{j+3}\}$, then we find that $S_{3,a,b}$ depicted in Figure~\ref{fig-S334-analysis}\,(j), where $w= y$ and $\Gamma$ is that path via $x$. If there there is some $y \in V_i \setminus V(P')$ such that $N(y)=\{p_{j+2}, p_{j+4}\}$, then we find that $S_{3,a,b}$ depicted in Figure~\ref{fig-S334-analysis}\,(k), where $w = x$ and $w' = y$. Finally, as we are not in Case~$2$, there is no vertex $y \in V_i \setminus V(P')$ such that $y \neq x$ and $N(y)=\{p_{j}, p_{j+2}\}$. It follows that $V_i = \{x,p_j, p_{j+1}, p_{j+2}, p_{j+3}, p_{j+4}\}$. As, again, $V_i$ has size $6$, this concludes Case~$4$.

    \medskip
        
    \noindent\textbf{Case $5$:} Otherwise, $V_i \setminus V(P') = \emptyset$.

    \medskip\noindent
    We now claim that if we are in none of the above cases then either we find some $S_{3,a,b}$ or $|V_i| \leq 6$. 
        
    Let $p_j := s_i$ and $p_k := s_{i+1}$. That is, by definition both $p_j$ and $p_k$ separate $P'$. Combining Claim~\ref{claim-component-adjacencies} and Lemma~\ref{l-adjacencytopath} with the fact that we are not in Cases $3$ or $4$, we find that $V_i$ contains only vertices from $P'$. That is, $V_i = \{p_j, \ldots, p_k\}$. We now show that if $|V_i| \geq 7$, i.e. $k \geq j+6$, then $G$ contains some $S_{3,a,b}$.

    Suppose $k \geq j+6$. It follows that no vertex from  $p_{j+1}, p_{j+2}, p_{j+3}, p_{j+4}, p_{j+5}$ separates~$P'$, that is, each of these vertices have some bypass path. We now consider those bypass paths which exist as a result showing that in each case we find some $S_{3,a,b}$. As $V_i$ contains only vertices from $P'$, each of these paths correspond to an edge.

    We first consider the bypass path for $p_{j+1}$. As $p_j$ separates $P'$, the bypass path for $p_{j+1}$ corresponds either to the edge $p_{j}p_{j+3}$ or the edge $p_{j}p_{j+2}$. 

    We first suppose that this corresponds to the edge $p_{j}p_{j+3}$. Let us now consider the bypass path for $p_{j+3}$. This corresponds to either the edge $p_{j+2}p_{j+5}$, the edge $p_{j+1}p_{j+4}$, or the edge $p_{j+2}p_{j+4}$. Suppose this is $p_{j+2}p_{j+5}$, that is, we are under the assumption that $G'$ contains both the edge $p_{j}p_{j+3}$ and the edge $p_{j+2}p_{j+5}$. 
    
    We now claim that the bypass path for $p_{j+5}$ results in some $S_{3,a,b}$. The bypass path for $p_{j+5}$ must correspond to either the edge $p_{j+4}p_{j+7}$, $p_{j+3}p_{j+6}$ or $p_{j+4}p_{j+6}$. If this is either $p_{j+4}p_{j+7}$ or $p_{j+4}p_{j+6}$, then we find that $S_{3,a,b}$ depicted in Figure~\ref{fig-S334-analysis}\,(m). If this is $p_{j+3}p_{j+6}$, then we find that $S_{3,a,b}$ depicted in Figure~\ref{fig-S334-analysis}\,(n). It follows that if $G$ contains both the edges $p_{j}p_{j+3}$ and $p_{j+2}p_{j+5}$, then $G$ contains $S_{3,a,b}$, that is our claim holds in this case.

    We now assume that $G$ contains the edge $p_{j}p_{j+3}$ but not $p_{j+2}p_{j+5}$. It now follows that the bypass path for $p_{j+3}$ must either correspond to the edge $p_{j+1}p_{j+4}$ or $p_{j+2}p_{j+4}$. 
    
    Suppose that this is the edge $p_{j+1}p_{j+4}$. We now claim that the bypass path for $p_{j+4}$ results in some $S_{3,a,b}$. As per the previous paragraph, $G$ does not contain the $p_{j+2}p_{j+5}$, that is, bypass path for $p_{j+4}$ either corresponds to the edge $p_{j+3}p_{j+6}$ or $p_{j+3}p_{j+5}$. In both cases we find that $S_{3,a,b}$ depicted in Figure~\ref{fig-S334-analysis}\,(o). It follows, if $G$ contains both the edge $p_{j}p_{j+3}$ and $p_{j+1}p_{j+4}$, then $G$ contains $S_{3,a,b}$. Hence, our claim holds in this case.

    We now consider the case where $G$ contains the edge $p_{j}p_{j+3}$ but not $p_{j+2}p_{j+5}$ or $p_{j+1}p_{j+4}$. It follows that the bypass path for $p_{j+3}$ corresponds to the edge $p_{j+2}p_{j+4}$. We now claim that the path for $p_{j+4}$ results in some $S_{3,a,b}$. As $G$ does not contain the edge $p_{j+2}p_{j+5}$, the bypass path for $p_{j+4}$ either corresponds to the edge $p_{j+3}p_{j+6}$ or $p_{j+3}p_{j+5}$. In both cases we find that $S_{3,a,b}$ depicted in Figure~\ref{fig-S334-analysis}\,(i). That is, if $G$ contains the edge $p_{j}p_{j+3}$ and $p_{j+2}p_{j+4}$, then $G$ contains $S_{3,a,b}$. Further, as this was the final case when considering the bypass path for $p_{j+3}$, if $G$ contains the edge $p_{j}p_{j+3}$, then the bypass path for $p_{j+3}$ results in some $S_{3,a,b}$. That is, if $G$ contains the edge $p_{j}p_{j+3}$, then our claim holds.
    
    We now consider the case where $G$ does not contain the edge $p_{j}p_{j+3}$. It follows that the bypass path for $p_{j+1}$ corresponds to the edge $p_{j}p_{j+2}$.

    Let us now consider the bypass path for $p_{j+2}$. As $G'$ does not contain the edge $p_{j}p_{j+3}$, the bypass path for $p_{j+2}$ must either correspond to the edge $p_{j+1}p_{j+4}$ or $p_{j+1}p_{j+3}$. Suppose this is the edge $p_{j+1}p_{j+4}$. That is, $G'$ contains the edge $p_{j}p_{j+2}$ and $p_{j+1}p_{j+4}$. 
    
    We now claim that either the bypass path for $p_{j+4}$ corresponds to the edge $p_{j+3}p_{j+5}$, or we find $S_{3,a,b}$, in which case our claim holds. If the bypass path for $p_{j+4}$ corresponds to the edge $p_{j+3}p_{j+6}$, then we find a $S_{3,a,b}$ symmetric to that depicted in Figure~\ref{fig-S334-analysis}\,(m), where $\delta' = j$ and $\delta = j+6$. If this bypass corresponds to the edge $p_{j+2}p_{j+5}$, then we find a $S_{3,a,b}$ symmetric to that depicted in Figure~\ref{fig-S334-analysis}\,(o), where $\delta' = j$ and $\delta = j+5$. That is, either we find $S_{3,a,b}$ or the bypass path for $p_{j+4}$ corresponds to the edge $p_{j+3}p_{j+5}$. In the first case our claim holds, that is, we assume that the bypass path for $p_{j+4}$ corresponds to the edge $p_{j+3}p_{j+5}$.

    We now claim that the bypass path for $p_{j+5}$ results in some $S_{3,a,b}$. As the bypass path for $p_{j+4}$ corresponds to the edge $p_{j+3}p_{j+5}$, the bypass path for $p_{j+5}$ is not $p_{j+4}p_{j+6}$, it must either be $p_{j+4}p_{j+7}$ or $p_{j+4}p_{j+6}$. In both cases we find that $S_{3,a,b}$ depicted in Figure~\ref{fig-S334-analysis}\,(i), where $\delta = j+1$. We note that this completes the case where the bypass path for $p_{j+2}$ corresponds to the edge $p_{j+1}p_{j+4}$.

    This leaves only the case where the bypass path for $p_{j+2}$ corresponds to the edge $p_{j+1}p_{j+3}$ and $G$ does not contain the edge $p_{j+1}p_{j+4}$. 
    
    Now the bypass path for $p_{j+3}$ either corresponds to the edge  $p_{j+2}p_{j+5}$ or $p_{j+2}p_{j+4}$. If it corresponds to the edge $p_{j+2}p_{j+5}$, then we find a $S_{3,a,b}$ symmetric to that depicted in Figure~\ref{fig-S334-analysis}\,(i), where $\delta = j+5$ and $\delta' = j$. That is, we now assume that $G$ does not contain the edge $p_{j+2}p_{j+5}$, and so the bypass path for $p_{j+3}$ is the edge $p_{j+2}p_{j+4}$.

    We now consider the bypass path for $p_{j+4}$. If it corresponds to the edge $p_{j+3}p_{j+6}$, then we again find a $S_{3,a,b}$ symmetric to that depicted in Figure~\ref{fig-S334-analysis}\,(i), where $\delta = j+6$ and $\delta' = j+1$. That is we now suppose the bypass path for $p_{j+4}$ corresponds to the edge $p_{j+3}p_{j+5}$.

    We now claim that the bypass path for $p_{j+5}$ results in an $S_{3,a,b}$. If the bypass path for $p_{j+5}$ corresponds to the $p_{j+4}p_{j+7}$, then we again find a $S_{3,a,b}$ symmetric to that depicted in Figure~\ref{fig-S334-analysis}\,(i), where $\delta = j+7$ and $\delta' = j+2$, else, it corresponds to the edge $p_{j+4}p_{j+6}$, in which case we find that $S_{3,a,b}$ depicted in Figure~\ref{fig-S334-analysis}\,(p). As this was the final case where the bypass path for $p_{j+1}$ corresponds to the edge $p_{j}p_{j+2}$, we have shown that in each case $G$ contains $S_{3,a,b}$. That is, if $|V_i| \geq 7$, then $G$ contains some $S_{3,a,b}$.
    \end{claimproof}

    \noindent
    We now consider the entire graph $G$ and claim that $G$ contains a cycle of length at least least~$3b$. Since the arguments applied to $P'$ hold whenever $P$ has at least $3b$ vertices, the existence of such a cycle in $G$ allows us, by symmetry of the cycle,  to apply the reasoning used for the vertices of $G'$ to all vertices of $G$. It follows then from Claim~\ref{clm-G'-citrus}, the graph $G$ is a citrus bush flowering from a cycle, where each citrus is either a lemon or has size at most $6$.

    We first note that as $\ell \geq 3b+2a-4$, if there some path $Q$ between $p_i$ and $p_j$ such that  $V(Q) \cap P = \{p_i, p_j\}$ and $i \in \{1, \ldots, a\}$ and $j \in \{\ell-a+1, \ldots, \ell\}$ then $G$ contains some cycle of length at least~$3b$. We now assume that no such path exists and claim that this implies that $G$ contains $S_{3,a,b}$. Note that this implies that every path from $p_1$ to $p_{\ell}$ contains some vertex from $p_{a+1}, \ldots, p_{\ell-a}$. 

    \begin{figure}
        \centering
        \includegraphics[width=0.5\linewidth, page=13]{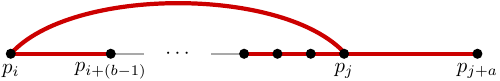}
        \caption{An occurrence of $S_{3,a,b}$ in the proof of Lemma~\ref{l-s3ab}.}
        \label{fig:s3ab_start-to-mid}
    \end{figure}

    It then follows by Claim~\ref{clm-G'-citrus} that either there is some cut vertex for $G$ in the set $\{p_{a+b+1}, \ldots, p_{\ell-(a+b+2)}\}$ or there exist vertices $p_i$, $p_j$ and a path $Q$ such that $V(Q) \cap P = \{p_i, p_j\}$ and either $i \in \{1, \ldots, a\}$ and $j \in \{a+b+1, \ldots, \ell-(a+b+2)\}$ or symmetrically $i \in \{\ell, \ldots, \ell-a+1\}$ and $j \in \{a+b+1, \ldots, \ell-(a+b+2)\}$. As $G$ is $2$-connected, without loss of generality we assume there is such a path between $p_i$ and $p_j$, where $i \in \{1, \ldots, a\}$ and $j \in \{a+b+1, \ldots, \ell-(a+b+2)\}$. It follows that $G$ contains some $S_{3,a,b}$ as shown in Figure~\ref{fig:s3ab_start-to-mid}. It follows that either we find some  $S_{3,a,b}$ or $G$ contains some cycle of length at least~$3b$. In the first case, our claim holds directly and in the second we find that $G$ is a citrus bush flowering from a cycle such that every citrus is either a lemon or has size at most~$6$.
\end{proof}

\noindent
We are now ready to prove Theorem~\ref{t-s334}, which we restate below.

\thmSThreeThreeFourFree*
\begin{proof}
    Let $(G,T,k)$ be an instance of \stf, where $G$ is a $S_{3,3,4}$-subgraph-free graph, $T$ a terminal set and $k$ an integer.
    From Lemma~\ref{l-2con} we may assume that $G$ is $2$-connected.
    If $G$ is $P_{11}$-subgraph-free, we apply Theorem~\ref{t-p11free}, that is, we assume that $G$ contains some $P_{11}$ subgraph.
    Let $P = (v_1, \dots , v_r)$ be a longest path in $G$, by Lemma~\ref{l-longestpath} such a path can be found in polynomial time.

    We now consider the components of $G-V(P)$.

    \begin{claim}\label{c-s334-size-2}
        Every component of $G-V(P)$ has size at most~$2$.
    \end{claim}
    \begin{claimproof}
            Suppose for a contradiction that there is a connected component $D$ of size at least~$3$ in $G-V(P)$. By the $2$-connectivity of $G$, there exist two distinct vertices $x$, $y$ in $D$ and two distinct indices $i$, $j$ such that $x$ is adjacent to $p_i$ and $y$ is adjacent to $p_j$. 

            \begin{figure}
                \centering
                \includegraphics[width=0.3\linewidth, page=14]{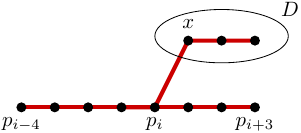}
                \caption{An occurrence of $S_{3,3,4}$ in the proof of Claim~\ref{c-s334-size-2}.}
                \label{fig:s334-size-2}
            \end{figure}

            As $D$ has size at least~$3$, either $x$ or $y$ is the end vertex of some path of length $2$ in $D$, without loss of generality this is $x$. It follows that $i \notin \{1, 2, 3,r-2,r-1,r\}$, else we find a longer path via $D$. We note that as $r \geq 11$, either $i \geq 6$ or $i \geq 6$. If $i \geq 5$, or symmetrically $i \leq r-4$, then we find that $S_{3,3,4}$ shown in Figure~\ref{fig-S334-analysis}. That is, if there is some connected component of size at least~$3$ in $G-V(P)$, then $G$ contains some $S_{3,3,4}$, a contradiction.
    \end{claimproof}

    \begin{claim}\label{c-s334-comp2neighbours}
        Every connected component $D$ of $G-V(P)$ with size~$2$ is adjacent to exactly two vertices of $P$.
    \end{claim}
    \begin{claimproof}
        Let $V(D) = \{x,y\}$. As $G$ is $2$-connected, there is some pair of indices $i$, $j$ such that $x$ is adjacent to $p_i$ and $y$ is adjacent to $p_j$. As $P$ is a longest path, $i,j \notin \{1,2,r-1,r\}$ and $j-i \geq 2$.

        \begin{figure}
            \centering
            \includegraphics[width=0.7\linewidth, page=15]{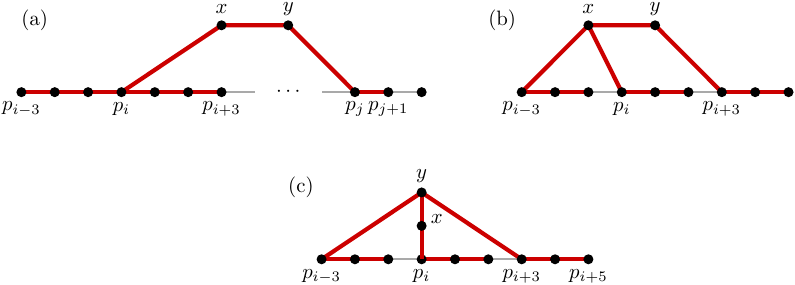}
            \caption{Occurrences of $S_{3,3,4}$ in the proof of Claim~\ref{c-s334-comp2neighbours}.}
            \label{fig:s334-comp2neighbours}
        \end{figure}

        Suppose first that $i \in \{4,\ldots, r-3\}$. If $|j-i| \geq 4$, then we find $S_{3,3,4}$ as show in Figure~\ref{fig:s334-comp2neighbours} a). Note this shows the case where $i<j$, the case where $j<i$ is symmetric. It follows that $N(D) \subseteq \{i, i-3, i+3\}$. If $D$ is adjacent to at least three vertices in $P$, then $N(D) = \{i, i-3, i+3\}$. That is, without loss of generality, either $p_i, p_{i-3} \in N(x)$ and $p_{i+3} \in N(y)$, or, $p_i \in N(x)$ and $p_{i-3}, p_{i+3} \in N(y)$. In both cases we find $S_{3,3,4}$ as show in Figure~\ref{fig:s334-comp2neighbours} b) and c). It follows that, if $i \in \{4,\ldots, r-3\}$, then $D$ is adjacent to exactly two vertices of $P$.

        Else, we may assume that $i \notin \{4,\ldots, r-3\}$, and symmetrically $j \notin \{4,\ldots, r-3\}$. As $i,j \notin \{1,2,r-1,r\}$, this implies that $i = 3$ and $j = r-2$, that is, $N(V(D)) = \{p_3,p_{r-2}\}$. As $D$ is adjacent to exactly two vertices of $P$, our claim holds.
    \end{claimproof}

    \noindent
    If $|V(P)| \geq 14$, then applying Lemma~\ref{l-s3ab}, we find that $G$ is a citrus bush flowering from a cycle and every citrus in this bush is either a lemon or has size at most~$6$. That is, $G$ is a $6$-lemon bush flowering path flowering from a cycle. In addition, we can find the stem set of this citrus bush in polynomial time. Let $C$ be a longest cycle, such a cycle can be found in polynomial time using Lemma~\ref{l-longestpath}. For every pair of vertices $u,v$ in $C$, we check if $C$ is connected in $C-\{u,v\}$. If $C$ is disconnected in $C-\{u,v\}$, then we add $u$ and $v$ to the stemset. Naively this can be done in time $O(n^3)$. It follows now by Theorem~\ref{t-cyclebush} we can solve {\sc Steiner Forest} on $G$ in polynomial time.

    Suppose now $|V(P)| \leq 13$. Recall that every component $D$ of $G-V(P)$ has size at most~$2$ and $D$ is adjacent to exactly two vertices in $P$. It follows that $V(D) \cup N(D)$ is a wedge with ends $N(D)$, further, if $D$ has size~$1$ then this wedge is juicy and if $D$ has size~$2$ then this wedge is seeded. It follows that $G$ is a lemon bush flowering from a graph with stemset $V(P)$. As $|V(P)| \leq 13$, by Theorem~\ref{t-solvetangle}, we can solve {\sc Steiner Forest} for the instance $(G,T,k)$ in polynomial time.
\end{proof}

\section{\fl{Hardness Results}}%The Full Proofs of Section~\ref{s-hard}}

We finish our paper by proving Theorem~\ref{t-newhard}, which we restate below.

\label{s-fp-hard}
\thmNewHard*

\medskip
\noindent
It is convenient to present the proof of Theorem~\ref{t-newhard} as a reduction from a specific Constraint Satisfaction Problem (CSP). In general, a CSP instance is a triple $(V ,D, C)$,
where:

\begin{itemize}
\item $V$ is a set of variables,
\item $D$ is a domain of values,
\item $C$ is a set of constraints, $\{c_1,c_2,\dots ,c_q\}$.
Each constraint $c_i\in C$ is a pair $\langle
s_i,R_i\rangle$, where:
\begin{itemize}
\item $s_i$ is a tuple of variables of length $m_i$, called the {\em constraint scope,} and
\item $R_i$ is an $m_i$-ary relation over $D$, called the {\em constraint
  relation.}
\end{itemize}
\end{itemize}

\noindent
For each constraint $\langle s_i,R_i\rangle$, the tuples of $R_i$
indicate the allowed combinations of simultaneous values for the
variables in $s_i$. The length $m_i$ of the tuple $s_i$ is called the
{\em arity} of the constraint.  A {\em solution} to a constraint
satisfaction problem instance is a function $f$ from the set of
variables $V$ to the domain of values $D$ such that for each
constraint $\langle s_i,R_i\rangle$ with $s_i = \langle
v_{i_1},v_{i_2},\dots,v_{i_m}\rangle$, the tuple $\langle f(v_{i_1}),
f(v_{i_2}),\dots,f(v_{i_m})\rangle$ is a member of $R_i$.

If $D$ is a fixed domain and $\mathcal{R}$ is a set of relations over $D$, then $\textup{CSP}(\mathcal{R})$ is the special case where every relation $R_i$ is from $\mathcal{R}$. We will specifically consider the domain $\overline{D}=\{0,1,2\}$ and the set $\overline{\mathcal{R}}$ of relations that contains six relations:

\begin{itemize}
    \item for every $d\in \overline D$, the unary relation $x\neq d$, and
    \item for every $d\in \overline D$, the binary relation $(x=d)\to (y=d)$.
\end{itemize}

\noindent
There are powerful general classification results describing the polynomial-time and \NP-complete cases of CSP \cite{Bu17,Bu06,Bu11,Zh20}, which should show that $\textup{CSP}(\overline{\mathcal{R}})$ is \NP-complete. However, we give a very simple, self-contained proof here.

\begin{lemma}\label{lem:csphard}
$\textup{CSP}(\overline{\mathcal{R}})$ is \NP-complete.
\end{lemma}

\begin{proof}
The proof is by reduction from \textsc{3-Colouring}. We construct a gadget that expresses $x\neq y$. For every $\ell\in D$, we introduce the variables $\alpha_\ell$, $\beta_\ell$, $\gamma_\ell$ and the following five constraints (thus, the gadget contains a total of nine variables besides $x$ and $y$, and 15 constraints):

\begin{itemize}
\item $(x=\ell)\to (\alpha_\ell=\ell)$
\item $(y=\ell)\to (\beta_\ell=\ell)$.
\item $(\gamma_\ell\neq \ell)$
\item $(\gamma_\ell=\ell+1)\to (\alpha_\ell=\ell+1)$
\item $(\gamma_\ell=\ell+2)\to (\beta_\ell=\ell+2)$
\end{itemize}

\noindent
Note that addition is modulo 3. Observe that $x$ and $y$ cannot both take the value $\ell$ in the solution: then $\alpha_\ell$ and $\beta_\ell$ would both have value $\ell$ as well, which contradicts both $\gamma_\ell=\ell+1$ and $\gamma_\ell=\ell+2$. On the other hand, if the values of $x$ and $y$ are not both $\ell$, then it is easy to see that one can choose values for $\alpha_\ell$, $\beta_\ell$, and $\gamma_\ell$.

Given an instance of \textsc{3-Colouring} (that is, a graph $G$), we introduce a variable for every vertex of $G$. For every edge of $G$, we introduce a copy of the gadget described above, introducing nine new variables and 15 new constraints per edge. It is clear that the constructed instance has a satisfying assignment if and only if $G$ is 3-colourable.
\end{proof}

\noindent
With the hardness of $\textup{CSP}(\overline{\mathcal{R}})$ at hand, we can prove Theorem~\ref{t-newhard}.
\begin{proof}[Proof (of Theorem~\ref{t-newhard})]
The proof is by reduction from $\textup{CSP}(\overline{\mathcal{R}})$. Consider an instance $I$ of $\textup{CSP}(\overline{\mathcal{R}})$ with $n$ variables and $m_1$ unary constraints and $m_2$ binary constraints. We may assume that each  variable appears in at least one binary constraint; otherwise, the role of that variable is trivial. We construct a graph $G$ and a set $T$ of terminal pairs as follows (see Figure~\ref{fig:np}).

\begin{figure}
\centering
\includegraphics[width=0.9\linewidth, page=10]{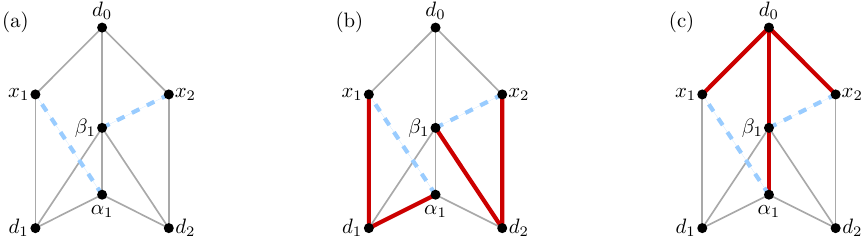}
\caption{An example of the reduction in Theorem~\ref{t-newhard}. (a) The graph corresponding to the $\textup{CSP}(\overline{\mathcal{R}})$ instance $(x_1\neq 1) \wedge (x_2\neq 2) \wedge ((x_1=0)\to(x_2=0))$. The dashed blue lines show the two terminal pairs in the set $T$. (b) A \textsc{Steiner Forest} solution  corresponding to $x_1=2$, $x_2=1$. (c) A \textsc{Steiner Forest} solution corresponding to $x_1=x_2=0$.}\label{fig:np}
\end{figure}
\begin{itemize}
    \item We introduce three vertices $d_0$, $d_1$, $d_2$.
    \item For every variable $x_i$, we introduce a vertex $x_i$ and, for $\ell=0,1,2$, connect $x_i$ to $d_\ell$ \emph{unless} the constraint $(x_i\neq \ell)$ exists.
    \item For every binary constraint $c_j$ of the form $(x_{i_1}=\ell)\to (x_{i_2}=\ell)$, we introduce two adjacent vertices $\alpha_j$ and $\beta_j$, and we make both of them adjacent to $d_0$, $d_1$, $d_2$ \emph{except} that $\alpha_j$ is not adjacent to $d_\ell$. Furthermore, we add the terminal pairs $(x_{i_1},\alpha_j)$ and $(x_{i_2},\beta_j)$ to $T$.
    \end{itemize}
    
\noindent    
This completes the construction of $G$ and $T$. Observe that if we delete the vertices $d_0$, $d_1$, and $d_2$ from $G$, then every component is of the form $\{\alpha_j,\beta_j\}$, that is, consisting only of two vertices. Thus $G$ has 2-deletion number at most~$3$. 
Also, $G$ has treewidth~$3$, since the graph $G'=G \cup \{d_0d_1,d_1d_2,d_0d_2\}$ is a chordal completion of $G$ and has clique number 4. To see this, it suffices to observe that each vertex $\alpha_j$ is simplicial with degree 3 in $G'$, and after removing all of these, each vertex $\beta_j$ is again simplicial with degree 3.
We claim that $I$ has a satisfying assignment if and only if the constructed \textsc{Steiner Forest} instance has a solution with $k=n+2m_2$ edges.

For the forward direction, consider a satisfying assignment $f$ of $I$. The solution $S$ of \textsc{Steiner Forest} consists of the following edges. For every variable $x_i$, let us add the edge $x_id_{f(x_i)}$ into $S$. Consider a constraint $c_j$ of the form $(x_{i_1}=\ell)\to (x_{i_2}=\ell)$ and the corresponding two vertices $\alpha_j$ and $\beta_j$. First, let us add the edge $\beta_jd_{f(x_{i_2})}$ into $S$. Then the pair $(x_{i_2},\beta_j)$ in $T$ is satisfied, as the two vertices are connected via $d_{f(x_{i_2})}$. If $f(x_{i_1})\neq \ell$, then similarly we add the edge $\alpha_jd_{f(x_{i_1})}$ into $S$, satisfying the pair $(x_{i_1},\alpha_j)$. However, if $f(x_{i_1})=\ell$, then the edge $\alpha_jd_{f(x_{i_1})}$ does not exist in $G$. In this case the constraint $(x_{i_1}=\ell)\to (x_{i_2}=\ell)$ ensures that $f(x_{i_2})=\ell$ and hence the edge $\beta_jd_{f(x_{i_2})}=\beta_jd_\ell$ was already added to $S$. Thus by adding the edge $\alpha_j\beta_j$, we can connect $\alpha_j$ and $x_{i_1}$ via $\beta_j$ and $d_\ell$. Note that we have introduced a total of $n+2m_2$ edges into $S$, as required. 

For the reverse direction, consider a solution $S$ of the \textsc{Steiner Forest} instance with at most $n+2m_2$ edges. For each variable $x_i$, there has to be at least one edge of $S$ incident to the vertex $x_i$, otherwise it cannot be connected to any $\alpha_j$ or $\beta_j$ vertex (here we use the assumption that $x_i$ appears in at least one binary constraint, and thus at least one pair of $T$ involves $x_i$). Furthermore, for every constraint $c_j$, the introduced $\alpha_j$ and $\beta_j$ vertices need to be adjacent to at least two edges of $S$, otherwise at least one of them is not connected to any of $\{d_0,d_1,d_2\}$. This way, we can identify $n+2m_2$ edges of $S$. Observe that an edge of $S$ cannot play two different roles (e.g., being incident to both $x_{i_1}$ and $x_{i_2}$), hence all these edges are distinct. As $S$ has size at most $n+2m_2$, this means that this way we identify each edge of $S$ precisely once: there is exactly one edge incident to $x_i$ and for every binary constraint $c_j$, there are precisely two edges incident to $\{\alpha_j,\beta_j\}$.

We claim that $d_0$, $d_1$, and $d_2$ are in different components of $S$. Indeed, they cannot be connected via $x_i$, as there is only a single edge incident to $x_i$. Furthermore, they cannot be connected via $\{\alpha_j,\beta_j\}$: the only way this could be possible is if the two edges incident to $\{\alpha_j,\beta_j\}$ are both incident to, say, $\alpha_j$, but then $\beta_j$ would become an isolated vertex and the pair of $T$ containing $\beta_j$ would not be satisfied. Hence, $d_0$, $d_1$, and $d_2$ are in different components of $S$.

Let $f(x_i)=\ell$ for the unique $\ell$ such that $x_i$ is in the same component of $S$ as $d_\ell$. We claim that $f$ is a satisfying assignment of the $\textup{CSP}(\overline{\mathcal{R}})$ instance $I$. The fact that the edge $x_id_\ell$ exists in $G$ implies that $f$ satisfies the unary constraints. Consider now a binary constraint~$c_j$ of the form $(x_{i_1}=\ell)\to (x_{i_2}=\ell)$ and the two pairs $(x_{i_1},\alpha_j)$ and $(x_{i_2},\beta_j)$ introduced into~$T$. Let us assume $f(x_{i_1})=\ell$; otherwise, the constraint is satisfied. This means that  $x_{i_1}$ is in the component of $S$ containing $d_\ell$. To satisfy the pair $(x_{i_1},\alpha_j)$, vertex $\alpha_j$ also has to be in the same component of $S$ as $d_\ell$. In particular, there has to be an edge in $S$ connecting $d_\ell$  and $\{\alpha_j,\beta_j\}$. By construction, the only such edge is $\beta_jd_\ell$. Thus the only way $\alpha_j$ can be in the component of $d_\ell$ if $S$ contains both $\alpha_j\beta_j$ and $\beta_jd_\ell$, implying that $\beta_j$ is in the component of $d_\ell$. Now to satisfy the pair $(x_{i_2},\beta_j)$, vertex $x_{i_2}$ has to be also in the component of $S$ containing~$d_\ell$. This means that $f(x_{i_2})=\ell$, satisfying the constraint $(x_{i_1}=\ell)\to (x_{i_2}=\ell)$.
\end{proof}

\bibliography{steinerforest}
\end{document}

%% file: figures/lemons.tex
\begin{tikzpicture}
    \node[vertex] (v1) at (0,0){};
    \node[vertex] (v2) at (3,0){};

    \node[vertex] (u1) at (1,0.7){};
    \node[vertex] (u2) at (2,0.7){};
    \node[vertex] (y1) at (1,1.1){};
    \node[vertex] (ym) at (1.5,1.35){};
    \node[vertex] (y2) at (2,1.1){};
    \node[vertex] (z1) at (1,-0.5){};
    \node[vertex] (z2) at (2,-0.5){};

    \node[vertex] (w) at (1.5, 0.3){};
    %\node[vertex] (x) at (1.5, -0.5){};

    \draw[edge] (v1) -- (v2);
    \draw[edge] (v1) -- (u1) -- (u2) -- (v2);
    \draw[edge] (v1) -- (y1) -- (ym) -- (y2) -- (v2);
    \draw[edge] (y1) -- (y2);
    \draw[edge] (v1) -- (z1) -- (z2) -- (v2);
    \draw[edge] (v1) -- (w) -- (v2);
    %\draw[edge] (v1) -- (x) -- (v2);
    \draw[edge] (z1) -- (v2);
    \draw[edge] (z2) -- (v1);
    
\begin{scope}[shift = {(5,0)}, scale = 0.5]
        \node[vertex, label = below: $u_1$] (v1) at (0,0){};
    \node[vertex] (v2) at (3,0){};

    \node[vertex] (u1) at (1,1){};
    \node[vertex] (u2) at (2,1){};
    \node[vertex] (y1) at (1,1.5){};
    \node[vertex] (y2) at (2,1.5){};

    \node[vertex] (w) at (1.5, 0.5){};
    \node[vertex] (x) at (1.5, -0.5){};

    \draw[edge] (v1) -- (v2);
    \draw[edge] (v1) -- (u1) -- (u2) -- (v2);
    \draw[edge] (v1) -- (y1) -- (y2) -- (v2);
    \draw[edge] (v1) -- (w) -- (v2);
    \draw[edge] (v1) -- (x) -- (v2);

    \begin{scope}[shift = {(3,0)}]
        \node[vertex, label = below: $u_2$] (v1) at (0,0){};
        \node[vertex] (v2) at (2,0){};
        \node[vertex] (w) at (1, 0.5){};
        \node[vertex] (x) at (1, -0.5){};
        \draw[edge] (v1) -- (w) -- (v2);
        \draw[edge] (v1) -- (x) -- (v2);
    \end{scope}

    \begin{scope}[shift = {(5,0)}]
        \node[vertex] (v1) at (0,0){};
        \node[vertex] (v2) at (1,0){};
        \draw[edge] (v1) -- (v2);
    \end{scope}

    \begin{scope}[shift = {(6,0)}]
        \node[vertex, label = below: $u_3$] (v1) at (0,0){};
        \node[vertex, label = below: $u_4$] (v2) at (3,0){};
        \node[vertex] (y1) at (0.8,1){};
        \node[vertex] (ym) at (1.5,1.5){};
        \node[vertex] (y2) at (2.2,1){};

        %\draw[edge] (v1) -- (v2);
    \draw[edge] (v1) -- (y1) -- (ym) -- (y2) -- (v2);
        
    \end{scope}
    
\end{scope}
    
\end{tikzpicture}

%% file: figures/s2ab.tex
\begin{tikzpicture}
\begin{scope}[scale = 0.7]
    \foreach \i in {1,..., 9}{
        \node[vertex](v\i) at (\i, 0){};
        
    }
    \draw[redge] (v4) to [bend left = 40] (v7);
    \draw[redge] (v1) -- (v2);
    \draw[redge] (v2) -- (v3);
    \draw[redge] (v3) -- (v4);
    \draw[redge] (v4) -- (v5);
    \draw[redge] (v5) -- (v6);
    \draw[edge] (v6) -- (v7);
    \draw[redge] (v7) -- (v8);
    \draw[redge] (v8) -- (v9);

    \draw[br] (3,-0.15) -- (1,-0.15);
    \node[](t) at (2, -0.5){$a$};

    \draw[br] (9,-0.15) -- (7,-0.15);
    \node[](t) at (8, -0.5){$b$};

    \node[label=below:$u_1$] (t) at (4,0){};
    \node[label=below:$u_2$] (t) at (7,0){};

\end{scope}

\begin{scope}[scale = 0.7, shift = {(10,0)}]

        \foreach \i in {1,..., 9}{
        \node[vertex](v\i) at (\i, 0){};
        
    }
    \draw[redge] (v3) to [bend left = 40] (v5);
    \draw[redge] (v5) to [bend left = 40] (v7);
    \draw[redge] (v4) to [bend right = 28] (v6);
    \draw[redge] (v1) -- (v2);
    \draw[redge] (v2) -- (v3);
    \draw[edge] (v3) -- (v4);
    \draw[redge] (v4) -- (v5);
    \draw[edge] (v5) -- (v6);
    \draw[edge] (v6) -- (v7);
    \draw[redge] (v7) -- (v8);
    \draw[redge] (v8) -- (v9);

    \draw[br] (3,-0.15) -- (1,-0.15);
    \node[](t) at (2, -0.5){$a$};

    \draw[br] (9,-0.15) -- (7,-0.15);
    \node[](t) at (8, -0.5){$b$};

    \node[label=below:$v_i$] (t) at (3,0){};
    \node[label=below:$v_{i+1}$] (t) at (4,0){};
    \node[label=below:$v_{i+2}$] (t) at (5,0){};
    \node[label=below:$v_{i+3}$] (t) at (6,0){};
    \node[label=below:$v_{i+4}$] (t) at (7,0){};

\end{scope}
\end{tikzpicture}

\begin{tikzpicture}
\begin{scope}[scale = 0.7]
        \foreach \i in {1,..., 9}{
        \node[vertex](v\i) at (\i, 0){};
        
    }

    \node[vertex](u) at (5,0.5){};
    \draw[redge] (v3) to [bend right = 28] (v5);
    %\draw[redge] (v5) to [bend left = 40] (v7);
    \draw[redge] (v4) to [bend left = 20] (u);
    \draw[edge] (u) to [bend left = 20] (v6);
    \draw[redge] (v1) -- (v2);
    \draw[redge] (v2) -- (v3);
    \draw[edge] (v3) -- (v4);
    \draw[redge] (v4) -- (v5);
    \draw[redge] (v5) -- (v6);
    \draw[redge] (v6) -- (v7);
    \draw[redge] (v7) -- (v8);
    \draw[redge] (v8) -- (v9);

    \draw[br] (3,-0.15) -- (1,-0.15);
    \node[](t) at (2, -0.5){$a$};

    \draw[br] (9,-0.15) -- (6,-0.15);
    \node[](t) at (7.5, -0.5){$b$};

    \node[label=below:$v_i$] (t) at (3,0){};
    \node[label=below:$v_{i+1}$] (t) at (4,0){};
    \node[label=below:$v_{i+2}$] (t) at (5,0){};
    \node[label=below:$v_{i+3}$] (t) at (6,0){};

\end{scope}
\end{tikzpicture}

%% file: figures/s2ab-second.tex
\begin{tikzpicture}
\begin{scope}[scale = 0.7]
    \foreach \i in {1,..., 9}{
        \node[vertex](v\i) at (\i, 0){};
        
    }
    \draw[redge] (v1) to [bend left = 40] (v5);
    %\draw[redge] (v5) to [bend left = 40] (v7);
    %\draw[redge] (v4) to [bend right = 40] (v6);
    \draw[redge] (v1) -- (v2);
    \draw[redge] (v2) -- (v3);
    \draw[edge] (v3) -- (v4);
    \draw[redge] (v4) -- (v5);
    \draw[redge] (v5) -- (v6);
    \draw[redge] (v6) -- (v7);
    \draw[redge] (v7) -- (v8);
    \draw[redge] (v8) -- (v9);

    \draw[br] (3,-0.15) -- (1,-0.15);
    \node[](t) at (2, -0.6){$a$};

    \draw[br] (9,-0.15) -- (6,-0.15);
    \node[](t) at (7.5, -0.6){$b$};

    % DCK additions
    \foreach \i in {-2,..., 0}{
        \node[vertex](v\i) at (\i, 0){};
           
    }
    \draw[edge] (v-2) -- (v-1);
    \draw[edge] (v-1) -- (v0);
    \draw[edge] (v0) -- (v1);
    \draw[br] (0,-0.15) -- (-2,-0.15);
    \node[](t) at (-1, -0.6){$\le a-1$};

    % \node[label=below:$v_i$] (t) at (3,0){};
    % \node[label=below:$v_{i+1}$] (t) at (4,0){};
    % \node[label=below:$v_{i+2}$] (t) at (5,0){};
    % \node[label=below:$v_{i+3}$] (t) at (6,0){};
    % \node[label=below:$v_{i+4}$] (t) at (7,0){};

\end{scope}
\end{tikzpicture}

%% file: figures/s227.tex
\begin{tikzpicture}
\begin{scope}[scale = 0.6]
    \foreach \i in {1,..., 14}{
        \node[vertex, label = below:$v_{\i}$](v\i) at (\i, 0){};
        
    }
    
    \draw[edge] (v1) -- (v2);
    \draw[edge] (v2) -- (v3);
    \draw[edge] (v3) -- (v4);
    \draw[edge] (v4) -- (v5);
    \draw[redge] (v5) -- (v6);
    \draw[redge] (v6) -- (v7);
    \draw[redge] (v7) -- (v8);
    \draw[redge] (v8) -- (v9);
    \draw[redge] (v9) -- (v10);
    \draw[redge] (v10) -- (v11);
    \draw[redge] (v11) -- (v12);
    \draw[redge] (v12) -- (v13);
    \draw[redge] (v13) -- (v14);

\node[vertex](u1) at (7,1){};
\node[vertex](u2) at (8,1){};

    \draw[redge](v7) -- (u1);
    \draw[redge](u1) -- (u2);

\end{scope}

\end{tikzpicture}

%% file: figures/s235.tex
\begin{tikzpicture}
\begin{scope}[scale = 0.6]
    \foreach \i in {1,..., 11}{
        \node[vertex, label = below:$v_{\i}$](v\i) at (\i, 0){};
        
    }
    
    \draw[redge] (v1) -- (v2);
    \draw[redge] (v2) -- (v3);
    \draw[redge] (v3) -- (v4);
    \draw[redge] (v4) -- (v5);
    \draw[redge] (v5) -- (v6);
    \draw[redge] (v6) -- (v7);
    \draw[redge] (v7) -- (v8);
    \draw[redge] (v8) -- (v9);
    \draw[edge] (v9) -- (v10);
    \draw[edge] (v10) -- (v11);

\node[vertex](u1) at (5.5,1){};
\node[vertex](u2) at (6.5,1){};

    \draw[redge](v4) -- (u1);
    \draw[redge](u1) -- (u2);
    
\begin{scope}[shift = {(12,0)}]
    \foreach \i in {1,..., 11}{
        \node[vertex, label = below:$v_{\i}$](v\i) at (\i, 0){};
        
    }
    
    \draw[redge] (v1) -- (v2);
    \draw[redge] (v2) -- (v3);
    \draw[redge] (v3) -- (v4);
    \draw[redge] (v4) -- (v5);
    \draw[redge] (v5) -- (v6);
    \draw[edge] (v6) -- (v7);
    \draw[edge] (v7) -- (v8);
    \draw[edge] (v8) -- (v9);
    \draw[redge] (v9) -- (v10);
    \draw[redge] (v10) -- (v11);

\node[vertex](u1) at (5.5,1){};
\node[vertex](u2) at (6.5,1){};

    \draw[redge](v3) -- (u1);
    \draw[redge](u1) -- (u2);
    \draw[redge] (u2) -- (v9);
    
\end{scope}

\end{scope}
\end{tikzpicture}

%% file: figures/s244-longp.tex
\begin{tikzpicture}

\begin{scope}[scale = 0.6]
    \foreach \i in {1,..., 12}{
        \node[vertex, label = below:$v_{\i}$](v\i) at (\i, 0){};
        
    }
    
    \draw[redge] (v1) -- (v2);
    \draw[redge] (v2) -- (v3);
    \draw[redge] (v3) -- (v4);
    \draw[redge] (v4) -- (v5);
    \draw[redge] (v5) -- (v6);
    \draw[redge] (v6) -- (v7);
    \draw[redge] (v7) -- (v8);
    \draw[redge] (v8) -- (v9);
    \draw[edge] (v9) -- (v10);
    \draw[edge] (v10) -- (v11);
    \draw[edge] (v11) -- (v12);

\node[vertex](u1) at (6,1){};
\node[vertex](u2) at (7,1){};

    \draw[redge](v5) -- (u1);
    \draw[redge](u1) -- (u2);

\begin{scope}[shift = {(12,0)}]
    \foreach \i in {1,..., 12}{
        \node[vertex, label = below:$v_{\i}$](v\i) at (\i, 0){};
        
    }
    
    \draw[redge] (v1) -- (v2);
    \draw[redge] (v2) -- (v3);
    \draw[redge] (v3) -- (v4);
    \draw[redge] (v4) -- (v5);
    \draw[redge] (v5) -- (v6);
    \draw[redge] (v6) -- (v7);
    \draw[edge] (v7) -- (v8);
    \draw[edge] (v8) -- (v9);
    \draw[redge] (v9) -- (v10);
    \draw[edge] (v10) -- (v11);
    \draw[edge] (v11) -- (v12);

\node[vertex](u1) at (6,1){};
\node[vertex](u2) at (7,1){};

    \draw[redge](v3) -- (u1);
    \draw[redge](u1) -- (u2);
    \draw[redge](u2) -- (v9);
\end{scope}

\end{scope}

\end{tikzpicture}

\begin{tikzpicture}
\begin{scope}[scale=0.6]

    \foreach \i in {1,..., 12}{
        \node[vertex, label = below:$v_{\i}$](v\i) at (\i, 0){};
        
    }
    
    \draw[edge] (v1) -- (v2);
    \draw[redge] (v2) -- (v3);
    \draw[redge] (v3) -- (v4);
    \draw[redge] (v4) -- (v5);
    \draw[redge] (v5) -- (v6);
    \draw[redge] (v6) -- (v7);
    \draw[redge] (v7) -- (v8);
    \draw[edge] (v8) -- (v9);
    \draw[redge] (v9) -- (v10);
    \draw[edge] (v10) -- (v11);
    \draw[edge] (v11) -- (v12);

\node[vertex](u1) at (6,1){};
\node[vertex](u2) at (7,1){};

    \draw[redge](v4) -- (u1);
    \draw[redge](u1) -- (u2);
    \draw[redge](u2) -- (v9);

\end{scope}
\end{tikzpicture}

%% file: figures/s244-concomp.tex
\begin{tikzpicture}

\begin{scope}[scale = 0.6]
    \foreach \i in {1,..., 11}{
        \node[vertex, label = below:$v_{\i}$](v\i) at (\i, 0){};
        
    }
    
    \draw[redge] (v1) -- (v2);
    \draw[redge] (v2) -- (v3);
    \draw[redge] (v3) -- (v4);
    \draw[redge] (v4) -- (v5);
    \draw[redge] (v5) -- (v6);
    \draw[redge] (v6) -- (v7);
    \draw[redge] (v7) -- (v8);
    \draw[redge] (v8) -- (v9);
    \draw[edge] (v9) -- (v10);
    \draw[edge] (v10) -- (v11);

\node[vertex](u1) at (5.5,1){};
\node[vertex](u2) at (6.5,1){};

    \draw[redge](v5) -- (u1);
    \draw[redge](u1) -- (u2);
   % \draw[redge](u2) -- (v9);
    
\end{scope}

\begin{scope}[scale = 0.6, shift = {(12,0)}]
    \foreach \i in {1,..., 11}{
        \node[vertex, label = below:$v_{\i}$](v\i) at (\i, 0){};
        
    }
    
    \draw[redge] (v1) -- (v2);
    \draw[redge] (v2) -- (v3);
    \draw[redge] (v3) -- (v4);
    \draw[redge] (v4) -- (v5);
    \draw[redge] (v5) -- (v6);
    \draw[redge] (v6) -- (v7);
    \draw[edge] (v7) -- (v8);
    \draw[edge] (v8) -- (v9);
    \draw[redge] (v9) -- (v10);
    \draw[edge] (v10) -- (v11);

\node[vertex](u1) at (5.5,1){};
\node[vertex](u2) at (6.5,1){};

    \draw[redge](v3) -- (u1);
    \draw[redge](u1) -- (u2);
    \draw[redge](u2) -- (v9);
    
\end{scope}
\end{tikzpicture}

\begin{tikzpicture}
\begin{scope}[scale = 0.6]
    \foreach \i in {1,..., 11}{
        \node[vertex, label = below:$v_{\i}$](v\i) at (\i, 0){};
        
    }
    
    \draw[edge] (v1) -- (v2);
    \draw[redge] (v2) -- (v3);
    \draw[redge] (v3) -- (v4);
    \draw[redge] (v4) -- (v5);
    \draw[redge] (v5) -- (v6);
    \draw[redge] (v6) -- (v7);
    \draw[redge] (v7) -- (v8);
    \draw[edge] (v8) -- (v9);
    \draw[redge] (v9) -- (v10);
    \draw[edge] (v10) -- (v11);

\node[vertex](u1) at (5.5,1){};
\node[vertex](u2) at (6.5,1){};

    \draw[redge](v4) -- (u1);
    \draw[redge](u1) -- (u2);
    \draw[redge](u2) -- (v9);

\end{scope}
\end{tikzpicture}

%% file: figures/s244-neighboursstar.tex
\begin{tikzpicture}
\begin{scope}[scale = 0.6]
    \foreach \i in {1,..., 11}{
        \node[vertex, label = below:$v_{\i}$](v\i) at (\i, 0){};
        
    }
    
    \draw[redge] (v1) -- (v2);
    \draw[redge] (v2) -- (v3);
    \draw[redge] (v3) -- (v4);
    \draw[edge] (v4) -- (v5);
    \draw[redge] (v5) -- (v6);
    \draw[redge] (v6) -- (v7);
    \draw[redge] (v7) -- (v8);
    \draw[edge] (v8) -- (v9);
    \draw[edge] (v9) -- (v10);
    \draw[edge] (v10) -- (v11);

\node[vertex](u1) at (5.5,1){};
\node[vertex](u2) at (6.5,1){};
\node[vertex](u3) at (6,1.66){};

    \draw[redge](v4) -- (u1);
    \draw[redge](v8) -- (u1);
    \draw[redge](u1) -- (u3);
    \draw[redge](u2) -- (u3);

\end{scope}
\end{tikzpicture}